\documentclass{article}
\usepackage[%
journal=QT,    
lang=american,   
]{ems-journal}

\usepackage{standalone} 

\usepackage{hyperref}
\hypersetup{
    colorlinks=true,
    linkcolor=purple,
    citecolor=teal,      
    urlcolor=purple,
}

\usepackage{physics}
\usepackage{mathdots}
\usepackage{cancel}
\usepackage{color}
\usepackage{siunitx}
\usepackage{array}
\usepackage{multirow}
\usepackage{gensymb}
\usepackage{tabularx}
\usepackage{extarrows}
\usepackage{booktabs}

\usepackage{tikz}
    \usetikzlibrary{cd} 
\usepackage{tikz-cd}
\usetikzlibrary{fadings}
\usetikzlibrary{patterns}
\usetikzlibrary{shadows.blur}
\usetikzlibrary{shapes}

\allowdisplaybreaks
\usepackage{stmaryrd} 
\usepackage{MnSymbol} 
\usepackage{textcomp} 
\usepackage{enumitem}
\usepackage{bbm}

\theoremstyle{plain}
\newtheorem{theorem}{{Theorem}}[section]
\newtheorem{lemma}[theorem]{{Lemma}}
\newtheorem{proposition}[theorem]{{Proposition}}

\newtheorem{corollary}[theorem]{{Corollary}}

\newtheorem{remark}[theorem]{Remark}
\newtheorem{definition}[theorem]{{Definition}}

\newcommand{\R}{\mathbb{R}}
\newcommand{\Z}{\mathbb{Z}}
\newcommand{\N}{\mathbb{N}} 
\newcommand{\BF}{\mathbb{F}} 
\newcommand{\gl}{\mathfrak{gl}} 

\newcommand{\map}[1]{\xrightarrow{#1}}
\newcommand{\inv}{^{-1}}
\newcommand{\st}{\ \mid \ }

\newcommand{\idmap}{\mathrm{id}}  
\newcommand{\dual}{^\vee} 

\newcommand{\CC}{\mathcal{C}} 
\newcommand{\Kh}{\mathrm {Kh}}
\newcommand{\CKh}{\mathrm{CKh}} 

\newcommand{\BN}{\mathrm{BN}}
\newcommand{\KhR}{\mathrm{KhR}} 
\newcommand{\CKhR}{\mathrm{CKhR}} 
\newcommand{\unframedKhR}{\mathbf{KhR}} 
\newcommand{\newKh}{\overline{\Kh}} 
\newcommand{\newCKh}{\overline{\CKh}} 

\newcommand{\gVect}{\mathrm{gVect}} 
\newcommand{\ggVect}{\mathrm{ggVect}} 
\newcommand{\Lee}{\mathrm{Lee}} 
\newcommand{\Cob}{\mathbf{Cob}} 
\newcommand{\skein}{\mathcal{S}}
\newcommand{\BNcat}{\mathcal{BN}} 
\newcommand{\TLcat}{\mathcal{TL}} 
\newcommand{\BNbrack}[1]{\llbracket #1 \rrbracket_{\BN}} 
\newcommand{\KhRbrack}[1]{\llbracket #1 \rrbracket_{\KhR}} 

\newcommand{\FT}{\mathrm{FT}} 
\newcommand{\fraks}{\mathfrak{s}} 

\newcommand{\F}{\mathcal{F}} 
\newcommand{\gr}{\mathrm{gr}} 
\newcommand{\colim}{\mathrm{colim}\,}
\newcommand{\cone}{\mathrm{Cone}}
\newcommand{\Cone}{\cone} 
\newcommand{\Ext}{\mathrm{Ext}} 
\newcommand{\Ob}{\mathrm{Ob}} 
\newcommand{\Mor}{\mathrm{Mor}} 
\newcommand{\Mat}{\mathrm{Mat}} 
\newcommand{\Kom}{\mathrm{Kom}} 
\newcommand{\Kar}{\mathrm{Kar}} 
\newcommand{\Ind}{\mathrm{Ind}} 
\newcommand{\Tot}{\mathrm{Tot}} 
\newcommand{\symgroup}[1]
{\mathfrak{S}_{#1}} 
\newcommand{\symmetrized}[1]{\mathrm{Sym}(#1)} 

\newcommand{\hocolim}{\mathrm{hocolim}\,}

\newcommand{\acup}{
    \begin{tikzpicture}[scale=.125, baseline=-1ex]
    \draw[thick] (1,0) arc (0:-180:1);
    \foreach \x in {-1,1}{
        \draw[thick] (\x,0) -- (\x,1);
    }
    \end{tikzpicture}    
}

\newcommand{\acupdot}{
    \begin{tikzpicture}[scale=.125, baseline=-1ex]
    \draw[thick] (1,0) arc (0:-180:1);
    \foreach \x in {-1,1}{
        \draw[thick] (\x,0) -- (\x,1);
    }
    \node at (0,-1) {$\bullet$};
    \end{tikzpicture}    
}

\DeclareRobustCommand\robacupdot{\acupdot}

\newcommand{\acap}{
    \begin{tikzpicture}[yscale=-1, scale=.125, baseline=-1ex]
    \draw[thick] (1,0) arc (0:-180:1);
    \foreach \x in {-1,1}{
        \draw[thick] (\x,0) -- (\x,1);
    }
    \end{tikzpicture}    
}

\newcommand{\capcup}{
    \tikzset{every picture/.style={line width=0.75pt}} 
    \begin{tikzpicture}[x=0.75pt,y=0.75pt,yscale=-1,xscale=1]
    \draw    (325.75,19.58) .. controls (325.85,13.87) and (335.1,13.87) .. (334.89,19.4) ;
    \draw    (325.85,8.43) .. controls (325.75,14.78) and (335.2,14.69) .. (335,8.25) ;
    \end{tikzpicture}
}

\newcommand{\acapdot}{
    \begin{tikzpicture}[yscale=-1, scale=.125, baseline=-1ex]
    \draw[thick] (1,0) arc (0:-180:1);
    \foreach \x in {-1,1}{
        \draw[thick] (\x,0) -- (\x,1);
    }
    \node at (0,-1) {$\bullet$};
    \end{tikzpicture}    
}

\newcommand{\kap}{
    \begin{scope}[xscale=.5]
    \draw (0,1) arc (90:270:1cm);
    \draw[dotted] (0,1) arc (90:-90:1cm);
    \end{scope}
    \begin{scope}[xscale=1.5]
        \draw (0,1) arc (90:-90:1cm);
    \end{scope}
}
\newcommand{\kapdot}{
    \kap
    \node (dot) at (1,0) {$\bullet$};
}
\newcommand{\kup}{
    \begin{scope}[xscale=-.5]
    \draw (0,1) arc (90:270:1cm);
    \draw (0,1) arc (90:-90:1cm);
    \end{scope}
    \begin{scope}[xscale=-1.5]
        \draw (0,1) arc (90:-90:1cm);
    \end{scope}
}
\newcommand{\kupdot}{
    \kup
    \node (dot) at (-1,0) {$\bullet$};
}

\newcommand{\Id}{\mathrm{Id}}
\newcommand{\one}{\mathbf{1}}
\newcommand{\End}{\mathrm{End}}
\newcommand{\Hom}{\mathrm{Hom}}
\newcommand{\cl}[1]{\mathrm{Tr} \left ( #1 \right)} 
\newcommand{\close}{\mathrm{Tr}} 

\numberwithin{equation}{section}

\begin{document}

\title{Kirby belts, categorified projectors, and the skein lasagna module of $S^2 \times S^2$}



\emsauthor{1}{
	\givenname{Ian}
	\surname{Sullivan}
	\mrid{}
	\orcid{0000-0002-2731-251X}}{I.~A.~Sullivan}
\emsauthor{2}{
	\givenname{Melissa}
	\surname{Zhang}
	\mrid{1261780}
	\orcid{0000-0002-2750-2794}}{M.~Zhang}

\Emsaffil{1}{
	\department{Department of Mathematics}
	\organisation{University of California, Davis}
	\rorid{05rrcem69}
	\address{One Shields Ave.}
	\zip{95616}
	\city{Davis, CA}
	\country{U.S.A.}
	\affemail{iasullivan@ucadvis.edu}}
\Emsaffil{2}{
	\department{Department of Mathematics}
	\organisation{University of California, Davis}
	\rorid{05rrcem69}
	\address{One Shields Ave.}
	\zip{95616}
	\city{Davis, CA}
	\country{U.S.A.}
	\affemail{mlzhang@ucadvis.edu}}

\classification[57K41]{57K18}


\begin{abstract}
We interpret Manolescu-Neithalath's cabled Khovanov homology formula for computing Morrison-Walker-Wedrich's $\KhR_2$ skein lasagna module as a homotopy colimit (mapping telescope) in a completion of the category of complexes over Bar-Natan's cobordism category.
Using categorified projectors, we compute the $\KhR_2$ skein lasagna modules of (manifold, boundary link) pairs $(S^2 \times B^2, \tilde \beta)$, where $\tilde \beta$ is a geometrically essential boundary link, identifying a relationship between the lasagna module and the Rozansky projector appearing in the Rozansky-Willis invariant for nullhomologous links in $S^2 \times S^1$. 
As an application, we show that the $\KhR_2$ skein lasagna module of $S^2 \times S^2$ is trivial, confirming a conjecture of Manolescu.
\end{abstract}

\maketitle


\section{Introduction}

In recent years, link homology theories have proven to be useful tools in the study of smooth 4-manifolds. For example, Khovanov homology classes can distinguish pairs of exotic slice disks for certain knots \cite{hay-sun}, and Rasmussen's $s$-invariant \cite{RasInv} gives bounds on a knot's $4$-ball genus. 

Morrison-Walker-Wedrich's \emph{skein lasagna modules} are a promising new tool; in \cite{MWW-lasagna}, the authors describe a method that extends link homology theories for links in $S^{3}$ to diffeomorphism invariants of a pair $(W,L\subset{\partial{W}})$, where $W$ is a $4$-manifold with a link $L$ in its boundary. Methods for computing skein lasagna modules were developed by Manolescu-Neithalath \cite{MN22} (for 2-handlebodies) and extended by Manolescu-Walker-Wedrich \cite{MWWhandles}. 
The hope is that, by improving the tools for computing lasagna modules, these 4-manifold invariants based in quantum topology will be able to distinguish smooth structures in cases where gauge-theoretic invariants (e.g.\ Seiberg-Witten invariants) may be incomputable or inconclusive.

In this paper, we develop and employ novel computational techniques and compute $\skein_{0}^{2}$ for some new $(W,L)$ pairs. Note that $\skein_{0}^{2}(W;L)$ is an $H_{2}(W,L;\mathbb{Z})\times{\mathbb{Z}}\times{\mathbb{Z}}$-graded invariant, where the $H_{2}(W,L;\mathbb{Z})$ grading is called the \emph{homological level} (see Section \ref{subsec:skeinlasagnamodules} for more details). 
Throughout, we work over a field
$\mathbb{F}$ of characteristic 0, and we denote the skein lasagna module of a pair $(W,L)$ at homological level $\alpha$ by $\skein_{0}^{2}(W;L,\alpha)$.
Let $A_0$ and $A_1$ be formal variables with $q$-degree $0$ and $-2$ respectively, and
let $\mathbb{F}_{|\alpha|}[A_{0},A_{0}^{-1},A_{1}]$ denote the subgroup of $\mathbb{F}[A_{0},A_{0}^{-1},A_{1}]$ consisting of degree $\alpha$ homogeneous polynomials in $A_0$, $ A_0\inv,$ and $A_1$. 
Manolescu and Neithalath in \cite{MN22} show that $\skein_{0}^{2}(S^{2}\times{B^{2}};\emptyset,\alpha)$ is isomorphic to $\mathbb{F}_{|\alpha|}[A_{0},A_{0}^{-1},A_{1}]$. 

Let $P_{n,k}$ denote the \emph{higher order projectors} introduced in \cite{Cooper-Hog-family}, where $P_{n}=P_{n,n}$ is the $n$th \emph{categorified Jones-Wenzl projector}, and $P_{n,0}$ is the $n$th \emph{Rozansky projector} (see Section \ref{sec:projectors} for more details). For any projector $P_{n}$, let $P_{n}^{\vee}$ denote the corresponding dual projector.
Letting $\widetilde{\one}_{n}$ denote the \emph{geometrically essential} $n$-component link in $\partial{S^{2}\times{B^{2}}}=S^{1}\times{S^{2}}$ (see Figure \ref{fig:geom-essen}), we prove the following.

\begin{theorem}\label{thrm:S2xB2thrm}
    Let $\alpha\in{H_{2}(S^{2}\times{B^{2}};\Z)}\cong{\mathbb{Z}}$. Then
    \begin{equation}
            \skein^{2}_{0}(S^{2}\times{B^{2};\widetilde{\one}_{n},\alpha)\cong}
        \begin{cases}
            0 & \text{if $\alpha$ or $n$ is odd.}\\
            \mathbb{F}_{|\alpha|}[A_{0},A_{0}^{-1},A_{1}]\otimes{\KhR_{2}(\close(P^{\vee}_{n,0}))} & \text{if $\alpha$ and $n$ are even.}
        \end{cases}
    \end{equation}
    Here, $\close(P^{\vee}_{n,0})$ denotes the trace of the dual Rozansky projector on $n$ strands, when $n$ is even.
\end{theorem}

An expected property for any diffeomorphism invariant of $4$-manifolds with a connect sum formula is some notion of triviality on $S^{2}\times{S^{2}}$. As, given an exotic pair $(X_{1},X_{2})$, the manifolds $X_{1}\#{r}(S^{2}\times{S^{2}})$ and $X_{2}\#{r}(S^{2}\times{S^{2}})$ are diffeomorphic for a sufficiently large $r>0$ by Wall's stabilization theorem \cite{CTCWall}. Over a field of characteristic 0, the skein lasagna module is such an invariant. Our Theorem \ref{thrm:S2xB2thrm} admits the following corollary.

\begin{corollary}\label{cor:main}
We have that
    \begin{equation}
        \skein^{2}_{0}(S^{2}\times{S^{2}};\emptyset,(\alpha_{1},\alpha_{2}))\cong{0}
    \end{equation}
for every homological level $(\alpha_{1},\alpha_{2})\in{H_{2}(S^{2}\times{S^{2}}})$.    
\end{corollary}

This confirms a conjecture of Ciprian Manolescu \cite{manolescu-ictp-video}. 
This result was also simultaneously and independently proven by Ren-Willis in \cite{ren2024khovanov} using different techniques from ours.

We prove Theorem \ref{thrm:S2xB2thrm} by proving algebraic properties about homotopy colimits of directed systems of chain complexes associated to cablings of specific tangle diagrams. These homotopy colimits, called \emph{Kirby belts} or \emph{Kirby-belted tangles}, model the skein lasagna module of $(S^{2}\times{B^{2}}; \widetilde{L})$. 

Let $\KhR_{2}(L)$ denote the $\mathfrak{gl}_{2}$ Khovanov-Rozansky homology group of a framed, oriented link $L$. By the Manolescu-Neithalath $2$-handlebody formula \cite[Theorem 1.1]{MN22}, the skein lasagna module $\skein_{0}^{2}(S^{2}\times{S^{2}};\emptyset,(\alpha_{1},\alpha_{2}))$ is isomorphic to the \textit{cabled} Khovanov homology $\underline{\KhR}_{2,(\alpha_{1},\alpha_{2})}(L)$, where $L$ is the oriented Hopf link with $0$-framing on both components. After fixing a homological level $(\alpha_{1},\alpha_{2})$, computing $\skein_{0}^{2}(S^{2}\times{S^{2}};\emptyset,(\alpha_{1},\alpha_{2}))$ amounts to computing the colimit of a \emph{cabling directed system} (see Figure \ref{fig:mountain}). Roughly, a cabling directed system for $\KhR_{2}$ is a directed system of symmetrized $\KhR_{2}$ vector spaces associated to a directed system given by a cabling pattern and (dotted) annulus cobordisms between cables (see Definition \ref{def:CDS}). The Hopf link has two components, so the corresponding cabling directed system is $2$-dimensional; we consider the two directions in the cabling directed system individually and compute their colimits. 

\begin{figure}[t]
    \begin{center}
        \includegraphics[width=\linewidth]{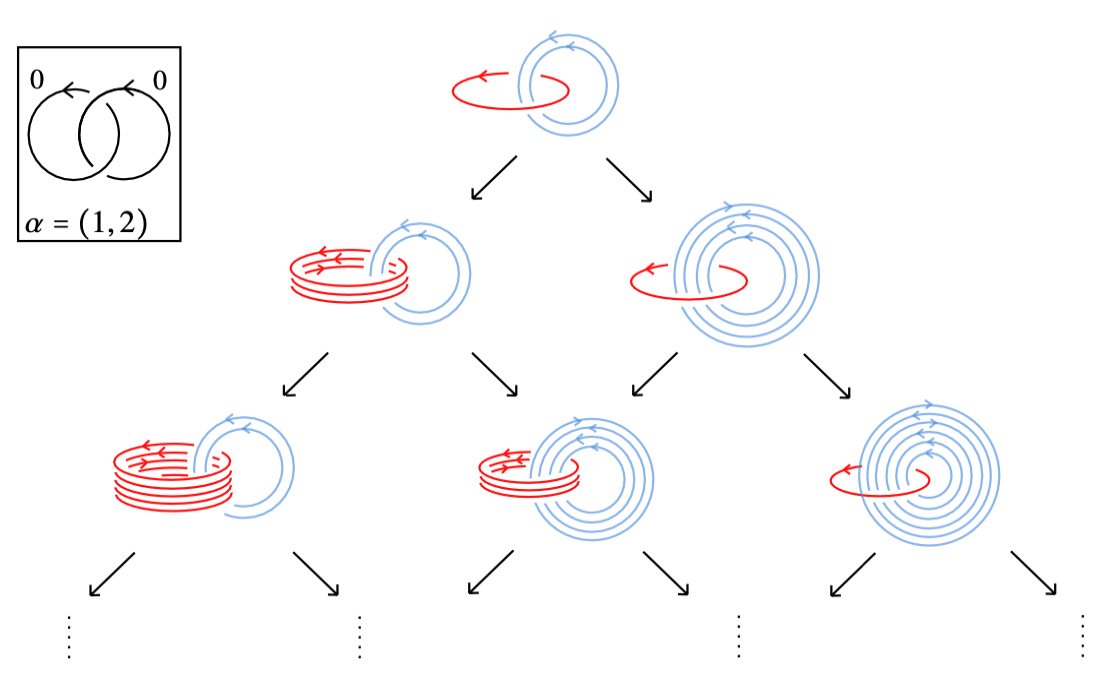}
    \end{center}
    \caption{An illustration of the cabling directed system 
    $\mathcal{B}(S^{2}\times{S^{2}};\emptyset,(\alpha_{1},\alpha_{2}))$ for the cabled Khovanov homology of the $(0,0)$-framed Hopf link in homological level $(1,2)\in{H_{2}(S^{2}\times{S^{2}};\mathbb{Z})}$. Each link diagram represents the symmetrized $\KhR_{2}$ homology group of said cabled link. The arrows correspond to dotted annulus cobordism maps, increasing the number of link components according to the direction of the arrow.}
    \label{fig:mountain}
\end{figure}

We construct a homotopy colimit whose homology is isomorphic to the skein lasagna module of the pair $(S^{2}\times{B^{2}};\widetilde{\one}_{n})$. We begin our construction by considering tangle diagrams and corresponding complexes of the form shown below.
\begin{center}
    \includegraphics[width=1.5in]{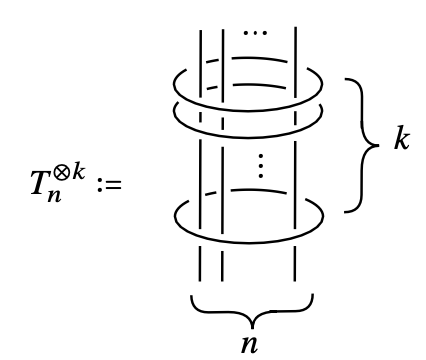}
\end{center}

Let $\symmetrized{T_{n}^{\otimes{k}}}$ denote the symmetrized complex associated to $T_{n}^{\otimes{k}}$ under Grigsby-Licata-Wehrli's braid group action \cite{GLW-schur-weyl}. Then the \emph{Kirby-belted identity braid} on $n$ strands is defined as the following colimit, denoted $T_{n}^{\omega_{i}}$:

\begin{equation}
\raisebox{-0.67cm}{\definecolor{deeppink}{rgb}{1.0, 0.08, 0.58}
\begin{tikzpicture}[yscale=.5, scale=.25]
%
\draw [draw = deeppink] (0,0) circle (3cm);
%
\def\g{.3};
\foreach \x in {-2,-0.7,2}{
    \fill[white](\x-\g,0) rectangle (\x+\g,6); 
    \draw (\x,0) -- (\x,6);
}
\foreach \x in {-2,2}{
    \draw (\x,0) -- (\x,-1.5);
    \draw (\x,-3) -- (\x,-6);
}
\draw (-0.7,0) -- (-0.7,-2);
\draw (-0.7,-4) -- (-0.7,-6);
\foreach \y in {5,0,-5}{
    \node at (0.6,\y) {$\cdots$};
}
\node at (3.1,3) {$\omega_{i}$};

\raisebox{-.67cm}

\end{tikzpicture} } := \colim{\left(\symmetrized{T_{n}^{\otimes{i}}}\xrightarrow{\mathrm{Sym}\left({\acupdot}\right)}{\symmetrized{T_{n}^{\otimes{(i+2)}}}}\xrightarrow{\mathrm{Sym}\left(\acupdot\right)}{\symmetrized{T_{n}^{\otimes{(i+4)}}}}\rightarrow{\cdots}\right)}
\label{eq:omegakirbybelt}
\end{equation}

The \emph{Kirby belt} is the colored unknotted component, decorated with a label $\omega_{i}$. This $\omega_{i}$ labelling was chosen to reference the Kirby objects and Kirby-colored Khovanov homology of \cite{hogancamp2022kirby}; our constructions are similar to theirs, but we do not work in the annular setting (see Remark \ref{rmk:annularkirby}). To make our computations explicit, we choose instead to study the homotopy colimit of the directed system in \eqref{eq:omegakirbybelt}. This homotopy colimit, denoted $T_{n}^{\Omega_{i}}$, is the total complex of the diagram shown in \eqref{eq:1}. The $[i]$ denotes a shift in homological degree, and $f$ denotes the dotted annulus map as in Equation \eqref{eq:omegakirbybelt}. 

\begin{figure}[t]
\begin{center}
\begin{equation}
\label{eq:1}
\begin{tikzcd}
    \symmetrized{T_{n}^{\otimes{i}}}[0]
    & \symmetrized{T_{n}^{\otimes{(i+2)}}}[0] 
        & \symmetrized{T_{n}^{\otimes{(i+4)}}}[0] 
        & \text{\phantom{A4cell}}  \\
    \symmetrized{T_{n}^{\otimes{i}}}[-1] \arrow{u}{\idmap} \arrow{ur}{-\symmetrized{f}}
        & \symmetrized{T_{n}^{\otimes{(i+2)}}}[-1] \arrow{u}{\idmap} \arrow{ur}{-\symmetrized{f}}
        & \symmetrized{T_{n}^{\otimes{(i+4)}}}[-1] \arrow{u}{\idmap} \arrow{ur}{-\symmetrized{f}}
        & \\
\end{tikzcd}
    \cdots
\end{equation}
\end{center}
\end{figure}

In the Ind-completion of the homotopy category of $\Kom(\TLcat_{n})$, these two notions of colimits agree, so using the complex $T_{n}^{\Omega_{i}}$ in place of $T_{n}^{\omega_{i}}$ is legitimate.
Observe that the trace of $T_{n}^{\otimes{k}}$ is the Khovanov-Rozansky complex of a cable of the Hopf link; this allows us to compute the homology of the trace of 0-framed Kirby-belted diagrams.
Using properties of the Kirby-belted identity braid, we are able to prove Theorem \ref{thrm:S2xB2thrm}.

Observe that the vector spaces $\skein_{0}^{2}(S^{2}\times{B^{2}};\widetilde{\one}_{n},\alpha)$ are the colimits of `diagonals' of the cabling directed system of $S^{2}\times{S^{2}}$. Hence, the skein lasagna module of $S^{2}\times{S^{2}}$ can be realized as a colimit of a directed system involving skein modules of the form $\skein_{0}^{2}(S^{2}\times{B^{2}};\widetilde{\one}_{n})$. This observation, along with Theorem \ref{thrm:S2xB2thrm}, yields the desired vanishing result in Corollary \ref{cor:main}.

The main technical foundation for Theorem \ref{thrm:S2xB2thrm} is the following collection of properties about a Kirby-belted categorified Jones-Wenzl projector $P^{\vee}_{n}$.

\begin{proposition}\label{prop:main}
    Let $P^{\vee}_{n}$ denote the dual $n$th categorified Jones-Wenzl projector, and let $n>0$. Then 
    \[ P^{\vee}_{n}\otimes{T_{n}^{\omega_{\alpha}}}:=
    \colim{\left(\symmetrized{P^{\vee}_{n} \otimes{T_{n}^{\otimes{\alpha}}}}\xrightarrow{\idmap_{P^{\vee}_{n}}\otimes{\mathrm{Sym}{\left(\rule{0mm} {0.85mm}\acupdot\right)}}} \symmetrized{P^{\vee}_{n}\otimes{T_{n}^{\otimes{(\alpha+2)}}}}\xrightarrow{\idmap_{P^{\vee}_{n}}\otimes{\mathrm{Sym}\left(\rule{0mm}{0.85mm}\acupdot\right)}}\ldots\right)}
    \]
    is 0 for $\alpha\in{\{0,1\}}$. Similarly, $P^{\vee}_{n}\otimes{T_{n}^{\Omega_{\alpha}}}\simeq{0}$ for $\alpha\in{\{0,1\}}$.
\end{proposition}

\begin{figure}[t]
    \centering
    \includegraphics[width=1.2in]{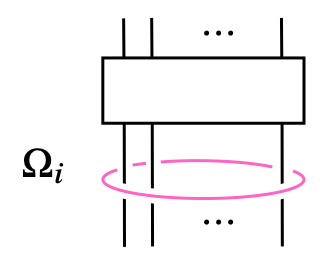}
    \caption{The complex $P^{\vee}_{n}\otimes{T^{\Omega_{i}}_{n}}$ above is contractible when $n>0$. We label the belt component with $\Omega_{i}$ when referring to the homotopy colimit.}
    \label{fig:PnTnOmega}
\end{figure}

In other words, $P^{\vee}_{n}$ annihilates $T_{n}^{\omega_{\alpha}}$ (see Figure \ref{fig:PnTnOmega}). Proposition \ref{prop:main} is proven by determining the complexes $\symmetrized{P^{\vee}_{n}\otimes{T_{n}^{\otimes{k}}}}$ and symmetrized maps $\symmetrized{\acupdot}$ explicitly, then arguing for contractibility using techniques from homological algebra. We then use the \emph{resolution of the identity} of \cite[Theorem 7.4]{Cooper-Hog-family} to relate $T^{\omega_{\alpha}}_{n}$ to a complex with chain groups of the form $P^{\vee}_{n,k}\otimes{T_{n}^{\omega_{\alpha}}}$, where $P^{\vee}_{n,k}$ is the dual $k$th \emph{higher order projector} on $n$ strands. We show that these $P^{\vee}_{n,k}\otimes{T_{n}^{\omega_{\alpha}}}$ terms are trivial for certain values of $n$ and $k$, and the next result follows.

\begin{theorem}
    If $n$ is an odd positive integer, then $T_{n}^{\Omega_{\alpha}}\simeq{0}$. If $n$ is an even positive integer, then $T_{n}^{\Omega_{\alpha}}$ is chain homotopy equivalent to the complex associated to ${U^{\Omega_{\alpha}}}\sqcup{P^{\vee}_{n,0}}$, where $U^{\Omega_{\alpha}}$ denotes a Kirby belt that is not linked with any strands. 
\end{theorem}
 
To prove the results involving Rozansky projectors (see \cite{WillisS1xS2} for more details) in the above theorems, we require a `sliding-off' property for the Kirby-colored belt wrapped around a complex associated to a through-degree $0$ tangle diagram. 
Let $A$ be a complex consisting of through-degree $0$ flat tangles. We show that $A\otimes{T_{n}^{\Omega_{\alpha}}}\simeq{A\sqcup{T_{n}^{\Omega_{\alpha}}}}$ as shown in Figure \ref{fig:kirbyslideoff}.

\begin{figure}[t]
    \includegraphics[width=.9\linewidth]{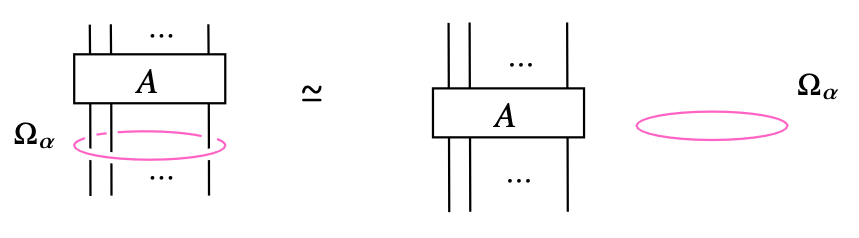}
    \caption{A Kirby belt `slides through' a through-degree-0 complex $A$.}
    \label{fig:kirbyslideoff}
\end{figure}

This result allows us to characterize the skein lasagna modules given in Theorem \ref{thrm:S2xB2thrm} and Theorem \ref{cor:main} for even integer-valued homological levels. There is a natural extension to other tangles wearing a Kirby-colored belt, denoted $T\otimes{T_{n}^{\Omega_{\alpha}}}$, for any $n$-strand tangle $T$.

\begin{corollary}
    If $n$ is an odd positive integer, then $T\otimes{T_{n}^{\Omega_{\alpha}}}\simeq{0}$. 
\end{corollary}

\begin{remark}
Due to convention differences, our dual projectors $P^\vee$ algebraically behave the same as projectors $P$ in \cite{WillisS1xS2} (up to a formal variable change $q \leftrightarrow q\inv$). 
In particular, our methods show that, for a nullhomologous link in $S^2\times S^1$, viewed as the geometrically essential closure $\widetilde{T}$ of an $(2k,2k)$ tangle $T$ (see Figure \ref{fig:geom-essen}), the skein lasagna module $\skein_0^2(S^2 \times D^2; \widetilde{T})$ is isomorphic to the tensor product of $\KhR(\cl{P_{2k,0}^\vee \otimes T})$ and the skein lasagna module of $S^2 \times D^2$, which was computed in \cite{MN22}. 
This result relates the skein lasagna module of a nullhomologous link in the boundary of $S^2 \times D^2$ with its Rozansky--Willis invariant, since 
$\KhR(\cl{P_{2k,0}^\vee \otimes T})$ is effectively the Rozansky--Willis invariant.
This result can be extended to nullhomologous links in $\#^n S^2 \times S^1$.
\end{remark}

\subsection{Organization}

Sections \ref{sec:algprelim} and \ref{sec:BN-cats} contain conventions and the necessary definitions and background theorems about homotopy colimits and Khovanov-Rozansky homology. Section \ref{sec:lasagnabackground} contains background on skein lasagna modules and cabling. Section \ref{sec:mainsec} presents the main results.

\section{Algebraic preliminaries}\label{sec:algprelim}

In this section we compile and prove algebraic results that we reference in the remainder of this work. 
Throughout, let $\mathbb{F}$ be a field of characteristic $0$. Let $\gVect$ and $\ggVect$ denote the categories of singly graded and bigraded vector spaces over $\BF$, respectively.

\subsection{Category theory preliminaries}
{
We first briefly discuss some useful constructions from the category theory of linear graded categories, including the Karoubi envelopes and Ind-completions, both of which are treated in Section 2 of \cite{hogancamp2022kirby}.
For more on Karoubi envelopes and homotopy idempotents, a useful reference is Appendix A of \cite{Gorsky-Wedrich}; see also the footnote under Definition 3.1 in \cite{BN-Morrison} for a brief history.
For more on Ind-objects and Ind-completions, see \cite{Kashiwara-Schapira-categories-sheaves}, Definition 6.1.
Finally, for an introduction to dg categories and twisted complexes, see \cite{BKeller-Ainfinity}.
}

\begin{definition}
    Let $G$ be an abelian group, and let $\CC$ be a $G$-graded $\mathbb{F}$-linear category. Then the \emph{$G$-additive completion} of $\CC$, denoted $\Mat(\CC)$, is the category whose objects are finite formal direct sums of objects in $\CC$, and each morphism $f:\bigoplus_{i=1}^{n}A_{i}\rightarrow{\bigoplus_{j=1}^{m}B_{i}}$ is given by an $m\times{n}$ matrix of morphisms $f_{ij}:A_{i}\rightarrow{B_{j}}$ in $\CC$.
\end{definition}

Unless specified otherwise, let $\mathcal{C}$ be a $\mathbb{Z}\oplus{\mathbb{Z}}$-graded $\mathbb{F}$-linear category. The additive completion of such a category formally adjoins grading shifts and finite sums, but we also need a completion that formally adjoins images of idempotent maps.

\begin{definition}
    The \emph{(graded) Karoubi envelope} of $\CC$, denote $\Kar(\CC)$, is the category whose objects are pairs $(A,e_{A})$, where $A$ is an object of $\CC$ and $e_{A}\in{\End_{\CC}^{0}}(A)$ is idempotent, i.e. $e_{A}^{2}=e_{A}$. The morphisms are given by maps $f\in{\Hom_{\CC}(A,B)}$ such that $f=e_{B}\circ{f}\circ{e_{A}}$.
\end{definition}

The colimits we study will be of the following form.

\begin{definition}
    A \emph{directed system} in $\CC$ is a diagram in $\CC$ indexed by a filtered small category. Furthermore, a \emph{filtered colimit} of $\CC$ is a colimit of a directed system in $\CC$.
\end{definition}

The following completion formally adjoins filtered colimits.

\begin{definition}[\cite{hogancamp2022kirby}, Def.\ 2.22]
    The \emph{Ind-completion} of $\CC$ (denoted $\Ind{(\CC)}$) is the category whose objects are directed systems $\alpha:\mathcal{I}\rightarrow{\CC}$ (where $\mathcal{I}$ is a directed indexing set). Given objects $\alpha:\mathcal{I}\rightarrow{\CC}$ and $\beta:\mathcal{J}\rightarrow{\CC}$, the morphism set is given by
    \[
    \Hom_{\Ind{(\CC)}}(\alpha,\beta):=\lim{}_{i\in{\mathcal{I}}}\colim_{j\in{\mathcal{J}}}\Hom_{\CC}(\alpha(i),\beta(j)).
    \]
\end{definition}

{
\begin{definition}
    A $\Z$-graded $\mathbb{F}$-linear category $\CC_{dg}$ is a \emph{dg-category} if 
    its morphism spaces are differential graded $\Z$-modules, and
    morphism compositions 
    $\Hom_{\CC_{dg}}(X,Z)\otimes{\Hom}(Y,X)\rightarrow{\Hom(Y,Z)}$ 
    are differential graded $\Z$-module homomorphisms.
    That is, each morphism space decomposes as a direct sum of graded pieces 
    $\Hom(X,Y) = \bigoplus_{k \in \Z} \Hom^k(X,Y)$,
    where $\Hom^k(X,Y)$ is the space of homogeneous degree-$k$ morphisms from $X$ to $Y$.
    The morphism spaces $\Hom(X,Y)$ form cochain complexes $(\Hom(X,Y), d)$, where $d: \Hom^k(X,Y) \to \Hom^{k+1}(X,Y)$ satisfies $d^2 = 0$ 
    and
    morphism compositions are chain maps.
\end{definition}
}

For the category of chain complexes $\Kom(\CC)$ over 
{
a $\Z \oplus \Z$-graded $\mathbb{F}$-linear category
} $\CC$, the differentials of morphism spaces are defined as commutators with internal differentials,
i.e. for $f\in{\Hom^{k}_{\Kom(\CC)}(M^{\bullet},N^{\bullet})}$, the differential is $d(f)=d_{N}\circ f-(-1)^{k}f\circ d_{M}$.

\begin{definition}\label{def:twistedcomplex}
    A \emph{one-sided twisted complex} in the dg-category $\Kom(\CC)$ is a collection of chain complexes and chain maps
$\{B_{i},g_{i,j}:B_{i}\rightarrow{B_{j}}\}$  such that if $i\geq{j}$, then $g_{i,j}=0$, and the morphisms satisfy
    \begin{equation}
    \label{eq:twisted-condition}
    (-1)^{j}d(g_{i,j})+\sum_{k}g_{k,j}\circ{g_{i,k}}=0.
    \end{equation}
\end{definition}

Throughout, all twisted complexes will be one-sided, so we refer to them simply as twisted complexes. Let $\text{Tw}(\CC)$ denote the dg-category of twisted complexes over {$\CC$}.

\begin{definition}
    There is a functor $\Tot:\text{Tw}(\CC)\rightarrow{\Kom(\CC)}$ sending a twisted complex $B=\{\{B_{i}\},g_{i,j}:B_{i}\rightarrow{B_{j}}\}$ to its \emph{total complex}, denoted $\Tot(B)$, given by $\Tot(B):=
    \{\bigoplus_{i}B_{i}[i],d\}$. The brackets denote homological degree shifts and the differential $d$ is given by
    \[
    d:=\begin{bmatrix}
        d_{B_{0}} & 0 & 0 & \cdots\\
        g_{0,1} & -d_{B_{1}} & 0 & \cdots\\
        g_{0,2} & g_{1,2} & d_{B_{2}} & \cdots\\
        \vdots & \vdots & \vdots & \ddots
    \end{bmatrix}
    \]
\end{definition}

We use the same notation, $\Tot(A)$, to denote the \emph{total complex} of a double complex $A$.

The \emph{homotopy category} of chain complexes $K(\CC)$ is the category with the same objects as
{$\Kom(\CC)$}, but with morphisms taken up to chain homotopy. 
In fact, the morphisms are precisely given by $H^0(\Hom_{\Kom(\CC)}(X,Y))$.

Let $\mathcal{A}$ denote the following directed system $(A_{k},d_{k})$, where each $A_{k}$ is a chain complex with internal differential $d_{k}$:
\begin{equation}
\label{eq:directed-system-A}
    \mathcal{A} := A_0 \map{f_0} A_1 \map{f_1} A_2 \map{f_2} \cdots.
\end{equation}
Let $\mathcal{D}_{\mathcal{A}}$ denote the double complex in Figure \ref{fig:double-complex-B}. The differentials from the bottom row to the top row ($\idmap$ and $-f_k$ maps)  commute with the internal differentials $d_k$. The square brackets indicate homological shift, and also serve to differentiate the vertices of the diagram.

    \begin{figure}[t]
        \centering
        $\mathcal{D}_{\mathcal{A}} :=$ 
        \begin{tikzcd}
        A_0[0]
            & A_1[0] 
            & A_2[0] 
            & \text{\phantom{A4cell}}  \\
        A_0[-1] \arrow{u}{\idmap} \arrow{ur}{-f_{0}}
            & A_1[-1] \arrow{u}{\idmap} \arrow{ur}{-f_{1}}
            & A_2[-1] \arrow{u}{\idmap} \arrow{ur}{-f_2}
            & \\
        \end{tikzcd}
        $\cdots$
        \caption{The 2-term double complex associated to a directed system $\mathcal{A}$.}
        \label{fig:double-complex-B}
    \end{figure}

For such a double complex $\mathcal{D}_{\mathcal{A}}$, in the Ind-completion $\Ind(K(\CC))$ we have the \emph{totalization} of $\mathcal{D}_{\mathcal{A}}$, denoted $\Tot(\mathcal{D}_{\mathcal{A}})$, with signs chosen as in Figure \ref{fig:tot-B}. 

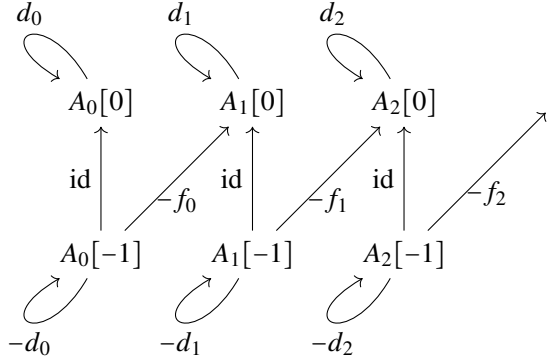
\begin{figure}[t]
    \centering
    \begin{tikzpicture}
        \def\s{2}; 
        \node (1N) at (1*\s,0*\s) {$A_0[0]$};
        \node (2N) at (2*\s,0*\s) {$A_1[0]$};
        \node (3N) at (3*\s,0*\s) {$A_2[0]$};
        \node (4N) at (4*\s,0*\s) {}; 
        \node (1S) at (1*\s,-1*\s) {$A_0[-1]$};
        \node (2S) at (2*\s,-1*\s) {$A_1[-1]$};
        \node (3S) at (3*\s,-1*\s) {$A_2[-1]$};
        \draw[->,out=120,in=150,looseness=8] (1N) to node[above]{$d_0$}  (1N);
        \draw[->,out=120,in=150,looseness=8] (2N) to node[above]{$d_1$}  (2N);
        \draw[->,out=120,in=150,looseness=8] (3N) to node[above]{$d_2$}  (3N);
        \draw[->,out=240,in=210,looseness=8] (1S) to node[below]{$- d_0$}  (1S);
        \draw[->,out=240,in=210,looseness=8] (2S) to node[below]{$- d_1$}  (2S);
        \draw[->,out=240,in=210,looseness=8] (3S) to node[below]{$- d_2$}  (3S);
        \draw[->] (1S) to node[left]{$\idmap$} (1N);
        \draw[->] (2S) to node[left]{$\idmap$} (2N);
        \draw[->] (3S) to node[left]{$\idmap$} (3N);
        \draw[->] (1S) to node[below]{$-f_0$} (2N);
        \draw[->] (2S) to node[below]{$-f_1$} (3N);
        \draw[->] (3S) to node[below]{$-f_2$} (4N);
    \end{tikzpicture}
    \caption{The total complex $\Tot(\mathcal{D}_{\mathcal{A}})$ of the double complex $\mathcal{D}_{\mathcal{A}}$. The components of the total differential $\partial^{\Tot}$ are depicted as the arrows.}
    \label{fig:tot-B}
\end{figure}

\begin{proposition}\cite[Proposition 2.28]{hog-youngsym}\label{prop:tot=twistedcomplex}
    Let $\mathcal{A}=(A_{0}\xrightarrow{f_{0}}{A_{1}}\xrightarrow{f_{1}}{A_{2}}\xrightarrow{f_{2}}{\cdots})$ be a directed system of chain complexes such that $f_{i+1}\circ{f_{i}}=0$ for all $i$. Assume also that each chain complex $A_{i}$ is homotopy equivalent to a corresponding chain complex $B_{i}$, then $\Tot(\mathcal{D}_\mathcal{A})$ is homotopy equivalent to the total complex of a twisted complex of the form
    \[
        \begin{tikzcd}
        B_{0} \arrow[r, "{g_{0,1}}"] \arrow[rr, "{g_{0,2}}", bend right=49] \arrow[rrr, "{g_{0,3}}", bend right=60] & B_{1} \arrow[r, "{g_{1,2}}"] \arrow[rr, "{g_{1,3}}", bend right=49] & B_{2} \arrow[r, "{g_{2,3}}"] & B_{3} \arrow[r] & \cdots
        \end{tikzcd},
    \]
where each $g_{i,j}$ is a map of homological degree $j-i-1$.
\end{proposition}

{
If the homotopy equivalences are given by the collection of maps $\phi_i: A_i \to B_i$, with homotopy inverses $\bar \phi_i$, then the maps $g_{i, i+1}$ may be chosen to be $\phi_{i+1}\circ f_i \circ \bar \phi_i$, so that the following diagram commutes (on the nose):
}
\begin{equation}
\label{eq:homotopy-commutes}
\begin{tikzcd}
    A_0 \arrow{r}{f_0}  \arrow{d}{\phi_0}
        & A_1 \arrow{r}{f_1} \arrow{d}{\phi_1}
        & A_2 \arrow{r}{f_2} \arrow{d}{\phi_2}
        & A_3 \arrow{r}{} \arrow{d}{\phi_3}
        & \cdots \\
    B_0 \arrow{r}{g_{0,1}} 
        & B_1 \arrow{r}{g_{1,2}} 
        & B_2 \arrow{r}{g_{2,3}} 
        & B_3 \arrow{r}{}
        & \cdots \\
\end{tikzcd}
\end{equation}
{
Then $g_{i+1, i+2} \circ g_{i, i+1} 
= (\phi_{i+2} \circ  f_{i+1} \circ \bar \phi_{i+1} ) \circ (\phi_{i+1} \circ f_i \circ \bar \phi_i)
\sim \phi_{i+2} \circ (f_{i+1} \circ f_i) \circ \bar \phi_i
\sim 0$, which is why we include the higher homotopies $g_{i,j}$ where $j-i > 1$ in the statement of Proposition \ref{prop:tot=twistedcomplex}.
}

\begin{definition}
\label{defn:colimit}
(Universal property for colimits)
A \emph{colimit} $C$ of the directed system \eqref{eq:directed-system-A} is an object in $\Ind(K(\CC))$ satisfying the following two properties:
\begin{enumerate}
    \item[(C-1)] 
    There are chain maps $(\phi_k)$ such that 
    \begin{center}
    \begin{tikzcd}
        A_k \arrow{rr}{f_k} \arrow[swap]{ddr}{\phi_k} 
            &
            & A_{k+1} \arrow{ddl}{\phi_{k+1}} \\
          & & \\
        & C & \\
    \end{tikzcd}
    \end{center}
    commutes up to homotopy for all $k \in \N$. That is, there are homotopies $(h_k:A_{k}\rightarrow{C})$ such that
    \begin{equation}
    \label{eq:colim-structure-maps}
        \phi_k - \phi_{k+1} \circ f_k = h_k \circ d_{A_k} + d_C \circ h_k.
    \end{equation}
    \item[(C-2)] 
    Let $C'$ be an object satisfying 
    (C-1), with structure maps $(\phi_k')$ and homotopies $(h_k')$. Then there exists a chain map $\xi$, unique up to homotopy, making the following diagram homotopy-commute for all $k \in \N$:
    \begin{center}
    \begin{tikzcd}
        A_k 
            \arrow{rr}{f_k} 
            \arrow[swap]{ddr}{\phi_k'}
            \arrow[swap,bend right]{ddddr}{\phi_k}
            &
            & A_{k+1} 
                \arrow{ddl}{\phi_{k+1}'} 
                \arrow[bend left]{ddddl}{\phi_{k+1}} \\
          & & \\
        & C' & \\
         & & \\
        & C \arrow[dotted]{uu}{\xi} & \\
    \end{tikzcd}
    \end{center}
\end{enumerate}
\end{definition}

By standard category theory arguments, the reader may check that if an object $C$ satisfies the conditions in Definition \ref{defn:colimit} exists, then it is unique up to homotopy equivalence.

\begin{definition} 
    Let $\mathcal{A}=(A_{0}\xrightarrow{f_{0}}{A_{1}}\xrightarrow{f_{1}}{A_{2}}\xrightarrow{f_{2}}{\cdots})$ be a directed system of chain complexes. Suppose that the homotopy category $K(\mathcal{C})$ contains infinite direct sums, then the \emph{homotopy colimit} of $\mathcal{A}$, denoted $\hocolim{\mathcal{A}}$, can be identified with the total complex of the $2$-term complex $\mathcal{D}_{\mathcal{A}}$ (see Figure \ref{fig:double-complex-B}).
\end{definition}

We now show $\hocolim(\mathcal{A})$ and $\colim(\mathcal{A})$ are equivalent in this context.

\begin{proposition}
\label{prop:totB-is-colim}
    We have that $\hocolim(\mathcal{A})=\Tot(\mathcal{D}_{\mathcal{A}})$ satisfies the conditions of Definition \ref{defn:colimit}. 
\end{proposition}

Note that this is true, as the indexing category is freely generated by the graph $(\bullet\rightarrow{\bullet}\rightarrow{\bullet\rightarrow{\cdots}})$, hence the homotopy colimit of $\mathcal{A}$ is the representing object of the homotopy-commutative version of the cocone of our directed system \cite[Section 10]{Shul-colimits}. However, we prove it now explicitly.

\begin{proof}
To check condition (C-1) in Definition \ref{defn:colimit}, let $\mathcal{A}=(A_{0}\xrightarrow{f_{0}}{A_{1}}\xrightarrow{f_{1}}{A_{2}}\xrightarrow{f_{2}}{\cdots})$ be a directed system in $\CC$, and let $\Tot{(\mathcal{D}_{\mathcal{A}})}$ denote the 2-term chain complex given in Figure \ref{fig:tot-B}. Note that we may define the following collection of maps $\phi^{\Tot}:=\{\phi_{k}^{\Tot}:A_{i}\rightarrow{\Tot(\mathcal{D}_{\mathcal{A}})\}}$ by:

\begin{center}
\begin{equation}
\label{eq:diagram1-prop2.12}
\begin{tikzcd}
\cdots \arrow[r]  & A_{k-1} \arrow[loop, "d_{A_{k-1}}", distance=2em, in=55, out=125] \arrow[r, "f_{k-1}"] \arrow[ddd, red, "\phi^{\Tot}_{k-1}"] & A_{k} \arrow[loop, "d_{A_{k}}", distance=2em, in=55, out=125] \arrow[r, "f_{k}"] \arrow[ddd, red, "\phi^{\Tot}_{k}"] & A_{k+1} \arrow[loop, "d_{A_{k+1}}", distance=2em, in=55, out=125] \arrow[r] \arrow[ddd, red, "\phi^{\Tot}_{k+1}"] & \cdots \\
  & & & & \\
  & & &  & \\
\cdots  & {A_{k-1}[0]} \arrow[loop, distance=2em, in=55, out=125]         & {A_{k}[0]} \arrow[loop, distance=2em, in=55, out=125]        & {A_{k+1}[0]} \arrow[loop, distance=2em, in=55, out=125]       & \cdots \\
\cdots \arrow[ru] & {A_{k-1}[-1]} \arrow[u, "\idmap"] \arrow[ru, "-f_{k-1}"] \arrow["-d_{A_{k-1}}"', loop, distance=2em, in=305, out=235]   & {A_{k}[-1]} \arrow[u, "\idmap"] \arrow[ru, "-f_{k}"] \arrow["-d_{A_{k}}"', loop, distance=2em, in=305, out=235] & {A_{k+1}[-1]} \arrow[u, "\idmap"] \arrow["-d_{A_{k+1}}"', loop, distance=2em, in=305, out=235] \arrow[ru]    & \cdots
\end{tikzcd}
\end{equation}
\end{center}

where $\phi^{\Tot}_{k}:=-\idmap_{A_{k}}$. We first verify that $\phi_{k}^{\Tot}$ is a chain map. Letting $d_{A_k}$ denote the internal differential of $A_{k}$, note that $\partial^{\Tot}\circ{\phi_{k}^{\Tot}}=d_{A_k}\circ{-\idmap_{A_{k}}}=-\idmap_{A_{k}}\circ{d_{A_{k}}}=\phi^{\Tot}_{k}\circ{d_{A_k}}$. We then verify that the $\phi_{k}^{\Tot}$ maps commute with $f_{k}$ maps up to homotopy; $\phi^{\Tot}_{k}\sim{\phi_{k+1}^{\Tot}\circ{f_{k}}}$. We define a new collection of maps $h^{\Tot}:=\{h_{k}^{\Tot}\}$, where $h_{k}^{\Tot}:=-\idmap_{A_{k}}$ from $A_{k}$ to the copy of $A_{k}$ in degree $-1$.

\begin{center}
\begin{equation}
\begin{tikzcd}
\cdots \arrow[r]  & A_{k-1} \arrow[r, "f_{k-1}"] \arrow[ddd, red, "\phi^{\Tot}_{k-1}" description] \arrow[dddd, blue, "h_{k-1}^{\Tot}" description, bend right=49] \arrow[loop, "d_{A_{k-1}}", distance=2em, in=55, out=125] & A_{k} \arrow[r, "f_{k}"] \arrow[ddd, red, "\phi^{\Tot}_{k}" description] \arrow[dddd, blue, "h_{k}^{\Tot}" description, bend right=49] \arrow[loop, "d_{A_{k}}", distance=2em, in=55, out=125] & A_{k+1} \arrow[r] \arrow[ddd, red, "\phi^{\Tot}_{k+1}" description] \arrow[dddd, blue, "h_{k+1}^{\Tot}" description, bend right=49] \arrow[loop, "d_{A_{k+1}}", distance=2em, in=55, out=125] & \cdots \\
 &  &   &  & \\
& & &  &  \\
\cdots  & {A_{k-1}[0]} \arrow[loop, distance=2em, in=55, out=125]       & {A_{k}[0]} \arrow[loop, distance=2em, in=55, out=125]  & {A_{k+1}[0]} \arrow[loop, distance=2em, in=55, out=125]  & \cdots \\ \cdots \arrow[ru] & {A_{k-1}[-1]} \arrow[u, "\idmap"] \arrow[ru, "-f_{k-1}"] \arrow["-d_{A_{k-1}}"', loop, distance=2em, in=305, out=235] & {A_{k}[-1]} \arrow[u, "\idmap"] \arrow[ru, "-f_{k}"] \arrow["-d_{A_{k}}"', loop, distance=2em, in=305, out=235]   & {A_{k+1}[-1]} \arrow[u, "\idmap"] \arrow[ru] \arrow["-d_{A_{k+1}}"', loop, distance=2em, in=305, out=235] & \cdots
\end{tikzcd}
\end{equation}
\end{center}

Note that $\phi^{\Tot}_{k}-\phi^{\Tot}_{k+1}\circ{f_{k}}=-\idmap_{A_{k}}+\idmap_{A_{k+1}}\circ{f_{k}}=-\idmap_{A_{k}}+f_{k}$. However, $\partial^{\Tot}\circ{h^{\Tot}_{k}}+h^{\Tot}_{k}\circ{d_{A_{k}}}=-(-d_{A_{k}}+\idmap_{A_{k}}-f_{k})-d_{A_{k}}=\phi_{k}^{\Tot}-\phi^{\Tot}_{k+1}\circ{f_{k}}$, thus, the diagram in Equation \eqref{eq:diagram1-prop2.12} commutes up-to-homotopy. Next, for \ref{defn:colimit} (C-2), let $B$ be an arbitrary chain complex in $\Kom$ with internal differential $d_{B}$, and suppose that we have structure chain maps $\phi^{B}:=\{\phi_{i}^{B}:A_{i}\rightarrow{B}\}$, and homotopies $h^{B}:=\{h_{i}^{B}\}$, such that $\phi_{k}^{B}-\phi_{k+1}^{B}\circ{f_{k}}=d_{B}\circ{h_{k}^{B}+h_{k}^{B}\circ{d_{A_k}}}$. Let $\overline{\phi}^{B}_{k}:=-\phi^{B}_{k}$ and $\overline{h}_{k}^{B}:=-h_{k}^{B}$. We may define a the chain map $\xi:\Tot(\mathcal{D}_{\mathcal{A}})\rightarrow{B}$ by assembling both collections $\{\overline{h}^{B}_{k}\}$ and $\{\overline{\phi}^{B}_{k}\}$:

\begin{center}
\begin{equation}
\begin{tikzcd}
\cdots \arrow[r]  & B \arrow[r, "id"] \arrow["d_{B}", loop, distance=2em, in=55, out=125]  & B \arrow[r, "id"] \arrow["d_{B}", loop, distance=2em, in=55, out=125] & B \arrow[r] \arrow["d_{B}", loop, distance=2em, in=55, out=125]        & \cdots \\
 & & & & \\
&  & &  &  \\
\cdots   & {A_{k-1}[0]} \arrow[loop, distance=2em, in=55, out=125] \arrow[uuu, purple, "\overline\phi^{B}_{k-1}" description] & {A_{k}[0]} \arrow[loop, distance=2em, in=55, out=125] \arrow[uuu, purple, "\overline\phi^{B}_{k}" description]  & {A_{k+1}[0]} \arrow[loop, distance=2em, in=55, out=125] \arrow[uuu, purple, "\overline\phi^{B}_{k+1}" description]  & \cdots \\
\cdots \arrow[ru] & {A_{k-1}[1]} \arrow[u, "\idmap"] \arrow[ru, "-f_{k-1}"] \arrow["-d_{A_{k-1}}"', loop, distance=2em, in=305, out=235] \arrow[uuuu, teal, "\overline h_{k-1}^{B}" description, bend left=49] & {A_{k}[1]} \arrow[u, "\idmap"] \arrow[ru, "-f_{k}"] \arrow["-d_{A_{k}}"', loop, distance=2em, in=305, out=235] \arrow[uuuu, teal, "\overline h_{k}^{B}" description, bend left=49] & {A_{k+1}[1]} \arrow[u, "\idmap"] \arrow[ru] \arrow["-d_{A_{k+1}}"', loop, distance=2em, in=305, out=235] \arrow[uuuu, teal, "\overline h_{k+1}^{B}" description, bend left=49] & \cdots
\end{tikzcd}
\end{equation}
\end{center}

Note first that $\xi$ is a chain map:
\begin{align*}
    d_{B}\circ{\xi}-\xi\circ{\partial^{\Tot}}&=d_{B}\circ{\overline{h}^{B}_{k}}-\xi\circ(-d_{A_k}+\idmap_{A_{k}}-f_{k})\\
    &=-d_{B}\circ{h^{B}_{k}}-(-\overline{h}_{k}^{B}\circ{d_{A_{k}}}+\overline{\phi}^{B}_{k}-\overline{\phi}^{B}_{k+1}\circ{f_{k}})\\
    &=-d_{B}\circ{h_{k}^{B}}+\overline{h}_{k}^{B}\circ{d_{A_{k}}}-\overline{\phi}_{k}^{B}+\overline{\phi}^{B}_{k+1}\circ{f_{k}}\\
    &=-(d_{B}\circ{h_{k}^{B}+h_{k}^{B}\circ{d_{A_{k}}})}-(-(d_{B}\circ{h_{k}^{B}}+h_{k}^{B}\circ{d_{A_{k}}}))\\
    &=0.
\end{align*}
Also, $\xi$ is clearly a chain homotopy equivalence map, as $\xi\circ{\phi_{k}^{\Tot}}=-\xi\circ{\idmap_{A_{k}}}=-\overline{\phi}_{k}^{B}=\phi_{k}^{B}$. Thus, $\Tot(\mathcal{D}_{\mathcal{A}})$ satisfies the universal property of the colimit of $\mathcal{A}$ up to homotopy, and Proposition \ref{prop:totB-is-colim} follows.
\end{proof}

The specific directed systems we study will satisfy the property that $f_{i+1}\circ{f_{i}}$ is zero. The corresponding homotopy colimits then admit useful properties, which we now discuss.

\begin{proposition}
\label{prop:contractiblehocolim}
Let $\mathcal{D}_{\mathcal{A}}$ be the double complex associated to a directed system $\mathcal{A}:=(A_{0}\xrightarrow{f_{0}}{A_{1}}\xrightarrow{f_{1}}{A_{2}}\xrightarrow{f_{2}}{\cdots})$ and let $\Tot(\mathcal{D}_{\mathcal{A}})$ denote the corresponding homotopy colimit (as in \eqref{fig:tot-B}). If $f_{i+1}\circ{f_{i}}\sim{0}$ for all $i\in{\Z_{\geq{0}}}$, then $\Tot(\mathcal{D}_{\mathcal{A}})$ is contractible.
\end{proposition}

\begin{proof}
The proof of this proposition follows directly from the following lemma.

\begin{lemma}
\label{lem:endoB}

    There exists an endomorphism $F:\mathcal{D}_{\mathcal{A}}\rightarrow{\mathcal{D}_{\mathcal{A}}}$, given by chain maps $F_{i}:A_{i}\rightarrow{A_{i+1}}$, where each $F_{i}$ is a copy of the chain map $f_{i}$. The endomorphism $F$ is denoted by the dotted arrows in the following figure.

    \begin{center}
        \begin{tikzcd}
        {A_{0}[0]} \arrow[r, dashed, magenta, "{F_{0}}"] & {A_{1}[0]} \arrow[r, dashed, magenta, "{F_{1}}"]  & {A_{2}[0]} \arrow[r, dashed, magenta]            & \cdots \\
        {A_{0}[-1]} \arrow[u, "\idmap"] \arrow[ru, "-f_{0}"] \arrow[r, dashed, magenta, "{F_{0}}"'] & {A_{1}[-1]} \arrow[u, "\idmap"] \arrow[ru, "-f_{1}"] \arrow[r, dashed, magenta, "{F_{1}}"'] & {A_{2}[-1]} \arrow[u, "\idmap"] \arrow[ru, "-f_{2}"] \arrow[r, dashed, magenta] & \cdots
        \end{tikzcd}
    \end{center}
    \vspace{2.5mm}
    The endomorphism $F:\mathcal{D}_{\mathcal{A}}\rightarrow{\mathcal{D}_{\mathcal{A}}}$ is homotopic to $\idmap_{\mathcal{D}_{\mathcal{A}}}$.
\end{lemma}

\begin{proof}
    The chain homotopy maps are given by identity maps $\idmap_{A_{i}}^{-1}:A_{i}[0]\rightarrow{A_{i}[-1]}$. Note that $[\partial^{\Tot},\idmap^{-1}_{A_{i}}]=\partial^{\Tot}\circ{\idmap_{A_{i}}^{-1}+\idmap_{A_{i}}^{-1}\circ{d_{A_{i}}}=(-d_{A_{i}}-f_{i}+\idmap_{A_{i}})+d_{A_{i}}}=-f_{i}+\idmap_{A_{i}}$; thus, the chain map $F$ is homotopic to $\idmap_{\mathcal{D}_{\mathcal{A}}}$.
\end{proof}

To complete the proof of Proposition \ref{prop:contractiblehocolim}, note that the condition $f_{i+1}\circ f_{i}\sim{0}$ implies that $F^{2}\sim{0}$. Thus, by Lemma \ref{lem:endoB}, we have that $\idmap_{\mathcal{D}_{\mathcal{A}}}\sim{F}\sim{F^{2}}\sim{0}$, proving the claim.
\end{proof}

Recall that homotopy equivalences of objects in a directed system can be extended to chain homotopy equivalences of homotopy colimits of said directed system.

\begin{lemma}\label{lem:forSwapsandSlides}
    Let $\mathcal{A}:=\{A_{i},f_{i}\}_{i\in{\mathbb{Z}_{\geq{0}}}}$ and $\mathcal{B} :=\{B_{i},g_{i}\}_{i\in{\mathbb{Z}}_{\geq{0}}}$ be directed systems in $\CC$.
    Suppose we have a collection of chain maps $\alpha_{i}:A_{i}\rightarrow{B_{i}}$ such that $\alpha_{i+1}\circ{f_{i}}\sim{g_{i}\circ{\alpha_{i}}}$, and let $C_{i}$ denote the cone $\Cone(A_{i}\xrightarrow{\alpha_{i}}B_{i})$.
 
    \begin{enumerate}
        \item[(a)]{There exists a chain map $\alpha:\hocolim(\mathcal{A})\rightarrow{\hocolim(\mathcal{B})}$ corresponding to the collection $\{\alpha_{i}\}$.}\\
        \item[(b)]{We have that $\cone{(\alpha)}=\hocolim(\cone{(A_{i}\xrightarrow{\alpha_{i}}B_{i}}))$}.\\
        \item[(c)] If $\alpha_{i}$ is a homotopy equivalence for all $i\in{\Z_{\geq{0}}}$, then $C_{i}\simeq{0}$ and $\hocolim{(C_{i})}\simeq{0}$ for all $i$, so $\cone{(\alpha)}\simeq{0}$.
    \end{enumerate}
\end{lemma}

\begin{proof}
\begin{enumerate}
    \item[(a)] The induced chain map $\alpha$ is given by each $\alpha_{i}:A_{i}[k]\rightarrow{B_{i}[k]}$ for all $k\in{\{0,1\}}$. The relevant homotopy maps are given by $h_{i}:A_{i}[0]\rightarrow{B_{i+1}[1]}$.\\
\item[(b)] Define a map $\Phi_{i}:C_{i}\rightarrow{C_{i+1}}$ by the following diagram:

\begin{center}
\begin{tikzcd}
A_{i} \arrow[rr, "\alpha_{i}"] \arrow[dd, "f_{i}"] \arrow[rrdd, "h_{i}"] &  & B_{i} \arrow[dd, "g_{i}"] \\
&  &  \\
A_{i+1} \arrow[rr, "\alpha_{i+1}"] &  & B_{i+1}                  
\end{tikzcd}
\end{center}
Then $\cone{(\alpha)}=\hocolim{(C_{0}\xrightarrow{\Phi_{1}}C_{1}\xrightarrow{{\Phi_{2}}}C_{2}\rightarrow{\cdots}})$.\\
\item[(c)] Suppose that each $\alpha_{i}$ is a homotopy equivalence. Then each cone $C_{i}$ is contractible, and therefore the homotopy colimit is contractible:
\[
    \hocolim{(C_{0}\xrightarrow{\Phi_{0}}C_{1}\xrightarrow{{\Phi_{1}}}C_{2}\rightarrow{\cdots}})\simeq{0}.
\] 
This implies $\cone(\alpha)\simeq{0}$ by part (b); therefore $\alpha$ is a homotopy equivalence.
\end{enumerate}
\end{proof}

Finally, we recall the following standard lemma from homological algebra.

\begin{lemma}\label{lem:Hcone}
Let $X,Y$ be complexes of vector spaces over $\BF$, and let $f:X \to Y$ be a chain map. Then
\[
    H^*\left ( \cone(f) \right ) \cong H^*\left ( 
        \cone \left (H^*(X) \map{f^*} H^*(Y) \right )
        \right ).
\]
\begin{proof}
    The short exact sequence of chain complexes 
    \[
        0 \to Y \map{\iota} \Cone(f) \map{\pi} X[1] \to 0
    \]
    induces a long exact sequence on homology
    \[
        \cdots \map{f^{*}} H^i(Y) \map{\iota^*} 
        H^i(\Cone(f)) \map{\pi^*}
        H^{i+1}(X) \map{f^*} \cdots, 
    \]
    i.e.\ there is an exact triangle 
    \[
        \begin{tikzcd}
         H^*(Y) \arrow{rr}{\iota^*} 
            & & H^*(\Cone(f)) \arrow{dl}{\pi^*} \\
        & H^*(X[1]) \arrow{ul}{f^*} &
        \end{tikzcd}
    \]
\end{proof}
\end{lemma}

In Section \ref{sec:mainsec}, we use Lemma \ref{lem:Hcone} to compute homotopy colimits of complexes by passing to graded vector spaces:

\begin{corollary}\label{cor:Hcone}
Let $\CC = \{C_i, f_i\}_{i \in \Z_{\geq 0}}$ be a directed system of complexes $(C_i, d_i)$ of vector spaces over $\BF$. Then the associated homotopy colimit can be computed by first computing the homology of each $C_i$: 
\[
    \hocolim(\CC) 
    \quad \simeq  \quad 
    \begin{tikzcd}
        H^*(C_0)
            & H^*(C_1) 
            & H^*(C_2) 
            & \ \\
        H^*(C_0) \arrow{u}{\idmap} \arrow{ur}{-f_0^*}
            & H^*(C_1) \arrow{u}{\idmap} \arrow{ur}{-f_1^*}
            & H^*(C_2) \arrow{u}{\idmap} \arrow{ur}{-f_2^*} 
            &
    \end{tikzcd}
    \quad 
    \cdots
\]
\end{corollary}

\section{Khovanov-Rozansky homology, Bar-Natan categories, and categorified projectors}
\label{sec:BN-cats}

We assume the reader is already roughly familiar with Khovanov's categorification of the Jones polynomial;  for references, see \cite{Kh00,BN-Kh,BN-tangles}. Throughout, we follow the conventions used in \cite{MWW-lasagna, Hog-polyn-action, MN22}. In particular, the quantum degree of cobordisms is reversed from that of \cite{BN-tangles}.

\subsection{Conventions and notations for $\KhR_2$}

We follow the conventions of \cite{MWW-lasagna} and \cite{MN22}, which we briefly recall and collect here in this section. The details of the construction of $\KhR_2$ are left to Section \ref{sec:BN-category-definition}, where we more carefully review Bar-Natan's categories, with grading choices determined by the conventions in the present section.

Let $L$ be a framed, oriented link in $\R^3$. By an abuse of notation, we also let $L$ denote a fixed diagram for this link. 

The $\gl_2$ Khovanov-Rozansky homology (See \cite{KR-matrix-factorizations})
of $L$ over $\BF$, denoted $\KhR_2(L)$, is a bigraded vector space
\[
\KhR_{2}(L)=\bigoplus_{i,j \in{\mathbb{Z}}}\KhR_{2}^{i,j}(L)
\]
where $i$ and $j$ denote the \emph{homological grading} and internal \emph{quantum grading} respectively.

These homology groups are computed via an iterated mapping cone construction (or equivalently, tensor product of mapping cones) defined by the following two-term complexes associated to positive and negative crossings in the diagram for $L$, respectively. 
(For background on iterated mapping cones, see Section 4.1 of \cite{OS-dbc}, for instance.)

\begin{equation}
\label{eq:oriented-skein-relations}
    
    \begin{tikzpicture}[scale=.3, baseline=-.67ex]
        
\begin{scope}[xshift=-.5cm, yshift=-.5cm]
\draw[->] (1,0) -- (0,1);
\fill[white] (.5,.5) circle (.2cm);
\draw[->] (0,0) -- (1,1);
\end{scope} 
    \end{tikzpicture}

    \ = \ h \inv q \ 
    
    \begin{tikzpicture}[scale=.3, baseline=-.67ex,rotate=-90]

\begin{scope}[xshift=-.5cm, yshift=-.5cm]
\draw (0,0) .. controls (.5, .5).. (0,1);
\draw (1,0) .. controls (.5, .5).. (1,1);
\end{scope}
    \end{tikzpicture}

    \to 
    
    \begin{tikzpicture}[scale=.3, baseline=-.67ex]

\begin{scope}[xshift=-.5cm, yshift=-.5cm]
\draw (0,0) .. controls (.5, .5).. (0,1);
\draw (1,0) .. controls (.5, .5).. (1,1);
\end{scope}
    \end{tikzpicture}

    \qquad
    \qquad

    \begin{tikzpicture}[scale=.3, baseline=-.67ex]
        
\begin{scope}[xshift=-.5cm, yshift=-.5cm]
\draw[->] (0,0) -- (1,1);
\fill[white] (.5,.5) circle (.2cm);
\draw[->] (1,0) -- (0,1);
\end{scope}
    \end{tikzpicture}

    \ = \ 
    
    \begin{tikzpicture}[scale=.3, baseline=-.67ex]

\begin{scope}[xshift=-.5cm, yshift=-.5cm]
\draw (0,0) .. controls (.5, .5).. (0,1);
\draw (1,0) .. controls (.5, .5).. (1,1);
\end{scope}
    \end{tikzpicture}

    \to \ h q\inv \ 
    
    \begin{tikzpicture}[scale=.3, baseline=-.67ex,rotate=-90]

\begin{scope}[xshift=-.5cm, yshift=-.5cm]
\draw (0,0) .. controls (.5, .5).. (0,1);
\draw (1,0) .. controls (.5, .5).. (1,1);
\end{scope}
    \end{tikzpicture}

\end{equation}
Here, multiplication by $h$ and $q$ respectively indicate the formal shift in homological and quantum gradings.
In various contexts in this article, it will also be necessary for us to use the bracket notation to indicate homological and quantum shift:
\[
    \BNbrack{L} [k] \{ \ell \} = h^k q^{\ell} \BNbrack{L}.
\]
Observe that we usually omit the Bar-Natan brackets $\BNbrack{\cdot}$ in figures; the reader may assume that all diagrams represent their algebraic counterpart in the appropriate Bar-Natan category.

The $\KhR_2$ theory is functorial for links in $S^{3}$ and link cobordisms in $S^{3}\times{[0,1]}$. Morrison-Walker-Wedrich (Theorem $3.3$ of \cite{MWW-lasagna}) showed that the \emph{sweep-around move} cobordism induces the identity map on $\KhR_2$. See \cite{CaprauFix, CMW09, Bla10, San21, Vog20, ETW, MR4598808}) for functoriality of Khovanov homology for links in $\R^3$. 

Let $\Sigma\subset{S^{3}\times{[0,1]}}$ be a properly embedded, framed, oriented surface intersecting the boundary $3$-spheres $S^{3}\times{\{0\}}$ and $S^{3}\times{\{1\}}$ in links $L_{0}$ and $L_{1}$ respectively. By the functoriality of $\KhR_{2}$, there is a well-defined induced homogeneous linear map
\[
\KhR_{2}(\Sigma):\KhR_{2}(L_{0})\rightarrow{\KhR_{2}(L_{1})}.
\]
of bidegree $(0,-\chi(\Sigma))$. (This is a special case of \eqref{eq:quantum-degree-of-cobordism}.)

\subsection{Khovanov homology conventions}

There are numerous conventions floating around in the literature. We collect them here for reference. 
Let $D$ denote a diagram for an oriented link $L$. 

\begin{itemize}

\item 
Khovanov's original homology theory $\Kh(D)$, defined in \cite{Kh00}, uses the oriented skein relations
\begin{equation*}
    
    \begin{tikzpicture}[scale=.3, baseline=-.67ex]
        
\begin{scope}[xshift=-.5cm, yshift=-.5cm]
\draw[->] (1,0) -- (0,1);
\fill[white] (.5,.5) circle (.2cm);
\draw[->] (0,0) -- (1,1);
\end{scope} 
    \end{tikzpicture}
 \ = \ q 
    \begin{tikzpicture}[scale=.3, baseline=-.67ex]

\begin{scope}[xshift=-.5cm, yshift=-.5cm]
\draw (0,0) .. controls (.5, .5).. (0,1);
\draw (1,0) .. controls (.5, .5).. (1,1);
\end{scope}
    \end{tikzpicture}
 \to hq^2 
    \begin{tikzpicture}[scale=.3, baseline=-.67ex,rotate=-90]

\begin{scope}[xshift=-.5cm, yshift=-.5cm]
\draw (0,0) .. controls (.5, .5).. (0,1);
\draw (1,0) .. controls (.5, .5).. (1,1);
\end{scope}
    \end{tikzpicture}

    \qquad \qquad
    
    \begin{tikzpicture}[scale=.3, baseline=-.67ex]
        
\begin{scope}[xshift=-.5cm, yshift=-.5cm]
\draw[->] (0,0) -- (1,1);
\fill[white] (.5,.5) circle (.2cm);
\draw[->] (1,0) -- (0,1);
\end{scope}
    \end{tikzpicture}
 \ = \ h\inv q^{-2} 
    \begin{tikzpicture}[scale=.3, baseline=-.67ex,rotate=-90]

\begin{scope}[xshift=-.5cm, yshift=-.5cm]
\draw (0,0) .. controls (.5, .5).. (0,1);
\draw (1,0) .. controls (.5, .5).. (1,1);
\end{scope}
    \end{tikzpicture}
 \to q\inv 
    \begin{tikzpicture}[scale=.3, baseline=-.67ex]

\begin{scope}[xshift=-.5cm, yshift=-.5cm]
\draw (0,0) .. controls (.5, .5).. (0,1);
\draw (1,0) .. controls (.5, .5).. (1,1);
\end{scope}
    \end{tikzpicture}
.
\end{equation*}
Other reference works relevant to our discussion that use this convention include \cite{BN-Kh, BN-tangles, lee-endo, RasInv, GLW-schur-weyl, CoopKrush}. In particular, Bar-Natan introduces an unoriented bracket 
\begin{equation*}
    \BNbrack{
        
    \begin{tikzpicture}[scale=.3, baseline=-.67ex]
        
\begin{scope}[xshift=-.5cm, yshift=-.5cm]
\draw (1,0) -- (0,1);
\fill[white] (.5,.5) circle (.2cm);
\draw (0,0) -- (1,1);
\end{scope}
    \end{tikzpicture}

    }
    = 
    \begin{tikzpicture}[scale=.3, baseline=-.67ex]

\begin{scope}[xshift=-.5cm, yshift=-.5cm]
\draw (0,0) .. controls (.5, .5).. (0,1);
\draw (1,0) .. controls (.5, .5).. (1,1);
\end{scope}
    \end{tikzpicture}
 \to h q \ 
    \begin{tikzpicture}[scale=.3, baseline=-.67ex,rotate=-90]

\begin{scope}[xshift=-.5cm, yshift=-.5cm]
\draw (0,0) .. controls (.5, .5).. (0,1);
\draw (1,0) .. controls (.5, .5).. (1,1);
\end{scope}
    \end{tikzpicture}

\end{equation*}
and then adds a global shift:
\begin{equation*}
    \CKh(D) = \BNbrack{D}[-n_-]\{n_+ - 2n_-\}
\end{equation*} 
where $n_{\pm}$ is the number of positive/negative ($\pm$) crossings in $D$. 
Because of the early adoption of this convention by those interested in the relationships between Khovanov and Floer theories, this is still the convention for Khovanov homology appearing in the Floer literature. 
This is invariant under framing changes (i.e.\ Reidemeister I moves), as the global shift accounts for the writhe of the diagram.

\item On the other hand, the literature involving Khovanov's arc algebras and other tangle-related constructions (excluding \cite{BN-tangles}) usually uses the conventions from \cite{Kh-functor-valued, Kh-patterns}, where the quantum degree is reversed. We denote this ``new'' Khovanov convention by ${\newKh}$. Thus 
\begin{equation*}
    {\newKh(L)}^{i,j} 
    \cong 
    \Kh(L)^{i,-j}.
\end{equation*}
For reference, the oriented skein relations are below.
\begin{equation*}
    
    \begin{tikzpicture}[scale=.3, baseline=-.67ex]
        
\begin{scope}[xshift=-.5cm, yshift=-.5cm]
\draw[->] (1,0) -- (0,1);
\fill[white] (.5,.5) circle (.2cm);
\draw[->] (0,0) -- (1,1);
\end{scope} 
    \end{tikzpicture}
 \ = \ 
    q\inv 
    \begin{tikzpicture}[scale=.3, baseline=-.67ex]

\begin{scope}[xshift=-.5cm, yshift=-.5cm]
\draw (0,0) .. controls (.5, .5).. (0,1);
\draw (1,0) .. controls (.5, .5).. (1,1);
\end{scope}
    \end{tikzpicture}
 \to h q^{-2} 
    \begin{tikzpicture}[scale=.3, baseline=-.67ex,rotate=-90]

\begin{scope}[xshift=-.5cm, yshift=-.5cm]
\draw (0,0) .. controls (.5, .5).. (0,1);
\draw (1,0) .. controls (.5, .5).. (1,1);
\end{scope}
    \end{tikzpicture}

    \qquad \qquad
    
    \begin{tikzpicture}[scale=.3, baseline=-.67ex]
        
\begin{scope}[xshift=-.5cm, yshift=-.5cm]
\draw[->] (0,0) -- (1,1);
\fill[white] (.5,.5) circle (.2cm);
\draw[->] (1,0) -- (0,1);
\end{scope}
    \end{tikzpicture}
 \ = \ 
    h\inv q^2 
    \begin{tikzpicture}[scale=.3, baseline=-.67ex,rotate=-90]

\begin{scope}[xshift=-.5cm, yshift=-.5cm]
\draw (0,0) .. controls (.5, .5).. (0,1);
\draw (1,0) .. controls (.5, .5).. (1,1);
\end{scope}
    \end{tikzpicture}
 
    \to q 
    \begin{tikzpicture}[scale=.3, baseline=-.67ex]

\begin{scope}[xshift=-.5cm, yshift=-.5cm]
\draw (0,0) .. controls (.5, .5).. (0,1);
\draw (1,0) .. controls (.5, .5).. (1,1);
\end{scope}
    \end{tikzpicture}
.
\end{equation*}
This is also insensitive to changes in framing induced by Reidemeister I moves.

\item Khovanov-Rozansky's \textit{unframed} link invariant, defined in \cite{KR-matrix-factorizations}, is denoted $\unframedKhR_2$ in the lasagna literature. This is related to the previous two constructions by $\unframedKhR_2(L) \cong \newKh(L^!)$ and $\unframedKhR_2^{i,-j}(L) \cong \Kh^{i,-j}(L^!)$, where $L^!$ denotes the mirror of a link $L$. The skein relations are 
\begin{equation*}
    
    \begin{tikzpicture}[scale=.3, baseline=-.67ex]
        
\begin{scope}[xshift=-.5cm, yshift=-.5cm]
\draw[->] (0,0) -- (1,1);
\fill[white] (.5,.5) circle (.2cm);
\draw[->] (1,0) -- (0,1);
\end{scope}
    \end{tikzpicture}
 \ = \ 
    q\inv 
    \begin{tikzpicture}[scale=.3, baseline=-.67ex]

\begin{scope}[xshift=-.5cm, yshift=-.5cm]
\draw (0,0) .. controls (.5, .5).. (0,1);
\draw (1,0) .. controls (.5, .5).. (1,1);
\end{scope}
    \end{tikzpicture}
 \to h q^{-2} 
    \begin{tikzpicture}[scale=.3, baseline=-.67ex,rotate=-90]

\begin{scope}[xshift=-.5cm, yshift=-.5cm]
\draw (0,0) .. controls (.5, .5).. (0,1);
\draw (1,0) .. controls (.5, .5).. (1,1);
\end{scope}
    \end{tikzpicture}

    \qquad \qquad
    
    \begin{tikzpicture}[scale=.3, baseline=-.67ex]
        
\begin{scope}[xshift=-.5cm, yshift=-.5cm]
\draw[->] (1,0) -- (0,1);
\fill[white] (.5,.5) circle (.2cm);
\draw[->] (0,0) -- (1,1);
\end{scope} 
    \end{tikzpicture}
 \ = \ 
    h\inv q^2 
    \begin{tikzpicture}[scale=.3, baseline=-.67ex,rotate=-90]

\begin{scope}[xshift=-.5cm, yshift=-.5cm]
\draw (0,0) .. controls (.5, .5).. (0,1);
\draw (1,0) .. controls (.5, .5).. (1,1);
\end{scope}
    \end{tikzpicture}
 
    \to q 
    \begin{tikzpicture}[scale=.3, baseline=-.67ex]

\begin{scope}[xshift=-.5cm, yshift=-.5cm]
\draw (0,0) .. controls (.5, .5).. (0,1);
\draw (1,0) .. controls (.5, .5).. (1,1);
\end{scope}
    \end{tikzpicture}
.
\end{equation*}

\item Manolescu-Neithalath's cabled Khovanov homology uses a \textit{framed} version of Khovanov-Rozansky's invariant, and is denoted $\KhR_2$. The oriented skein relations are
\begin{equation*}
    
    \begin{tikzpicture}[scale=.3, baseline=-.67ex]
        
\begin{scope}[xshift=-.5cm, yshift=-.5cm]
\draw[->] (1,0) -- (0,1);
\fill[white] (.5,.5) circle (.2cm);
\draw[->] (0,0) -- (1,1);
\end{scope} 
    \end{tikzpicture}
 \ = \ 
    h\inv q 
    \begin{tikzpicture}[scale=.3, baseline=-.67ex,rotate=-90]

\begin{scope}[xshift=-.5cm, yshift=-.5cm]
\draw (0,0) .. controls (.5, .5).. (0,1);
\draw (1,0) .. controls (.5, .5).. (1,1);
\end{scope}
    \end{tikzpicture}
 \to 
    \begin{tikzpicture}[scale=.3, baseline=-.67ex]

\begin{scope}[xshift=-.5cm, yshift=-.5cm]
\draw (0,0) .. controls (.5, .5).. (0,1);
\draw (1,0) .. controls (.5, .5).. (1,1);
\end{scope}
    \end{tikzpicture}

    \qquad \qquad
    
    \begin{tikzpicture}[scale=.3, baseline=-.67ex]
        
\begin{scope}[xshift=-.5cm, yshift=-.5cm]
\draw[->] (0,0) -- (1,1);
\fill[white] (.5,.5) circle (.2cm);
\draw[->] (1,0) -- (0,1);
\end{scope}
    \end{tikzpicture}
 \ = \ 
    
    \begin{tikzpicture}[scale=.3, baseline=-.67ex]

\begin{scope}[xshift=-.5cm, yshift=-.5cm]
\draw (0,0) .. controls (.5, .5).. (0,1);
\draw (1,0) .. controls (.5, .5).. (1,1);
\end{scope}
    \end{tikzpicture}
 \to hq\inv 
    \begin{tikzpicture}[scale=.3, baseline=-.67ex,rotate=-90]

\begin{scope}[xshift=-.5cm, yshift=-.5cm]
\draw (0,0) .. controls (.5, .5).. (0,1);
\draw (1,0) .. controls (.5, .5).. (1,1);
\end{scope}
    \end{tikzpicture}

\end{equation*}
Let $\overline D$ denote the mirror of the diagram $D$. Then 
\begin{equation*}
    \CKhR_2(D) \cong 
    \newCKh(D^!)\{-w(D)\} 
    = \newCKh(D^!)\{w(D^!)\}
\end{equation*}
where $w(\cdot)$ denotes the writhe of a diagram.
In some papers, such as \cite{Hog-polyn-action}, $\KhR_2$ is computed first using an unoriented skein relation
\begin{equation}
\label{eq:unoriented-skein-relation}
    \KhRbrack{
    \begin{tikzpicture}[scale=.3, baseline=-.67ex]
        
\begin{scope}[xshift=-.5cm, yshift=-.5cm]
\draw (1,0) -- (0,1);
\fill[white] (.5,.5) circle (.2cm);
\draw (0,0) -- (1,1);
\end{scope}
    \end{tikzpicture}
}
     \ = \ 
     h\inv q 
    \begin{tikzpicture}[scale=.3, baseline=-.67ex,rotate=-90]

\begin{scope}[xshift=-.5cm, yshift=-.5cm]
\draw (0,0) .. controls (.5, .5).. (0,1);
\draw (1,0) .. controls (.5, .5).. (1,1);
\end{scope}
    \end{tikzpicture}
 \to 
    \begin{tikzpicture}[scale=.3, baseline=-.67ex]

\begin{scope}[xshift=-.5cm, yshift=-.5cm]
\draw (0,0) .. controls (.5, .5).. (0,1);
\draw (1,0) .. controls (.5, .5).. (1,1);
\end{scope}
    \end{tikzpicture}
.
\end{equation}
Note that this agrees with the skein relation for $
    \begin{tikzpicture}[scale=.3, baseline=-.67ex]
        
\begin{scope}[xshift=-.5cm, yshift=-.5cm]
\draw[->] (1,0) -- (0,1);
\fill[white] (.5,.5) circle (.2cm);
\draw[->] (0,0) -- (1,1);
\end{scope} 
    \end{tikzpicture}
$; in general, if $D$ contains $n_-$ negative crossings, we have
\[
    \CKhR_2(D) = \KhRbrack{D}[n_-]\{-n_-\}.
\]

\end{itemize}

\subsection{Conventions for Bar-Natan's cobordism categories}
\label{sec:BN-category-definition}

Here we recall some preliminary definitions about the cobordism categories associated to the categorification of the Temperley-Lieb algebra (\cite{BN-Kh}, \cite{BN-tangles}, and subsequent works) with the grading conventions used in the skein lasagna literature.

For $n \geq 0$, let $D^2_n$ denote the disk with a fixed set of $2n$ marked points $X_n \subset D^2_n$ on the boundary. A \emph{planar tangle} $T \subset D^2_n$ is a properly embedded $1$-manifold in $D^2_n$ with boundary $\partial T = X_n$. 

On the other hand, a \emph{tangle} in general may have crossings, and is to be regarded as properly embedded in $D^2_n \times (-\varepsilon, \varepsilon)$ with $X_n \subset \partial D^2_n \times \{0\}$. These will be represented using chain complexes built from the planar tangles above, which we discuss in the following sections. We use the same notation for the homological and quantum shift operators on tangles (e.g.\ $h^k q^\ell \BNbrack{T} = \BNbrack{T} [k]\{\ell\}$).

\begin{figure}[t]
    \begin{center}
        \includegraphics[width=.8\linewidth]{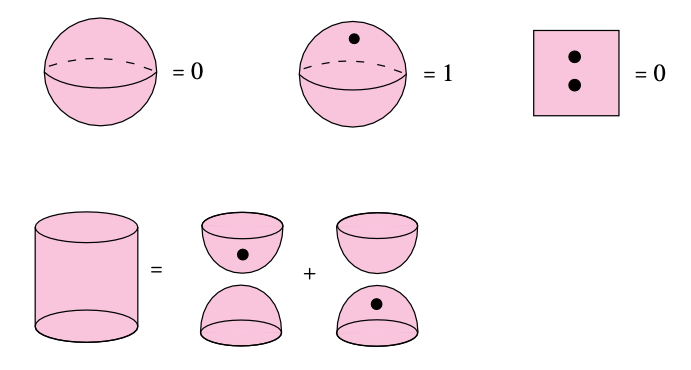}
    \end{center}
    \caption{Local relations in the Bar-Natan cobordism category. The bottom relation is called \emph{neck cutting}.}
    \label{fig:BNlocalrels}
\end{figure}

A \emph{(dotted) cobordism} $F:q^{i}T_{0}\rightarrow{q^{j}T_{1}}$ between (quantum-shifted) planar tangles $T_{0},T_{1}\subset{D_{n}^{2}}$ is a properly embedded surface $F \subset D^2_n \times [0,1]$ with boundary $\partial F = (T_0 \times \{0\}) \cup (T_1 \times \{1\}) \cup (X_n \times [0,1])$, possibly decorated with a finite number of dots. 

The \emph{quantum degree} of the cobordism $F$ is 
\begin{equation}
\label{eq:quantum-degree-of-cobordism}
    \deg_q(F) = n + j - i - \chi(F) + 2(\#\text{ of dots})
\end{equation}
where $\chi(F)$ is the Euler characteristic of the surface.

Furthermore, $\deg_q(F) = - \chi(F)$ for a closed surface $F$ without dots viewed as a cobordism from the unshifted $\emptyset$ to itself. The degree of a dot is $\deg_q(\bullet) = +2$. The category $\Cob_{n}$ is then defined as follows.

\begin{definition}
The objects $\Ob(\Cob_n)$ are formal shifts of planar tangles $T \subset D^2_n$. A morphism $f: q^i T_0 \to q^j T_1$ in $\Mor(\Cob_n)$ is a formal $\Z$-linear combination of dotted cobordisms, modulo isotopy rel boundary, movement of dots in the same connected component, and Bar-Natan's local relations, shown in Figure \ref{fig:BNlocalrels}. The morphisms of $\Cob_{n}$ are composed by vertical stacking. Occasionally, we require cobordisms categories of planar tangles with different numbers of specified endpoints. For planar tangles with $n$ bottom endpoints and $k$ top endpoints, the corresponding cobordism category is denoted $\Cob_{n,k}$.
\end{definition}

\begin{remark}\label{rmk:sqcupforcomplexes}
    Let $T_{1}$ and $T_{2}$ be tangle diagrams. We use the notation $\BNbrack{T_{1}}\sqcup{\BNbrack{T_{2}}}$ (resp. $\KhRbrack{T_{1}}\sqcup{\KhRbrack{T_{2}}}$) to denote the chain complex associated to the horizontal composition of tangles $T_{1}\sqcup{T_{2}}$.
\end{remark}


Let $\BNcat_n$ denote $\Mat(\Cob_n)$, and let $\Kom(\BNcat_n)$ denote the category of chain complexes over $\Mat(\Cob_n)$ where the morphisms are quantum degree $0$ chain maps, and where differentials have homological degree +1. 
Following \cite{Hog-polyn-action}, we generally drop the brackets, with the understanding that all instances of tangles should be interpreted as chain complexes in $\Kom(\TLcat_{n})$, defined below.

In the following sections we will often want to consider \emph{$(n,n)$ planar tangles}, or planar tangles in $D^2_n \cong [0,1] \times [0,1]$ where the boundary points $X_n$ are split into two sets, with $n$ each (equally spaced, say) along $\{0\} \times [0,1]$ and $\{1\} \times [0,1]$. In this case, we write $\TLcat_n$ in place of $\Mat(\Cob_n)$, and write $\Kom(\TLcat_n)$ for the category of chain complexes and degree-preserving chain maps. If we instead work with tangles with different numbers of top and bottom boundary points, we write $\TLcat_{n}^{k}$ instead.

Given two $(n,n)$ planar tangles $T, T'$, stacking $T'$ on top of $T$ gives a composition operation, forming the new planar tangle $T' \otimes T$. 
This composition induces a composition operation in $\TLcat_n$ and $\Kom(\TLcat_{n})$. 

An \emph{$(n,n)$ tangle}, which may contain crossings, is regarded as properly embedded in $D^2_n \times (-\varepsilon, \varepsilon) \cong [0,1]^2 \times (-\varepsilon, \varepsilon)$ with the marked points along  $\{0\} \times [0,1] \times \{0\}$ and $\{1\} \times [0,1] \times \{0\}$.

We now compile some of the techniques fundamental to the computation of Khovanov homology using Bar-Natan's local techniques. The \emph{delooping} operation, depicted in the following figure, describes an isomorphism in $\TLcat_{n}$ between an object with a closed loop and the same object with the closed loop removed. This operation is used to remove disjoint circles from diagrams.

\begin{center}
\begin{tikzpicture}[scale=.5]
\draw (0,0) circle (1cm);
\draw[->] (1,1) -- (4,2);
\begin{scope}[xshift=2cm, yshift=2cm, scale=.5]
    \kap
\end{scope}
\begin{scope}[yscale=-1]
    \draw[->] (1,1) -- (4,2);
    \begin{scope}[xshift=2cm, yshift=2cm, scale=.5]
        \kapdot
    \end{scope}
\end{scope}
\node (empty1) at (6,2) {$q \cdot \emptyset$};
\node (empty2) at (6,-2) {$q^{-1} \cdot \emptyset$};
\node (oplus) at (6,0) {$\oplus$};
\draw[->] (8,2) -- (11,1);
\begin{scope}[xshift=10cm, yshift=2cm, scale=.5]
    \kupdot
\end{scope}
\begin{scope}[yscale=-1]
    \draw[thick, ->] (8,2) -- (11,1);
    \begin{scope}[xshift=10cm, yshift=2cm, scale=.5]
        \kup
    \end{scope}
\end{scope}
\draw (12,0) circle (1cm);
\end{tikzpicture}
\end{center}

This operation is used in conjunction with Gaussian elimination:

\begin{lemma}[\cite{FastKh}, Lemma 4.2] (Gaussian Elimination) 
    Let $C\in{\Kom(\mathcal{A})}$ be a chain complex over an additive category $\mathcal{A}$, and suppose that $C$ contains the subcomplex
    \[
    A\xrightarrow{\begin{bmatrix}
    0 \\
    f
    \end{bmatrix}} \begin{tabular}{c}
        $B$ \\
        $\oplus$\\
        $C$
    \end{tabular}\xrightarrow{\begin{bmatrix}
        \Phi & \alpha\\
        \beta & \gamma
    \end{bmatrix}}{\begin{tabular}{c}
        $D$ \\
        $\oplus$\\
        $E$
    \end{tabular}
    }\xrightarrow{\begin{bmatrix}
        0 & g
    \end{bmatrix}}{F}
    \]
    where $\Phi:B\rightarrow{D}$ is an isomorphism. Then complex $C$ is chain homotopy equivalent to the complex $C^{\prime}\in{\Kom(\mathcal{A})}$ with the above portion of the complex replaced by
    \[
    A\xrightarrow{f}{C}\xrightarrow{\gamma-\beta\Phi^{-1}\alpha}{E}\xrightarrow{g}{F}
    .
    \]
\end{lemma}

By delooping and the unoriented skein relation \eqref{eq:unoriented-skein-relation}, we have the following Reidemeister I chain homotopy equivalences {for $\KhR_2$}:
\[
\begin{tikzpicture}[scale=.25, baseline=-.67ex, xscale=-1, rotate=-90]
\input{tikz-imgs/R1-neg-input}
\end{tikzpicture}
$ =\  h\inv q$
\begin{tikzpicture}[scale=.25, baseline=-.67ex]
\input{tikz-imgs/R1-neg-1res-input}
\end{tikzpicture}
$\to$
\begin{tikzpicture}[scale=.25, baseline=-.67ex]
\input{tikz-imgs/R1-neg-0res-input}
\end{tikzpicture}
$\simeq h^0q\inv$ 
\begin{tikzpicture}[scale=.25, baseline=-.67ex]
\input{tikz-imgs/R1-neg-reduced-input}
\end{tikzpicture} 
\qquad
\begin{tikzpicture}[scale=.25, baseline=-.67ex]
\input{tikz-imgs/R1-neg-input}
\end{tikzpicture}
$ =\  h\inv q$
\begin{tikzpicture}[scale=.25, baseline=-.67ex]
\input{tikz-imgs/R1-neg-0res-input}
\end{tikzpicture}
$\to$
\begin{tikzpicture}[scale=.25, baseline=-.67ex]
\input{tikz-imgs/R1-neg-1res-input}
\end{tikzpicture}
$\simeq h\inv q^2$ 
\begin{tikzpicture}[scale=.25, baseline=-.67ex]
\input{tikz-imgs/R1-neg-reduced-input}
\end{tikzpicture}. 
\]

For the chain maps associated to the other Reidemeister moves, we use the conventions set in \cite[Section 3.3]{MWW-lasagna}.

\subsection{Categorified projectors and the resolution of identity braids}
\label{sec:projectors}

The categorified Jones-Wenzl projectors (see \cite{Roz-Pn, CoopKrush}) are very useful objects in $\Kom(\TLcat_{n})$. Letting $FT_{n}^{\otimes m}$ denote the $n$-strand braid with $m$ positive twists, the Jones-Wenzl projector is given by $P_{n}:=\colim_{m\rightarrow{\infty}}FT_{n}^{\otimes m}$ with degree 0 connecting maps. Throughout, we denote these projectors by boxes in tangle diagrams
\begin{center}
    \tikzset{every picture/.style={line width=0.75pt}} 

\begin{tikzpicture}[x=0.75pt,y=0.75pt,yscale=-1,xscale=1]

\draw   (286.81,70.75) -- (354.18,70.75) -- (354.18,92.59) -- (286.81,92.59) -- cycle ;
\draw    (293.77,41.52) -- (293.87,70.83) ;
\draw    (303.15,41.36) -- (303.24,70.67) ;
\draw    (346.64,41.36) -- (346.74,70.67) ;
\draw [line width=1.5]  [dash pattern={on 1.69pt off 2.76pt}]  (319.1,62.62) -- (335.24,62.7) ;
\draw    (293.97,92.75) -- (294.06,125.95) ;
\draw    (303.34,92.59) -- (303.44,125.79) ;
\draw    (346.84,92.59) -- (346.93,125.79) ;
\draw [line width=1.5]  [dash pattern={on 1.69pt off 2.76pt}]  (317.5,114.62) -- (333.64,114.7) ;

\draw (312.22,75.57) node [anchor=north west][inner sep=0.75pt]  [font=\large] [align=left] {};

\end{tikzpicture}.
\end{center}

The categorified Jones-Wenzl projectors enjoy the following properties.

\begin{lemma}
\label{lem: projector eats crossings} Let $P_{n}$ be a categorified Jones-Wenzl projector on $n$ strands.
    \begin{enumerate}
        \item The $n$-strand identity braid appears in $P_{n}$ only once in homological degree $0$.\\
        \item $P_{n}$ is homologically bounded above.\\
        \item $P_{n}$ kills turnbacks and is idempotent: $P_{n}\otimes{P_{n}} \simeq {P_{n}}$.\\
        \item $P_n$ eats crossings: for any Artin generator $\sigma_i$ in the braid group $B_n$, we have 
        \[
            P_n \otimes \sigma_i \simeq P_n \simeq P_n \otimes \sigma_i\inv.
        \]
        Note that there are no grading shifts with our conventions.
    \end{enumerate}
\begin{proof}
    For a proof of these statements, see \cite[Theorem 3.7]{Hog-polyn-action} and \cite{CoopKrush}.
\end{proof}
\end{lemma}

Along with categorified Jones-Wenzl projectors $P_{n}$, we will employ a related family of chain complexes in $\Kom(\TLcat_{n})$ called \emph{higher order projectors}. 

\begin{definition}
    Let $D$ be a Temperley-Lieb diagram in $\TLcat_{n}$. The \emph{through-degree} of $T$, denoted $\tau(T)$, is the minimal integer $k$ such that $D$ factors into a vertical stacking $D_{1}\otimes{D_{2}}$, where $D_{1}\in{\TLcat_{k}^{n}}$ and $D_{2}\in{\TLcat_{n}^{k}}$.
\end{definition}

\begin{definition}
    Let $\tau$ denote the \emph{through-degree} of a chain complex of Temperley-Lieb diagrams in $\Kom(\TLcat_{n})$, meaning, for $A\in{\Kom(\TLcat_{n})}$, define $\tau(A)$ to be 
        $\max \{\tau(D) \ | \ D \text{ a Temperley-Lieb diagram appearing in }  A\}$.
\end{definition}

\begin{definition}\cite[Definition 8.4]{Cooper-Hog-family}
    The \emph{$k$th higher order projector} is a chain complex $P_{n,k}\in{\Kom(\TLcat_{n})}$ uniquely defined by the following properties
    \begin{enumerate}
        \item $\tau(P_{n,k})=k$.\\
        \item For each $l\in{\mathbb{Z}_{+}}$, and $a\in{\Cob_{n,l}}$, if $\tau(a)<k$, then $a\otimes{P_{n,k}}\simeq{0}$. ($P_{n,k}$ kills complexes with sufficiently low through-degree.)\\
        \item There exists $C\in{\Kom(\TLcat_{n})}$ with $\tau(C)<k$, and a twisted complex 
        \[
            D=\one_{n}\rightarrow{C}\rightarrow{hP_{n,k}}
        \]
        such that $a\otimes{D}\simeq{D\otimes{\overline{a}}}\simeq{0}$ for all $a\in{\Cob_{n,m}}$ such that $\tau(a)\leq{k}$.
    \end{enumerate}
    We call the higher order projector $P_{n,0}$ of through-degree $0$ the \emph{Rozansky projector} on $n$ strands in $\Kom(\TLcat_{n})$ (see \cite{ROZ-Cat,WillisS1xS2}).
\end{definition}

Note that the higher order projector $P_{n,k}$ factors through $P_{k}$. More precisely, restating Observation $8.8$ of Cooper-Hogancamp in \cite{Cooper-Hog-family}, given $A$, $B$ in $\Kom(\TLcat_{n})$ such that $\tau(A)\geq{k}$ and $\tau(B)\geq{k}$, there is an isomorphism $P_{n,k}\cong{A\otimes{P_{k}}\otimes{B}}$. In the decategorified setting, for some idempotents $p_{\epsilon}$ in the Temperley-Lieb algebra $TL_{n}$, there is a decomposition of the identity:
\begin{equation}\label{eq:decatresofid}
1=\sum p_{\epsilon}.
\end{equation}
Roughly, the categorification of \eqref{eq:decatresofid} may be realized as the following chain homotopy equivalence.

\begin{equation}\label{eq:resofidentity}
    \one_{n}\simeq{\left(P^{\vee}_{n,n(\text{mod 2})}\rightarrow{P^{\vee}_{n,n(\text{mod 2})+2}\rightarrow{\cdots}}\rightarrow{P^{\vee}_{n,n-2}}\rightarrow{P^{\vee}_{n}}\right)}
\end{equation}

Where the right-hand side has higher differentials $P^{\vee}_{n,i}\rightarrow{P^{\vee}_{n,j}}$ $(j>i)$. This homotopy equivalence, referred to as the \emph{resolution of the identity} \cite[Section 7, Observation 8.9]{Cooper-Hog-family}, will prove to be an instrumental tool in the arguments to follow.

\section{Cabled Khovanov homology and skein lasagna modules}\label{sec:lasagnabackground}

\subsection{Skein lasagna modules}\label{subsec:skeinlasagnamodules}

Morrison-Walker-Wedrich \cite{MWW-lasagna} define an invariant $\skein^{2}_{0}$ of a pair $(W,L\subset{\partial{W}})$, where $W$ is an oriented $4$-manifold, and $L$ a link in its boundary. For a null-homologous boundary link $L$ ($[L]=0\in{H_{1}(W;\mathbb{Z})}$), the invariant $\skein^{2}_{0}$ is a triply-graded module, with trigrading $(\alpha,i,j)$ in $H_{2}^{L}(W)\times{\Z}\times{\Z}$. The $H^{L}_{2}(W)$ term is the $H_{2}(W)$-torsor, defined as $\partial^{-1}([L])\subset{H_{2}(W;L)}$, where $\partial$ is the boundary map in the long exact sequence of the pair $(W,L)$. Note that the homological level may be taken to be a (non-canonical) element of $H_{2}(W;\Z)$. The gradings $i$ and $j$ are the homological and quantum gradings from $\KhR_{2}$ respectively, and the grading in $H^{L}_{2}(W)$ is referred to as the \emph{homological level} of $\skein_{0}^{2}$. The modules $\skein^{2}_{0}$ are generated by \emph{lasagna fillings}, which are defined as follows in Figure \ref{fig:lazfilling}.

\begin{figure}[t]
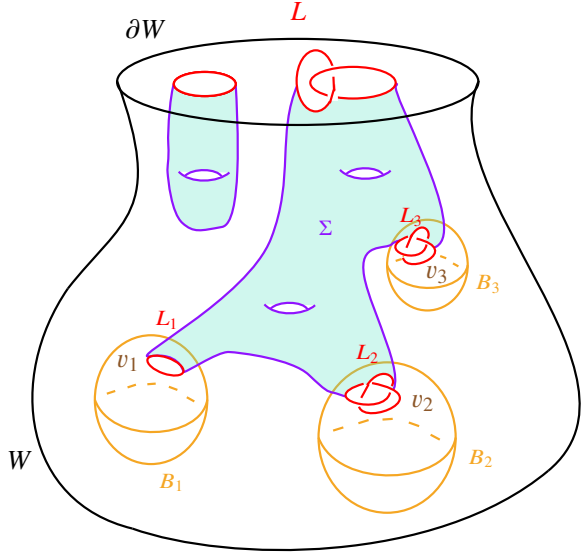

    \begin{center}
        \includestandalone[scale=1.1]{tikz-imgs/lasagnafilling-1}
    \end{center}
     \caption{A lasagna filling $\mathcal{F}$ of the pair $(W,L)$.}
     \label{fig:lazfilling}
\end{figure}

\begin{definition}
    A \emph{lasagna filling} of $(W,L\subset{\partial{W}})$ is an object consisting of the following data: $\mathcal{F}:=(\Sigma,\{(B_{i}, L_{i}, v_{i})\})$, where:
    \begin{itemize}
        \item{A finite set of input balls ("\textit{meatballs}") $\{B_{i}\}$, disjointly embedded in $W$, with a link $L_{i}\subset{\partial{B_{i}}}$, and a homogeneous label $v_{i}\in{\KhR_{2}(L_{i})}$}.\\
        \item{A framed oriented surface $\Sigma$ properly embedded in $W\setminus{\bigcup_{i}B_{i}}$ such that $\Sigma\cap{B_{i}}=L_{i}$, and $\Sigma\cap{\partial{W}}=L$.}
    \end{itemize}
\end{definition}

There is a well-defined bidegree for fillings $\mathcal{F}$ of a pair $(W,L)$:

\begin{definition}
    The \emph{bidegree of a lasagna filling} $\mathcal{F}$ is given by:
    \[
    \mathrm{deg}(\mathcal{F}):=\sum_{i}\mathrm{deg}(v_{i})+(0,-\chi(\Sigma))
    \]
\end{definition}

Furthermore, when $W=B^{4}$, we define $\KhR_{2}(\mathcal{F}):=\KhR_{2}(\Sigma)(\otimes_{i}x_{i})\in{\KhR_{2}(\partial{W};L)}$, where $\KhR_{2}(\Sigma)$ is the morphism induced by $\Sigma$ of $\mathcal{F}$.

\begin{definition}\label{def:skeinlasagnamodule}
    \normalfont
    For a $4$-manifold $W$ and a link $L\subset{\partial{W}}$, the \emph{skein lasagna module} of $(W;L)$ is the bigraded abelian group:
    \[
    \skein^{2}_{0}(W;L):=\mathbb{F}\{{\mathcal{F}\text{ of }(W,L)} \}/\sim
    \]
\end{definition}

The relation is defined as the transitive and linear closure of the following relations.

\begin{itemize}
    \item Linear combinations of lasagna fillings are multilinear in the $\KhR_{N}$ labels $\{v_{i}\}$.
    \item Two lasagna fillings $\mathcal{F}_{1}$ and $\mathcal{F}_{2}$ are equivalent if $\mathcal{F}_{1}$ has an input ball $B_{i}$ with boundary link $L_{i}$ labelled $v_{i}$, and the filling $\mathcal{F}_{2}$ is obtain from $\mathcal{F}_{1}$ by inserting a lasagna filling $\mathcal{F}_{3}$ of $(B_{i},L_{i})$ into $B_{i}$ such that $v_{i}=\KhR_{2}(\mathcal{F}_{3})$, possibly followed by an isotopy rel boundary (see Figure \ref{fig:lasagnarelation}).
\end{itemize}

\begin{figure}[t]
    \centering
    \includegraphics[width=\linewidth]{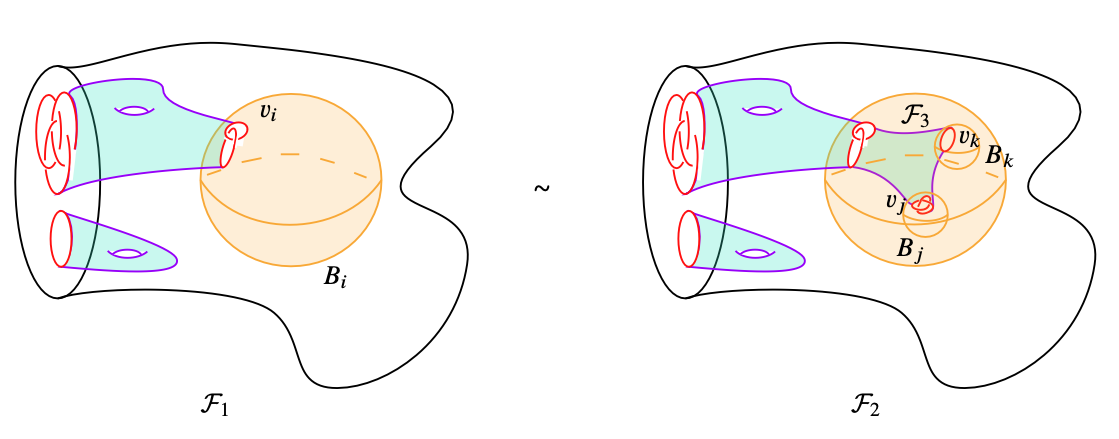}
    \caption{The second relation of Definition \ref{def:skeinlasagnamodule} on lasagna fillings.}
    \label{fig:lasagnarelation}
\end{figure}


Furthermore, by \cite{MN22} Proposition 1.6, the skein lasagna module's isomorphism class remains unchanged after the removal of a $4$-ball. In particular, if $W$ is a closed, smooth $4$-manifold, there is an isomorphism $\skein_{0}^{2}(W;\emptyset)\cong{\skein_{0}^{2}(W\setminus{B^{4}};\emptyset)}$.

\subsection{Cabling and Khovanov homology}\label{subsec:cabledkh}
We now recall Manolescu-Neithalath's \emph{$2$-handlebody formula} \cite[Theorem 1.1]{MN22}.
Let $L\subset{S^{3}}$ be a framed oriented link with components $L_{1},L_{2},...,L_{k}$, and let $r^{-}$, $r^{+}\in{\Z_{\geq{0}}^{k}}$.

\begin{definition}
    The \emph{$(r^{-},r^{+})$-cable of $L$}, denoted $L(r^{-},r^{+})$, is the framed oriented link consisting of $r^{-}_{i}$ many negatively oriented parallel strands and $r^{+}_{i}$ many positively oriented parallel strands for the component $L_{i}$. The notion of parallel is given by pushing off along the framing of each $L_{i}$, and `positively' (resp. negatively) oriented means the orientation of the parallel strand agrees (resp. disagrees) with the orientation of the component $L_{i}$.
\end{definition}

\begin{figure}[t]
    \begin{center}
        \includegraphics[width=1.5in]{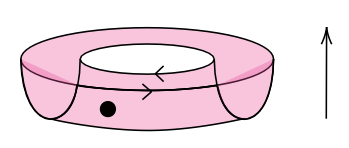}
    \end{center}
    \caption{The dotted annulus cobordism, denoted by $\robacupdot$ throughout.}
    \label{fig:dottedcakepan}
\end{figure}

Let $\acupdot$ denote the dotted annulus cobordism from the empty link to two oppositely oriented parallel strands in Figure \ref{fig:dottedcakepan}. Let $e_{i}$ denote the $i$th unit vector. Note that $\acupdot$ is a cobordism between cables of $L$;
\begin{equation}\label{eq:dottedcup}
    \acupdot:L(r^{-},r^{+})\rightarrow L(r^{-}+e_{i},r^{+}+e_{i}).
\end{equation}

\begin{definition}\label{def:cabledkh}
Let $\mathfrak{S}_{n}$ denote the symmetric group on $n$ elements and, for $\alpha\in{\mathbb{Z}^{n}}$, let $\alpha^{+}$ (resp. $\alpha^{-}$) denote the tuple $(\alpha_{1}^{+},...,\alpha_{n}^{+})$ where $\alpha_{i}^{+}:=\max{\{0,\alpha_{i}\}}$ (resp. $\alpha_{i}^{-}:=\min{\{0,\alpha_{i}\}}$).
The \emph{cabled Khovanov homology} over $\mathbb{F}$ of a framed oriented link $L$ with $n$ components at homological level $\alpha\in{\mathbb{Z}^{n}}$ is defined as

\[
\underline{\KhR}_{2,\alpha}(L)=\left(\bigoplus_{r\in{\mathbb{Z}_{>0}^{n}}}\KhR_{2}(L(r-\alpha^{-},r+\alpha^{+}))\{-2|r|-|\alpha|\}\right)/\sim
\]
where $\sim$ is the transitive and linear closure of the following identifications:
\[
\beta(b)(v)\sim{v}, \qquad \KhR_{2}(\acupdot)(v)\sim{v}
\]
for all $b\in{\mathfrak{S}_{2r_{i}+|\alpha_{i}|}}$ and for all $v\in{\KhR_{2}(L(r-\alpha^{-},r+\alpha^{+}))}$ where
\begin{enumerate}
    \item \label{item:braidaction} For $b$ an element of the braid group $B_{r_{i}-\alpha_{i}^{-}+r_{i}+\alpha_{i}^{+}}$, $\beta_{i}(b)$ is the automorphism induced on $\KhR_{2}(L_{i}(r_{i}-\alpha_{i}^{-},r_{i}+\alpha_{i}^{+}))$ by the braid group action interchanging parallel strands. By \cite{GLW-schur-weyl}, this braid group action on cables factors through the symmetric group.
    \item $\KhR_{2}(\acupdot)$ denotes the morphism induced by the dotted annulus cobordism $\acupdot$ (see Figure \ref{fig:dottedcakepan}).
\end{enumerate}

\end{definition}

Note that the undotted annulus relation is omitted from our definition as in \cite[Proposition 3.8]{MN22}. We present an equivalent definition of $\underline{\KhR}_{2}$ tailored to $4$-manifold and boundary link pairs $(W,L)$, which we will use in this article. Let $\symmetrized{\KhR_{2}(L)}$ denote the vector space $\KhR_{2}(L)$ symmetrized with respect to the braid group action in part (1) of Definition \ref{def:cabledkh}.
If $f$ is a linear map between vector spaces, let $\symmetrized{f}$ denote the induced map on symmetrized vector spaces; see Section \ref{subsec:kirbybeltconstruct} for more details.

\begin{definition}\label{def:CDS}
Let $W$ be a $4$-manifold with a $0$-handle, $k$ many $2$-handles, and possibly a $4$-handle, with $2$-handles attached along a framed oriented link $L=L_{1}\cup{L_{2}}\cup{...}\cup{L_{k}}$. Let $(I,\leq)$ be the directed set $\mathbb{Z}_{\geq{0}}^{k}$ with the poset relation induced by the total ordering $\leq$ on $\mathbb{Z}$; this forms a poset category. The \emph{cabling directed system for $W$ at homological level $\alpha$}, denoted $\mathcal{B}^{\alpha}(W;\emptyset)$, is
{a functor from  $\Z^k_{\geq 0}$ to $\ggVect_\F$ where}

\begin{itemize}
    \item for $a \in I$, $\mathcal{B}^{\alpha}(W;\emptyset)(a):=\symmetrized{\KhR_{2}(L_{a})}$, where $L_{a}:=L(a-\alpha^{-},a+\alpha^{+})$ and 
    \item the arrow from $\mathcal{B}^{\alpha}(W;\emptyset)(a)$ to $\mathcal{B}^{\alpha}(W;\emptyset)(a+e_{i})$ is $\symmetrized{\acupdot}$ at the corresponding $2$-handle attachment site $L_{i}\subset{L}$ as described above.
\end{itemize}

\end{definition}

The \emph{cabled Khovanov homology of $L$} at homological level $\alpha$ may then be defined as the colimit of the cabling system $\mathcal{B}^{\alpha}(W;\emptyset)$. The above construction for the framed oriented link in a $2$-handlebody Kirby diagram is isomorphic to the skein lasagna module of the pair $(W,\emptyset)$, where $W$ is the manifold described by said Kirby diagram. In particular, we have the following $2$-handle formula:

\begin{theorem}[\cite{MN22}, Theorem 1.1]
    Let $W$ be the $4$-manifold obtained by attaching $2$-handles to $B^{4}$ along an oriented, framed $n$-component link $L$. For each $\alpha\in{H_{2}(W;\Z)\cong{\mathbb{Z}^{n}}}$, there is an isomorphism
    \begin{equation}
        \Phi:\colim_{\ggVect}(\mathcal{B}^{\alpha}(W;\emptyset))\overset{\cong}{\longrightarrow}{\skein_{0}^{2}(W;\emptyset,\alpha)}.
    \end{equation}
\end{theorem}

It will be useful to have a relative version of the construction above for a pair $(W,L)$ with a nontrivial, null-homologous boundary link. Specifically, let $L^{\text{att}}$ denote the framed link that the $2$-handles of a $2$-handlebody $B^{4}(L^{\text{att}})$ are attached along, and let $L$ denote a link in $\partial{(B^{4}(L^{\text{att}}))}$ such that $[L]=0\in{H_{1}(B^{4}(L^{\text{att}}))}$. Recall that we may always isotope the boundary link $L$ away from the attaching regions of the $2$-handles. The authors of \cite{MWWhandles} describe such a cabled Khovanov homology construction for this setup by considering cables of the form $L^{\text{att}}(r^{-},r^{+})\cup{L}$ as follows. Note that the braid group action defined above yields a braid group action $\beta_{i}:B_{r_{i}^{-},r_{i}^{+}}\rightarrow{\mathrm{Aut}{(\KhR_{2}(L^{\text{att}}(r^{-},r^{+})\cup{L})})}$, and similarly for the dotted annulus map we have an induced map 
\[ 
    \KhR_{2}(\acupdot):\KhR_{2}(L^{\text{att}}(r^{-},r^{+})\cup{L})\rightarrow{\KhR_{2}(L^{\text{att}}(r^{-}+e_{i},r^{+}+e_{i})\cup{L})}.
\]

\begin{definition}
    Let $L^{\text{att}}$ and $L$ be the $2$-handle attaching link and boundary link respectively, where $L^{\text{att}}$ has $k$ components. Then the \emph{relative cabled Khovanov homology} is defined as
    \[
    \underline{\KhR}_{2,\alpha}(L^{\text{att}},L):=\left(\bigoplus_{r\in{\mathbb{Z}_{\geq{0}}^{k}}}\KhR_{2}(L^{\text{att}}(r-\alpha^{-},r+\alpha^{+})\cup{L})\{-2|r|+|\alpha|\}\right)/\sim
    \]
    where the relation $\sim$ is the same as the relation in Definition \ref{def:cabledkh}. For the equivalent cabling directed systems definition, let $L_{a}:=L^{\text{att}}(a-\alpha^{-},a+\alpha^{+})\cup{L}$ for each $a\in{I}$. We then obtain a new directed system $\mathcal{A}^{\alpha}(W;L)$ whose colimit is identically $\underline{\KhR}_{2,\alpha}(L^{\text{att}},L)$.
\end{definition}

\begin{remark}\label{rmk:cabledlasagna}
    The above definition agrees with the \textit{cabled skein lasagna module} construction in \cite{MWWhandles} with $W=B^{4}$. that is, $\underline{\KhR}_{2,\alpha}(L^{\text{att}},L)\cong{\skein_{0}^{2}(B^{4}(L^{\text{att}});L,\alpha)}$. 
\end{remark}

\subsection{Construction of the Kirby-colored belt around $n$ strands}
\label{subsec:kirbybeltconstruct}

The setting for the content of Sections \ref{subsec:skeinlasagnamodules} and \ref{subsec:cabledkh} is a chain complex category over bigraded vector spaces. In our approach, we instead work with tangles and cobordisms in (completions of) Bar-Natan's categories. In this setting, we postpone closing up tangles and taking homology until after we compute (homotopy) colimits. The relationship between this approach and the method used in \cite{MN22} is discussed in Section \ref{sec:mainsec}. In this subsection, we construct our primary objects of study. Let $\one_n$ denote the identity braid on $n$ strands, and let $T_n$ denote the unoriented chain complex associated to the "identity tangle wearing a belt":
\[
    \begin{tikzpicture}[yscale=.5, scale=.25]
%
\draw (0,0) circle (3cm);
%
\def\g{.3};
\foreach \x in {-2,0,2}{
    \fill[white](\x-\g,0) rectangle (\x+\g,6); 
    \draw (\x,0) -- (\x,6);
}
\foreach \x in {-2,2}{
    \draw (\x,0) -- (\x,-1.5);
    \draw (\x,-3) -- (\x,-6);
}
\draw (0,0) -- (0,-2);
\draw (0,-4) -- (0,-6);
\foreach \y in {5,0,-5}{
    \node at (1.1,\y) {$\cdots$};
}
\node at (0,7) {$\overbrace{\phantom{longword}}^n$};
\end{tikzpicture} 
\]
Observe that $T_n^{\otimes k}$ denotes the identity braid on $n$ strands wearing $k$ parallel, unlinked belts. We now describe the action of the symmetric group $\mathfrak{S}_{k}$ on the chain complex $T_{n}^{\otimes{k}}$. There are two ribbon maps\footnote{also referred to as "cake pans" or "Bundt cake pans"}
\[
    \acup: \one_n \to T_n^{\otimes 2} 
    \qquad
    \acap: T_n^{\otimes 2} \to \one_n.
\]
The reader should note that these are shorthand symbols for cobordisms; for example, $\acup$ represents a cobordism that is topologically $\acup \times S^1$. We will sometimes also compose these maps; for example, $\acap \circ \acup$ is a torus (wrapped around the identity cobordism) and therefore represents the morphism $\one_n \map{2} \one_n$. This follows directly from the local relations in Figure \ref{fig:BNlocalrels}. We will also use dotted ribbon cobordisms, in which case we use the symbols $\acupdot,\acapdot$, as seen in \eqref{eq:dottedcup} and Figure \ref{fig:dottedcakepan}. We write the composition $\acup \circ \acap$ as $\capcup$. 


Consider the braid group action on the $k$ belt loops in $T_n^{\otimes k}$. Let $\sigma_i \in B_k$, and let $\Sigma_i: T_n^{\otimes k} \to T_n^{\otimes k}$ denote the cobordism corresponding to the movie where the $(i+1)$st belt grows wider and moves up and around the $i$th belt, interchanging them; the cobordism looks like $\sigma_i \otimes S^1$ near the belts, along with $n$ identity sheets corresponding to the $n$ vertical strands. Let $\Sigma_i\inv$ denote the upside-down cobordism (i.e.\ time-reversed movie). 
Grigsby-Licata-Wehrli in \cite{GLW-schur-weyl} show that the cobordisms $\{\Sigma_i\}_{i=1}^k$ satisfy the braid relations on the nose. Furthermore, they show the braid group $B_k$ action descends to an action by the symmetric group $\symgroup{k}$, which we now describe. 

Define the \emph{swap} endomorphism $s: T_n^{\otimes 2} \to T_n^{\otimes 2}$ by
\[
    s = \idmap - \capcup
\]
on the two belt loops.
Since $(\capcup)^2 = 2\:\capcup$ by the local relations, we have $s^2 = \idmap$. 
Define $s_i$ to be the corresponding swap endomorphism involving only the $i$th and $(i+1)$st belts: 
\[
    s_i = \idmap_{i-1} \otimes s \otimes \idmap_{k-(i+1)}
\]
Note that in Grigsby-Licata-Wehrli's conventions, a torus evaluates to -2; since our torus evaluates to +2, the corresponding statement of \cite[Proposition 9]{GLW-schur-weyl} is 
\begin{equation}
    \Sigma_i = s_i = \Sigma_i\inv.
\end{equation}
Thus $s_i$ is the morphism realizing the transposition of the $i$th and $(i+1)$st belts under the $\symgroup{k}$ action. 

In order to symmetrize the complex $P_n \otimes T_n^{\otimes k}$ under the $\symgroup{k}$ action, we consider the morphism 
\[
    e_k := \frac{1}{k!} \sum_{g \in \symgroup{k}} g \quad \in \quad \End(T_n^{\otimes k}).
\]

\begin{definition}
    Let $\CC$ be a dg-category, a \emph{homotopy idempotent} $e\in{\End{(X)}}$ is a closed degree 0 endomorphism of some object $X$ such that $e^{2}\sim{e}$. Equivalently, it is an idempotent in the homotopy category of $\CC$. Furthermore, an object $X$ is an \emph{image} of a homotopy idempotent $e$ if there exist closed degree 0 maps $f:X\rightarrow{Y}$ and $g:Y\rightarrow{X}$ such that $f\circ{g}\sim{e}$ and $g\circ{f}\sim{\Id_{X}}$. (For more details, see \cite[Section 4]{GHWSoergel}.)
\end{definition}

By standard representation theory arguments, we have the following lemma.

\begin{lemma}
The endomorphism $e_k: T_n^{\otimes k} \to T_n^{\otimes k}$
is a homotopy idempotent, i.e.\ $e_k^2 \sim e_k$. 
\begin{proof}
The argument is standard and follows from the fact that left multiplication by any fixed $g \in \symgroup{k}$ gives a permutation of the set $\symgroup{k}$. 
Finally, note that we only have a homotopy equivalence between $e_k^2$ and $e_k$ because the action of the generators $s_i$ is defined only up to homotopy; the Reidemeister move equivalences are homotopy equivalences, not isomorphisms of complexes.
\end{proof}
\end{lemma}

Gorsky-Hogancamp-Wedrich in \cite{GHWSoergel} prove that the homotopy category $K(\CC)$ of a Karoubian category $\CC$ is also Karoubian, and the images of homotopy idempotents are unique up to homotopy. By \cite{Gorsky-Wedrich} Theorem A.10, bounded homotopy categories of Karoubian categories are Karoubian, so the images of homotopy idempotents are guaranteed to exist.

\begin{definition}\label{def:symmetrizedobjects}
Let $\symmetrized{T_n^{\otimes k}}$ denote an image of $T_{n}^{\otimes{k}}$ under the idempotent $e_k$ in $K(\Kar(\Kom(\TLcat_{n})))$. 
There exist maps 
\begin{center}
\begin{tikzcd}
    T_n^{\otimes k}
        \arrow[r, bend left, "p_k"]
    & \symmetrized{T_n^{\otimes k}}
        \arrow[l, bend left, "i_k"] 
\end{tikzcd}
\end{center}
such that 
\begin{equation}
    \label{eq:homotopy-projector}
    i_k \circ p_k \sim e_k 
    \qquad 
    \text{and}\qquad
    p_k \circ i_k \sim \idmap_{\symmetrized{T_n^{\otimes k}}}.
\end{equation}

For a morphism $f: T_n^{\otimes k} \to T_n^{\otimes l}$, let 
$\symmetrized{f} := p_l \circ f \circ i_k$
denote the induced morphism $\symmetrized{T_n^{\otimes k}} \to \symmetrized{T_n^{\otimes l}}$. 
\end{definition}

We verify that the map induced by the undotted annulus is identically $0$ on symmetrized $T_{n}^{\otimes{k}}$ complexes.

\begin{lemma}
Let $k \in \N$. Let $\acup: T_n^{\otimes k} \to T_n^{\otimes k+2}$ be the ribbon map that introduces the last pair of belts and is the identity sheet on all other components. Then $\symmetrized{\acup} \simeq 0$.
\begin{proof}
By \eqref{eq:homotopy-projector}, 
\begin{equation*}
\symmetrized{\acup} 
    = p_{k+2} \circ \acup \circ i_k 
    = p_{k+2} \circ i_{k+2} \circ p_{k+2} \circ \acup \circ i_k 
    = p_{k+2} \circ (e_{k+2} \circ \acup) \circ i_k.
\end{equation*}
So, it suffices to show that $e_{k+2} \circ \acup=0$. Note that, letting $s_{k}$ denote the swap endomorphism on the $k$th and $(k+1)$th strands, $e_{k+2}$ is the sum of all compositions of $s_{j}$, $j\in{\{1,...,k+1\}}$. Note that $s_{k+1}\circ{\acup}=-\acup$, and also that the set of permutations in $\mathfrak{S}_{k+2}$ can be decomposted into pairs $(g,g\circ s_{k+1})$. We then have that any permutation composed with the cup map gives
    \[
    g\circ{\acup}+g\circ{s_{k+1}}\circ{\acup}=g\circ{\acup}-g\circ{\acup}=0.
    \]
    Thus, $e_{k+2}\circ{\acup}=0$.

\end{proof}
\end{lemma}

Let $\underline{\mathcal{TL}}^{\oplus}$ denote the category $\Ind(K(\Kar(\Kom{(\TLcat_{n})})))$. We are now ready to define the Kirby-belted identity tangle $T_{n}^{\omega_{\alpha}}$.

\begin{definition}\label{def:dirsys}
For $\alpha \in \N$, let $T_n^{\omega_\alpha} \in \underline{\mathcal{TL}}^{\oplus}$ denote the colimit of the directed system
\[
    \mathcal{A}_{n}^{\alpha} := 
    \left (
        \symmetrized{T_n^{\otimes \alpha}}
        \map{\symmetrized{\acupdot}}
        \symmetrized{T_n^{\otimes \alpha+2}}
        \map{\symmetrized{\acupdot}}
        \symmetrized{T_n^{\otimes \alpha+4}}
        \map{\symmetrized{\acupdot}}
        \cdots
    \right ).
\]

\end{definition}

Note that only the parity of $\alpha$ matters on the level of colimits, so there are only two Kirby-belted identity objects, corresponding to $\alpha=0,1$. For $n$ vertical strands, we denote $\colim(\mathcal{A}_{n}^{0})$ and $\colim(\mathcal{A}_{n}^{1})$ by $T_{n}^{\omega_{0}}$ and $T_{n}^{\omega_{1}}$ respectively. We also consider the homotopy colimits of $\mathcal{A}_{n}^{0}$ and $\mathcal{A}_{n}^{1}$, described as follows.

\begin{definition}
    Let $T_n^{\Omega_\alpha}$ denote the homotopy colimit of the directed system $\mathcal{A}_{n}^{\alpha}$ in Definition \ref{def:dirsys}. In particular, $T_{n}^{\Omega_{\alpha}}$ is the total complex $\Tot(\mathcal{D}_{\mathcal{A}_{n}^{\alpha}})$ of the double complex $\mathcal{D}_{\mathcal{A}_{n}^{\alpha}}$ associated to $\mathcal{A}_{n}^{\alpha}$ as in Figure \ref{fig:double-complex-B}. (see Figure \ref{fig:TnOmega0} for the double complex representing $T_{n}^{\Omega_{0}}$).
\end{definition}

\begin{figure}[!h]
    \begin{center}
        \includegraphics[width=.9\linewidth]{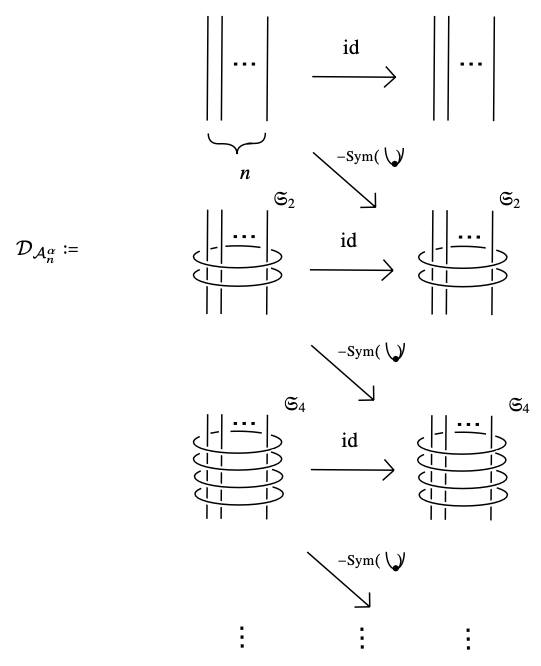}
    \end{center}
    \caption{For $\alpha=0$. The diagrams $(T_{n}^{\otimes{k}})^{\mathfrak{S}_{k}}$ denote the symmetrized complexes $\symmetrized{T_{n}^{\otimes{k}}}$.}
    \label{fig:TnOmega0}
\end{figure}

Note that $T^{\omega_{\alpha}}_{n}$ and $T^{\Omega_{\alpha}}_{n}$ are equivalent by Proposition \ref{prop:totB-is-colim}, so for the remainder of this work we will denote this object only by $T_{n}^{\Omega_{\alpha}}$, and label Kirby-colored components with $\Omega_{\alpha}$. 

\begin{remark}\label{rmk:annularkirby} Despite using `Kirby-colored' terminology, we work with homotopy colimits of tangle complexes not in the annular setting, so our construction is different from that of \cite{hogancamp2022kirby}. The \emph{$i$th Kirby object} (\cite{hogancamp2022kirby}) $\omega^{\prime}_{i}$ for $i\geq{0}$ is an object in the Ind-completion of the additive closure of the Karoubi envelope of the annular Bar-Natan category represented by the colimit 
\[
\omega^{\prime}_{i}:=\left(q^{-i}P_{i}\rightarrow{q^{-(i+2)}P_{i-2}\rightarrow q^{-(i+4)}}P_{i-4}\rightarrow{\cdots}\right)
\]
where the arrows are certain dotted maps between projectors. Letting $L^{\text{att}}$ once again denote the framed oriented link that a $2$-handle is attached to $B^{4}$ along, and letting $L$ be a boundary link in $\partial(B^{4}(L^{\text{att}}))$, a theorem of Hogancamp-Rose-Wedrich can be stated as follows.

\begin{theorem}[\cite{hogancamp2022kirby}, Theorem C]
    Let $(B^{4}(L^{att}),L)$ denote the $4$-manifold obtained by attaching a $2$-handle to $B^{4}$ along the $n$ component link $L^{att}$ and let $\omega^{\prime}_{\underline{i}}$ denote a collection of Kirby objects where $\underline{i}\in{\{0,1\}^{n}}$. Decorate the $n$ components of $L^{att}$ with $n$ Kirby objects $\omega^{\prime}_{i}$ for $i\in{\{0,1\}}$ (In other words, decorate $L^{att}$ with $\omega^{\prime}_{\underline{i}}$). Then the following bigraded vector spaces are isomorphic:
    \begin{enumerate}
        \item[(a)] The Kirby-colored Khovanov homology $\Kh(L\cup{(L^{att})^{\omega^{\prime}_{\underline{i}}})}$
        \item[(b)] The relative cabled Khovanov homology of $L\cup{L^{att}}$ at homological level $\underline{i}$.
        \item[(c)] The $N=2$ skein lasagna module of $(B^{4}(L^{att}),L)$ at homological level $\underline{i}$.
    \end{enumerate}
\end{theorem}

The reason that items (a) and (c) are equivalent to (b) is because the relative cabled Khovanov homology of $L\cup{L^{att}}$ is precisely the Manolescu-Walker-Wedrich \emph{cabled skein lasagna} construction in \cite{MWWhandles} for a $4$-manifold with no $1$ or $3$-handles and an arbitrary link $L$ in the boundary.
\end{remark}

\section{Skein lasagna module computations using Kirby-colored belts}\label{sec:mainsec}

Let $\{p_1, \ldots, p_n\}$ be a collection of $n$ distinct points on $S^2$. 
In this section, we compute the skein lasagna module $\skein_0^2(S^2 \times B^2; \widetilde{\one}_{n})$,
where $\widetilde{\one}_{n}$ is the geometrically essential link $\{ p_1, \ldots, p_n \} \times S^1 \subset S^1 \times S^2 = \partial W$ (see Figure \ref{fig:geom-essen}). 

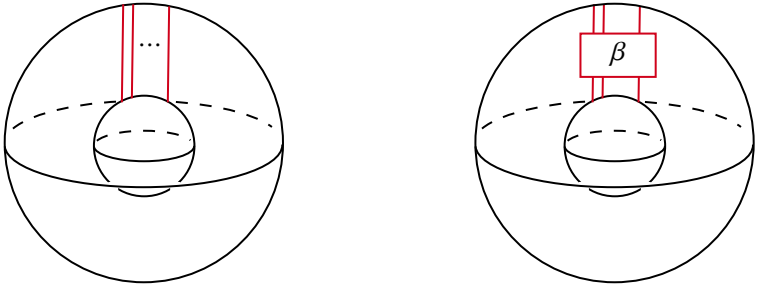
\begin{figure}[!h]
    \begin{center}
        \tikzset{every picture/.style={line width=0.75pt}} 

\begin{tikzpicture}[x=0.75pt,y=0.75pt,yscale=-1,xscale=1]

\draw  [draw opacity=0][dash pattern={on 4.5pt off 4.5pt}] (368.38,92.87) .. controls (378.06,84.76) and (403.67,78.96) .. (433.73,78.96) .. controls (462.77,78.96) and (487.64,84.36) .. (498.03,92.05) -- (433.73,100.07) -- cycle ; \draw  [dash pattern={on 4.5pt off 4.5pt}] (368.38,92.87) .. controls (378.06,84.76) and (403.67,78.96) .. (433.73,78.96) .. controls (462.77,78.96) and (487.64,84.36) .. (498.03,92.05) ;  
\draw  [color={rgb, 255:red, 255; green, 255; blue, 255 }  ,draw opacity=1 ][fill={rgb, 255:red, 255; green, 255; blue, 255 }  ,fill opacity=1 ] (424.4,76.73) -- (445.07,76.73) -- (445.07,83.07) -- (424.4,83.07) -- cycle ;
\draw   (409.73,101.57) .. controls (409.73,87.67) and (421,76.4) .. (434.9,76.4) .. controls (448.8,76.4) and (460.07,87.67) .. (460.07,101.57) .. controls (460.07,115.47) and (448.8,126.73) .. (434.9,126.73) .. controls (421,126.73) and (409.73,115.47) .. (409.73,101.57) -- cycle ;
\draw [color={rgb, 255:red, 208; green, 2; blue, 27 }  ,draw opacity=1 ]   (424.4,30.73) -- (423.73,79.73) ;
\draw [color={rgb, 255:red, 208; green, 2; blue, 27 }  ,draw opacity=1 ]   (429.4,30.73) -- (428.73,77.73) ;
\draw [color={rgb, 255:red, 208; green, 2; blue, 27 }  ,draw opacity=1 ]   (447.4,31.73) -- (446.73,79.73) ;
\draw  [color={rgb, 255:red, 208; green, 2; blue, 27 }  ,draw opacity=1 ][fill={rgb, 255:red, 255; green, 255; blue, 255 }  ,fill opacity=1 ] (417.67,45.33) -- (455.33,45.33) -- (455.33,67) -- (417.67,67) -- cycle ;
\draw  [draw opacity=0] (224.27,101.48) .. controls (224.27,101.51) and (224.27,101.54) .. (224.27,101.57) .. controls (224.27,105.78) and (213.03,109.19) .. (199.17,109.19) .. controls (185.59,109.19) and (174.53,105.92) .. (174.08,101.83) -- (199.17,101.57) -- cycle ; \draw   (224.27,101.48) .. controls (224.27,101.51) and (224.27,101.54) .. (224.27,101.57) .. controls (224.27,105.78) and (213.03,109.19) .. (199.17,109.19) .. controls (185.59,109.19) and (174.53,105.92) .. (174.08,101.83) ;  
\draw  [draw opacity=0][dash pattern={on 4.5pt off 4.5pt}] (175.56,98.97) .. controls (179.06,96.04) and (188.31,93.94) .. (199.17,93.94) .. controls (209.66,93.94) and (218.64,95.89) .. (222.4,98.67) -- (199.17,101.57) -- cycle ; \draw  [dash pattern={on 4.5pt off 4.5pt}] (175.56,98.97) .. controls (179.06,96.04) and (188.31,93.94) .. (199.17,93.94) .. controls (209.66,93.94) and (218.64,95.89) .. (222.4,98.67) ;  
\draw   (129.4,100.07) .. controls (129.4,61.59) and (160.59,30.4) .. (199.07,30.4) .. controls (237.54,30.4) and (268.73,61.59) .. (268.73,100.07) .. controls (268.73,138.54) and (237.54,169.73) .. (199.07,169.73) .. controls (160.59,169.73) and (129.4,138.54) .. (129.4,100.07) -- cycle ;
\draw  [draw opacity=0][dash pattern={on 4.5pt off 4.5pt}] (133.54,92.87) .. controls (143.22,84.76) and (168.83,78.96) .. (198.89,78.96) .. controls (227.92,78.96) and (252.8,84.36) .. (263.19,92.05) -- (198.89,100.07) -- cycle ; \draw  [dash pattern={on 4.5pt off 4.5pt}] (133.54,92.87) .. controls (143.22,84.76) and (168.83,78.96) .. (198.89,78.96) .. controls (227.92,78.96) and (252.8,84.36) .. (263.19,92.05) ;  
\draw  [color={rgb, 255:red, 255; green, 255; blue, 255 }  ,draw opacity=1 ][fill={rgb, 255:red, 255; green, 255; blue, 255 }  ,fill opacity=1 ] (188.73,76.73) -- (209.4,76.73) -- (209.4,83.07) -- (188.73,83.07) -- cycle ;
\draw   (174.07,101.57) .. controls (174.07,87.67) and (185.33,76.4) .. (199.23,76.4) .. controls (213.13,76.4) and (224.4,87.67) .. (224.4,101.57) .. controls (224.4,115.47) and (213.13,126.73) .. (199.23,126.73) .. controls (185.33,126.73) and (174.07,115.47) .. (174.07,101.57) -- cycle ;
\draw  [color={rgb, 255:red, 255; green, 255; blue, 255 }  ,draw opacity=1 ][fill={rgb, 255:red, 255; green, 255; blue, 255 }  ,fill opacity=1 ] (181.73,120.07) -- (185.73,120.07) -- (185.73,124.07) -- (181.73,124.07) -- cycle ;
\draw  [color={rgb, 255:red, 255; green, 255; blue, 255 }  ,draw opacity=1 ][fill={rgb, 255:red, 255; green, 255; blue, 255 }  ,fill opacity=1 ] (212.73,120.07) -- (216.73,120.07) -- (216.73,124.07) -- (212.73,124.07) -- cycle ;
\draw  [draw opacity=0] (268.38,100.84) .. controls (268.38,100.91) and (268.38,100.99) .. (268.38,101.07) .. controls (268.38,112.73) and (237.27,122.18) .. (198.89,122.18) .. controls (161.3,122.18) and (130.68,113.11) .. (129.44,101.78) -- (198.89,101.07) -- cycle ; \draw   (268.38,100.84) .. controls (268.38,100.91) and (268.38,100.99) .. (268.38,101.07) .. controls (268.38,112.73) and (237.27,122.18) .. (198.89,122.18) .. controls (161.3,122.18) and (130.68,113.11) .. (129.44,101.78) ;  
\draw [color={rgb, 255:red, 208; green, 2; blue, 27 }  ,draw opacity=1 ]   (188.73,30.73) -- (188.07,79.73) ;
\draw [color={rgb, 255:red, 208; green, 2; blue, 27 }  ,draw opacity=1 ]   (193.73,30.73) -- (193.07,77.73) ;
\draw [color={rgb, 255:red, 208; green, 2; blue, 27 }  ,draw opacity=1 ]   (211.73,31.73) -- (211.07,79.73) ;
\draw  [draw opacity=0] (459.94,101.48) .. controls (459.94,101.51) and (459.94,101.54) .. (459.94,101.57) .. controls (459.94,105.78) and (448.7,109.19) .. (434.84,109.19) .. controls (421.26,109.19) and (410.2,105.92) .. (409.75,101.83) -- (434.84,101.57) -- cycle ; \draw   (459.94,101.48) .. controls (459.94,101.51) and (459.94,101.54) .. (459.94,101.57) .. controls (459.94,105.78) and (448.7,109.19) .. (434.84,109.19) .. controls (421.26,109.19) and (410.2,105.92) .. (409.75,101.83) ;  
\draw  [draw opacity=0][dash pattern={on 4.5pt off 4.5pt}] (411.23,98.97) .. controls (414.73,96.04) and (423.98,93.94) .. (434.84,93.94) .. controls (445.32,93.94) and (454.31,95.89) .. (458.06,98.67) -- (434.84,101.57) -- cycle ; \draw  [dash pattern={on 4.5pt off 4.5pt}] (411.23,98.97) .. controls (414.73,96.04) and (423.98,93.94) .. (434.84,93.94) .. controls (445.32,93.94) and (454.31,95.89) .. (458.06,98.67) ;  
\draw   (365.07,100.07) .. controls (365.07,61.59) and (396.26,30.4) .. (434.73,30.4) .. controls (473.21,30.4) and (504.4,61.59) .. (504.4,100.07) .. controls (504.4,138.54) and (473.21,169.73) .. (434.73,169.73) .. controls (396.26,169.73) and (365.07,138.54) .. (365.07,100.07) -- cycle ;
\draw  [color={rgb, 255:red, 255; green, 255; blue, 255 }  ,draw opacity=1 ][fill={rgb, 255:red, 255; green, 255; blue, 255 }  ,fill opacity=1 ] (417.4,120.07) -- (421.4,120.07) -- (421.4,124.07) -- (417.4,124.07) -- cycle ;
\draw  [color={rgb, 255:red, 255; green, 255; blue, 255 }  ,draw opacity=1 ][fill={rgb, 255:red, 255; green, 255; blue, 255 }  ,fill opacity=1 ] (448.4,120.07) -- (452.4,120.07) -- (452.4,124.07) -- (448.4,124.07) -- cycle ;
\draw  [draw opacity=0] (504.04,100.84) .. controls (504.05,100.91) and (504.05,100.99) .. (504.05,101.07) .. controls (504.05,112.73) and (472.94,122.18) .. (434.56,122.18) .. controls (396.97,122.18) and (366.35,113.11) .. (365.11,101.78) -- (434.56,101.07) -- cycle ; \draw   (504.04,100.84) .. controls (504.05,100.91) and (504.05,100.99) .. (504.05,101.07) .. controls (504.05,112.73) and (472.94,122.18) .. (434.56,122.18) .. controls (396.97,122.18) and (366.35,113.11) .. (365.11,101.78) ;  

\draw (195.07,48.73) node [anchor=north west][inner sep=0.75pt]  [font=\normalsize] [align=left] {$\displaystyle \cdots $};
\draw (431,48.05) node [anchor=north west][inner sep=0.75pt]   [align=left] {$\displaystyle \beta $};

\end{tikzpicture}
    \end{center}
    \caption{\textbf{Left:} The $n$ component link $\widetilde{\one}_{n}$ in $S^{1}\times{S^{2}} = [0,1] \times S^2 / (0,p) \sim (1,p)$. \textbf{Right:} An example of a geometrically essential link $\widetilde{\beta}$ given by a braid $\beta$ in $S^{1}\times{S^{2}}$.} 
    \label{fig:geom-essen}
\end{figure}

\begin{remark}
    There is a notational ambiguity when referring to links in $S^{1}\times{S^{2}}$ and the closures of links more generally. Throughout this section, we use the symbol $\widehat{\one}_{n}$ when referring to the usual closure of the identity braid in the (thickened) plane, and use the symbol $\widetilde{\one}_{n}$ when referring to the specific link in $S^{1}\times{S^{2}}$ shown on the left in Figure \ref{fig:geom-essen}.
\end{remark}

\subsection{Equivalence of $H^{*}(\close({T}_{n}^{\Omega_{\alpha}}))$ and $\skein^{2}_{0}(S^{2}\times{B^{2}};\widetilde{\one}_{n},\alpha)$}
Here we confirm that the homology of the trace of $T^{\Omega_{\alpha}}_{n}$ is isomorphic to the skein lasagna module of $S^{2}\times{B^{2}}$ with $\widetilde{\one}_{n}$ in the boundary. 

\begin{definition}\label{def:trace}
    Let $\close:\TLcat_{n}\rightarrow{\BNcat}$ denote the trace functor.
    \[
        \tikzset{every picture/.style={line width=0.75pt}} 

\begin{tikzpicture}[x=0.75pt,y=0.75pt,yscale=-1,xscale=1]

\draw   (426.2,88.82) .. controls (426.2,68.61) and (434.21,52.22) .. (444.1,52.22) .. controls (453.99,52.22) and (462,68.61) .. (462,88.82) .. controls (462,109.03) and (453.99,125.42) .. (444.1,125.42) .. controls (434.21,125.42) and (426.2,109.03) .. (426.2,88.82) -- cycle ;
\draw    (207.8,54.2) -- (207.6,74.62) ;
\draw    (217.8,54.2) -- (217.6,74.62) ;
\draw    (248.6,53.8) -- (248.4,74.22) ;
\draw   (199.8,74.4) -- (258.4,74.4) -- (258.4,101.82) -- (199.8,101.82) -- cycle ;
\draw    (207.4,102) -- (207.2,122.42) ;
\draw    (217.4,102) -- (217.2,122.42) ;
\draw    (248.2,101.6) -- (248,122.02) ;
\draw    (287.8,93.8) -- (366.4,94.21) ;
\draw [shift={(368.4,94.22)}, rotate = 180.3] [color={rgb, 255:red, 0; green, 0; blue, 0 }  ][line width=0.75]    (10.93,-3.29) .. controls (6.95,-1.4) and (3.31,-0.3) .. (0,0) .. controls (3.31,0.3) and (6.95,1.4) .. (10.93,3.29)   ;
\draw   (396,90.02) .. controls (396,61.85) and (412.03,39.02) .. (431.8,39.02) .. controls (451.57,39.02) and (467.6,61.85) .. (467.6,90.02) .. controls (467.6,118.19) and (451.57,141.02) .. (431.8,141.02) .. controls (412.03,141.02) and (396,118.19) .. (396,90.02) -- cycle ;
\draw   (385.1,89.42) .. controls (385.1,58.27) and (404.82,33.02) .. (429.15,33.02) .. controls (453.48,33.02) and (473.2,58.27) .. (473.2,89.42) .. controls (473.2,120.57) and (453.48,145.82) .. (429.15,145.82) .. controls (404.82,145.82) and (385.1,120.57) .. (385.1,89.42) -- cycle ;
\draw  [fill={rgb, 255:red, 255; green, 255; blue, 255 }  ,fill opacity=1 ] (380.6,74.8) -- (439.2,74.8) -- (439.2,102.22) -- (380.6,102.22) -- cycle ;

\draw (225.2,63) node [anchor=north west][inner sep=0.75pt]   [align=left] {$\displaystyle \cdots $};
\draw (221,82.6) node [anchor=north west][inner sep=0.75pt]   [align=left] {$\displaystyle A$};
\draw (224.8,109.8) node [anchor=north west][inner sep=0.75pt]   [align=left] {$\displaystyle \cdots $};
\draw (313.4,68.4) node [anchor=north west][inner sep=0.75pt]   [align=left] {$\displaystyle \text{Tr}$};
\draw (405.2,66.8) node [anchor=north west][inner sep=0.75pt]   [align=left] {$\displaystyle \cdots $};
\draw (401.8,82) node [anchor=north west][inner sep=0.75pt]   [align=left] {$\displaystyle A$};
\draw (403.4,106.2) node [anchor=north west][inner sep=0.75pt]   [align=left] {$\displaystyle \cdots $};

\end{tikzpicture}
    \]
\end{definition}

By functoriality, we may take 
the tangle closure as in Definition \ref{def:trace} of each term in the double complex $\mathcal{D}_{\mathcal{A}^{\alpha}_{n}}$,
then take the homology of each term to obtain a
{double complex of (bigraded) vector spaces 
$\widehat{\mathcal{D}}_{\mathcal{A}^{\alpha}_{n}}$  depicted in Figure \ref{fig:traceofhocolim}.}

\begin{figure}[!h]
    \begin{center}
    \includegraphics[width=.8\linewidth]{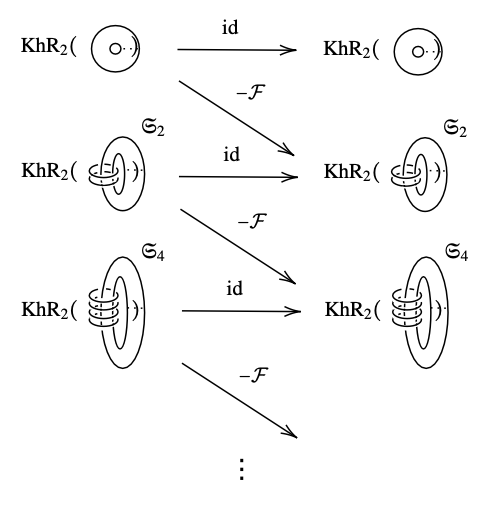}
    \end{center}
    \caption{The {double} complex $\widehat{\mathcal{D}}_{\mathcal{A}^{\alpha}_{n}}$ associated to the closure of $T_{n}^{\Omega_{0}}$, where $\mathcal{F}$ denotes the morphism induced by the symmetrized dotted annulus map. The quantum degree shifts from the definition of cabled Khovanov homology have been suppressed.
    }
    \label{fig:traceofhocolim}
\end{figure}

{
The totalization of this double complex is 
\[
\Tr(T_n^{\Omega_\alpha}) := \Tot(
\widehat{\mathcal{D}}_{\mathcal{A}^{\alpha}_{n}}
)
=\bigoplus_{k=0}^{\infty}\KhR_{2}(\symmetrized{\close(T_{n}^{\otimes{2k}+\alpha}}))\xrightarrow{\idmap-\mathcal{F}}\bigoplus_{k=0}^{\infty}\KhR_{2}(\symmetrized{\close(T_{n}^{\otimes{2k}+\alpha}})).
\]
where $\mathcal{F}$ is comprised of the morphisms induced by symmetrized dotted ribbon maps. 
}

Then, by Lemma \ref{lem:Hcone} and Corollary \ref{cor:Hcone}, we obtain the following proposition.

\begin{proposition}
    \label{prop:everythingthesame}
     Let $U_{0}\cup{\widehat{\one}_{n}}$ denote the link obtained by the tangle closure of $T_{n}$, where the belt is a $0$-framed unknot $U_{0}$. We have the following isomorphisms of vector spaces:
     
    \begin{align*}
        H^{*}(\close({T}^{\Omega_{\alpha}}_{n}))
        &\cong{H^{*}(\Tot(\widehat{\mathcal{D}}_{\mathcal{A}^{\alpha}_{n}}}))\\
        &\cong{\colim{(\mathcal{B}^{\alpha}({S^{2}\times{B^{2}};\widetilde{\one}_{n}})})}\\
        &\cong{\underline{\KhR}_{2,\alpha}(U_{0},\widehat{\one}_{n})}\\
        &\cong{\skein_{0}^{2}(S^{2}\times{B^{2}};\widetilde{\one}_{n}},\alpha)
        \end{align*}

    \begin{proof}
        The first isomorphism $H^{*}(\close(T_{n}^{\Omega_{\alpha}}))\cong{H^{*}(\Tot{(\widehat{\mathcal{D}}_{\mathcal{A}^{\alpha}_{n}})})}$ is an application of Lemma \ref{lem:Hcone} and Corollary \ref{cor:Hcone}. The second and third isomorphisms follow from the observation that the homology of $\Tot{(\widehat{\mathcal{D}}_{\mathcal{A}^{\alpha}_{n}})}$ is manifestly the relative cabled Khovanov homology $\underline{\KhR}_{2,\alpha}(U_{0},\widehat{\one}_{n})$, which is isomorphic to $\colim{(\mathcal{B}^{\alpha}(S^{2}\times{B^{2}};\widehat{\one}_{n}}))$ by construction, and is furthermore isomorphic to $\skein^{2}_{0}(S^{2}\times{B^{2}};\widetilde{\one}_{n},\alpha)$ by Remark \ref{rmk:cabledlasagna}.
    \end{proof}
\end{proposition}

Our results involve the link $\widetilde{\one}_{n}$ because we consider belts encircling the $n$-strand identity braid.
However, by stacking braids, we obtain similar results for other geometrically essential links in $\partial(S^{2}\times{B^{2}})=S^{1}\times{S^{2}}$. Letting $\beta$ denote the tangle complex associated to an $n$-strand braid, we can replace each copy of $T_{n}^{\otimes{\alpha+2k}}$ above with the tangle complex $\beta\otimes{T_{n}^{\otimes{\alpha+2k}}}$. 
Let $\widehat{\mathcal{D}}_{\mathcal{A}^{\alpha}_{n},\beta}$ denote the double complex obtained from $\mathcal{D}_{\mathcal{A}_{n}^{\alpha}}$ with each $\symmetrized{T_{n}^{\otimes{\alpha+2k}}}$ term replaced by $\symmetrized{\beta\otimes{T_{n}^{\otimes{\alpha+2k}}}}$ and trace and homology taken term-by-term. Let $\close(\beta\otimes{T_{n}^{\Omega_{\alpha}})}$ denote $\Tot{(\widehat{\mathcal{D}}_{\mathcal{A}^{\alpha}_{n},\beta})}$.

\begin{corollary}\label{cor:everythingthesame(w/braid)}
    Taking the trace of $\beta\otimes{T_{n}^{\otimes{k}}}$ and taking homology, we obtain isomorphisms:
    
    \begin{align*}H^{*}(\close(\beta\otimes{T_{n}^{\Omega_{\alpha}}}))&\cong{H^{*}(\Tot{(\widehat{\mathcal{D}}_{\mathcal{A}^{\alpha}_{n},\beta})}})\\
    &\cong{\colim{(\mathcal{B}^{\alpha}(S^{2}\times{B^{2}};\widetilde{\beta}}))}\\
    &\cong{\underline{\KhR}_{2,\alpha}(U_{0},\widehat{\beta})}\\
    &\cong{\skein_{0}^{2}(S^{2}\times{B^{2}};\widetilde{\beta},\alpha)}
    \end{align*}

\end{corollary}

Thus, we are able to study the skein lasagna modules of pairs $(S^{2}\times{B^{2}},\widetilde{\beta})$ by studying the homotopy colimits $T_{n}^{\Omega_{\alpha}}$ and $\beta\otimes{T_{n}^{\Omega_{\alpha}}}$. 
We begin the study of these homotopy colimits by calculating $\symmetrized{P_{n}\otimes{T_{n}^{\otimes{k}}}}$ for $k \geq 0$ computing $P_n \otimes T_n^{\Omega_\alpha}$ for $i = 0,1$.

\subsection{Projectors wearing symmetrized belts}

Our first goal is to explicitly describe the symmetric part of $P_n \otimes T_n^{\otimes k}$ under the $\symgroup{k}$ action permuting the $k$ belts. Let us first consider the complex $P_n \otimes T_n$, i.e.\ a projector wearing one belt. 

\begin{remark}
\label{rem:orientations}
Given a tangle diagram $D$, we first use the unoriented skein relation \eqref{eq:unoriented-skein-relation} to decompose the diagram into flat tangles. We then introduce the global bigrading shift dictated by \eqref{eq:oriented-skein-relations}. 
On the other hand, the projector $P_n$ is already defined to be an object in $\Kom(\TLcat_{n})$ (see Section \ref{sec:projectors}) with absolute gradings. 
\end{remark}

\begin{lemma}\cite[Corollary 3.51]{Hog-polyn-action}
The unoriented complex $P_{n}\otimes{T_{n}}$ splits as 
\begin{center}
    \includegraphics[width=.8\linewidth]{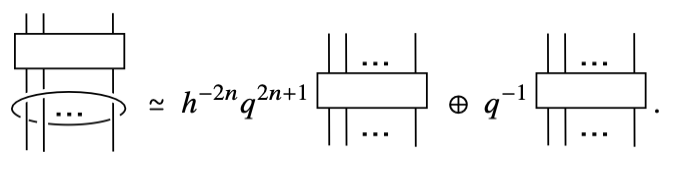}
\end{center}
\end{lemma}

\begin{corollary}
\label{cor:Pn-with-belt-split}
For the projector $P_{n}$ with $k$ belts, we have
\begin{center}
    \includegraphics[width=.7\linewidth]{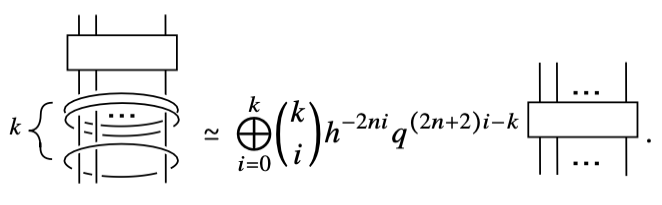}
\end{center}

\begin{proof}
Note that $P_{n}\otimes{T_{n}}\simeq{T_{n}\otimes{P_{n}}}$, and that $P_n$ is idempotent, so $(P_{n}\otimes{T_{n}})^{\otimes{k}}\simeq{P_{n}\otimes{T_{n}^{\otimes{k}}}}$. For the degree shifts, note that
$(h^{-2n}q^{2n+1})^i (h^0q^{-1})^{k-i}
= h^{-2ni} q^{(2n+2)i-k}$.
\end{proof}
\end{corollary}

We now identify the $\symgroup{k}$ action on the right-hand side of the homotopy equivalence in Corollary \ref{cor:Pn-with-belt-split}. Let $V_i$ be a ${k \choose i}$-dimensional vector space at bigrading $(-2ni, (2n+2)i-k)$.
Let $[k]$ denote the set of indices $\{1, 2, \ldots, k\}$. 
The standard basis vectors in $V_i$ can be identified with the set of multi-indices $\{I \subset [k] \st |I| = i \}$. Let $V$ denote the vector space $\bigoplus_{i=0}^k V_i$.

The symmetric group $\symgroup{k}$ acts on each $V_i$ by permuting the elements of $[k]$. To be precise, if $\sigma \in \symgroup{k}$, and $I = \{ j_1, j_2, \ldots, j_i\}$, then $\sigma I = \{ \sigma(j_1), \sigma(j_2), \ldots, \sigma(j_i)\}$. 
Let $\mathbf{V}$ denote the $\symgroup{k}$ representation $\bigoplus_{i=0}^k V_i$. 

Let $\sigma \in \symgroup{k}$ and let $\Ext^{i,j}(P_{n})$ denote the group of bidegree $(i,j)$ endomorphisms of $P_{n}$ modulo chain homotopy. A priori, with some choice of basis for $V_i$, the action of $\sigma$ on $P_n \otimes V_i$ is given by a 
${k \choose i} \times {k \choose i}$ matrix $M_\sigma$ with entries in $\Ext^{0,0}(P_n)$. However, by \cite[Corollary 3.35(2)]{Hog-polyn-action} (and the universal coefficient theorem), $\Ext^{0,0}(P_n) \cong \BF$, so we may view $M_\sigma$ as a matrix with coefficients in $\BF$. 

Here we wish to show that by some choice of basis, $M_\sigma$ is precisely the matrix representing the action of $\sigma$ on $V_i$. 
To do this, we will rely on Grigsby-Licata-Wehrli's description of the $\symgroup{k}$ action on the canonical generators in the Lee homology of $\cl{T_n^{\otimes k}}$, so some setup is in order.

Each multi-index $I$ determines a sign sequence $\epsilon_I$  dictating an orientation on the $k$ belts in $T_n^{\otimes k}$. 
Let $o_I$ denote the orientation on $T_n^{\otimes k}$ where the vertical strands are all oriented upwards, and the belts are oriented according to $\epsilon_I$, where the belt at position $j \in [k]$ links negatively
with the vertical strands if and only if $j \in I$.

By naturality of the trace functor, taking an $(n,n)$-tangle $T$ to $\cl{T}$, we may instead consider the object 
$\cl{P_n \otimes T_n^{\otimes k}}$ in $\Kom(\BNcat)$. 
Since the $\symgroup{k}$ acts by cobordism maps, we have that 
\[
    \cl{P_n \otimes T_n^{\otimes k}}
    \simeq
    \cl{P_n} \otimes V.
\]

We will now take the trace and apply the Lee homology functor, which will allow us to pick out a set of homology classes; by keeping track of the action of $\symgroup{k}$, on this set, we will identify the representation $V$. 

Let $\FT_n \in \Kom(\TLcat_n)$ denote the positive full twist on $n$ strands.
Recall that $P_n = \colim_{m\to\infty} {\FT_n^{\otimes m}}$, where each $\FT_n^{\otimes m} \overset{\iota}{\hookrightarrow} \FT_n^{\otimes m+1}$ as a subcomplex \cite{Roz-Pn}.
After applying the Lee homology functor $\Lee: \Kom(\BNcat) \to \gVect$, 
by \cite{lee-endo} we have that $\Lee(\cl{\FT_n^{\otimes m} \otimes T_n^{\otimes k}})$ is generated by the Lee canonical classes $\{\fraks_o\}$, which are in bijection with the set of orientations on the link $\cl{\FT_n^{\otimes m} \otimes T_n^{\otimes k}}$.

By an abuse of notation, we let $o_I$ also denote the orientation on the $n+k$ components of $\cl{\FT_n^{\otimes m} \otimes T_n^{\otimes k}}$ where the vertical strands are all oriented upwards, and the $k$ belts are oriented according to $I$. 
Let $\fraks^m_I$ denote the Lee generator corresponding to $o_I$. 
Observe that under the maps induced by the inclusion maps of subcomplexes
\[
    \Lee(\cl{\FT_n^{\otimes m} \otimes T_n^{\otimes k}})
    \map{\iota_*} \Lee(\cl{\FT_n^{\otimes m+1} \otimes T_n^{\otimes k}}),
\]
the Lee class $\fraks_I^m$ is mapped to the Lee class $\fraks_I^{m+1}$. (This can be deduced by considering the oriented resolution of the two links, and verifying that the inclusion map $\iota$ identifies the Lee cycles $s^m_I$ and $s^{m+1}_I$ whose (nonzero) homology classes are $\fraks^m_I$ and $\fraks^{m+1}_I$, respectively.)

Hence we may define colimits
\[
    \fraks_I := \colim_{m \to \infty} \fraks^m_I,
\]
which are (nonzero) classes in 
\begin{equation}
\label{eq:Lee-of-projector}
    \Lee(\cl{P_n \otimes T_n^{\otimes k}}) 
    := \colim_{m \to \infty} \Lee(\cl{\FT_n^{\otimes m} \otimes T_n^{\otimes k}})
    \cong \Lee(\cl{P_n}) \otimes V.
\end{equation}

\begin{lemma}
\label{lem:Lee-projector}
The Lee homology of the trace of the projector $\Lee(\cl{P_n})$ is two dimensional, generated by the Lee generators corresponding to the braidlike and antibraidlike orientations. 
\begin{proof}
    The braid $\FT_n$ contains ${n(n-1)}$ crossings, i.e.\ two crossings between any two given strands. 

    Let $J \subset [n]$ be a multiindex of weight $j$. That is, in the corresponding orientation on $\FT_n$, there are $n_{\uparrow} = j$ strands pointing upward (braidlike), and $n_{\downarrow} = n-j$ strands pointing downward (antibraidlike). 

    To understand the relative homological grading of $\fraks_{o_J}$, we must understand the number of negative crossings in $(\FT_n, o_J)$. 

    Each $\uparrow$ strand links with other  $\uparrow$ strands positively, but links with each $\downarrow$ strand once, i.e.\ they cross at two crossings. Since we will also consider the contribution from the other strand, we will count this as one negative crossing. 

    Each $\downarrow$ strand links with other $\downarrow$ strands positively, but links with each $\uparrow$ strand once, i.e. at two crossings. The contribution to negative crossings is again one. 

    Therefore the total number of negative crossings in $(\FT_n, o_J)$ is 
    \[
        n_{\uparrow} n_{\downarrow} + n_{\downarrow} n_{\uparrow}
         =  2n_{\uparrow}n_{\downarrow}
         = 2j(n-j).
    \]
    The total number of negative crossings in $(\FT_n^{\otimes m}, o_J)$ is then $2mj(n-j)$. So, if $j \neq 0$ or $n$, then as the number of full twists increases ($m \to \infty$), the number of negative crossings grows without bound. On the other hand, the braidlike and antibraidlike resolutions remain at homological grading 0 and survive to the colimit.
\end{proof}
\end{lemma}

\begin{proposition}
Under the chain homotopy equivalence in Corollary \ref{cor:Pn-with-belt-split}, the action of $\symgroup{k}$ on $P_n \otimes T_n^{\otimes k}$ agrees with the action of $\symgroup{k}$ on $P_n \otimes \mathbf{V}$, where the action of $P_n$ is trivial. 
\end{proposition}

\begin{proof}

For each fixed $m$, 
Grigsby-Licata-Wehrli show 
that the action of $\sigma \in \symgroup{k}$ sends $\fraks^m_{I} \mapsto \fraks^m_{\sigma I}$ (see \cite[Section 7]{GLW-schur-weyl}). 
By functoriality of Lee homology, the $\symgroup{k}$-action on  
$\Lee(\cl{FT^{\otimes m}_n \otimes T_n^{\otimes k}})$ 
is compatible with the action on 
$\Lee(\cl{FT^{\otimes m+1}_n \otimes T_n^{\otimes k}})$.
Thus, in the colimit, the action of $\sigma$ takes $\fraks_I \mapsto \fraks_{\sigma I}$. 

It remains to verify that for each $i$, the set $\{\fraks_I \st |I| = i\}$ forms a basis for the ${k \choose i}$-dimensional vector space at homological grading $-2ni$ in $\Lee(\cl{P_n \otimes T_n^{\otimes k}})$. 
Note that the homological grading is preserved as we pass from $P_n \otimes T_n^{\otimes k} \in \Kom(\TLcat_n)$ to $\cl{P_n \otimes T_n^{\otimes k}} \in \Kom(\BNcat)$ and further to $\Lee(\cl{P_n \otimes T_n^{\otimes k}}) \in \gVect$ (with $\KhR_2$ conventions). 

Since there are ${k \choose i}$ elements in the set, it suffices to show that they are all linearly independent. 
Indeed, since there are only finitely many of these classes, if there were some nonzero linear combination of the $\{\fraks_I \st |I| = i\}$, there would be some finite level $M$ where for all $m \geq M$, the same relation would hold among $\{\fraks^m_I \st |I| = i\}$; this is impossible because the set of all $\{\fraks_I \st I \subset [k] \}$ forms a subset of a basis for $\Lee(\cl{FT^{\otimes m+1}_n \otimes T_n^{\otimes k}})$. 

To summarize, we have shown that the action of $\symgroup{k}$ is standard on the subspace of $\Lee(P_n \otimes T_n^{\otimes k})$ corresponding to orientations of $P_n \otimes T_n^{\otimes k}$ where the vertical strands are oriented upward. 
The same holds for the set of orientations where the vertical strands are anti-braidlike; let $\bar \fraks_I$ denote the Lee generator corresponding to the orientation $\bar o_I$, where the orientation of \textit{all} $n+k$ strands are reversed from their orientation in $o_I$.

Finally, let $\varphi$ denote the chain homotopy equivalence realizing \eqref{eq:Lee-of-projector}. 
Then the images  of $\{\fraks_{I}\}_{I \subset [k]} \cup \{ \bar \fraks_I\}_{I \subset [k]}$ under $\varphi$ form a basis for $\Lee(P_n) \otimes V$ that realizes that the action of $\sigma \in \symgroup{k}$ as the standard permutation matrix on the $2^k$ subsets of $[k]$. 
Therefore $V \cong \mathbf{V}$ as $\symgroup{k}$ representations. 
\end{proof}

In other words, the $2^k$ $P_n$ components in Corollary \ref{cor:Pn-with-belt-split} correspond to the $2^k$ subsets of $[k]$, and the $\symgroup{k}$ action is the one induced by the natural action of $\symgroup{k}$ on $[k]$. This action has $k+1$ orbits, indexed by the subset size $0 \leq i \leq k$. Therefore 
\begin{equation}\label{eq:17}
    \symmetrized{P_n\otimes T_n^{\otimes k}}
    \simeq 
    \bigoplus_{i=0}^k h^{-2ni} q^{(2n+2)i-k} P_n.
\end{equation}

Finally, we add orientations to the computation of the unoriented bracket (which agrees with the $\KhR$ bracket if all crossings are positive). Let $T_{n}^{+}$ (resp. $T_{n}^{-}$) denote the $n$-strand identity braid with a single counterclockwise (resp. clockwise) oriented belt.

\subsection{Dual Projector with Kirby-colored belt}

For our computations, we will actually need the dual projector.
Recall that diagrams in $\TLcat_n$ are drawn in the unit square $[0,1]^2$ in the $xy$-plane, with $n$ endpoints each on the intervals $[0,1] \times \{0\}$ and $[0,1] \times \{1\}$. 
The dualizing functor 
$(\cdot)\dual:
\Kom^-(\TLcat_n) \to \Kom^+(\TLcat_n)$ reflects the diagrams across the line $y = 1/2$, reverses both homological and quantum degree, and is contravariant on morphisms (see more discussion following Theorem 4.12 in \cite{Hog-spin}). 
Observe, for example, that the complexes for a positive and a negative crossing are dual to each other.

\begin{theorem}[\cite{Hog-spin}, Theorem 4.12]
There is a natural isomorphism 
\[
\Hom^\bullet_{\Kom^-(\TLcat_n)}(A,B) \cong 
q^n \Hom^\bullet_{\Kom^\Pi(\TLcat_0)} 
\left (
    \emptyset, \cl{B A\dual} 
    \right )
\]
\end{theorem}

Since $\cl{P_n \otimes P_n\dual} \cong \cl{P_n\dual \otimes P_n}$ (in $\Kom^\Pi(\TLcat_0)$), we immediately deduce the following:

\begin{corollary}
There is a (grading-preserving) isomorphism
    \[ 
        \Ext(P_n, P_n) \cong \Ext(P_n\dual, P_n\dual).
    \]
\end{corollary}

The dual projector $P_n\dual$ satisfies the same properties as $P_n$ (cf.\ Lemma \ref{lem: projector eats crossings}): it is idempotent, eats crossings, and kills turnbacks, and the identity braid appears only at homological degree 0. Of course, the dual projector is bounded below rather than above. 
This difference will become important later, because it guarantees that any endomorphism of $P_n\dual$ of negative homological degree is nilpotent. 

\begin{lemma}
\label{lem:Pn-Tn2-oriented}
Regardless of how the identity braid $\one_n$ is oriented, $\one_n \otimes T_n^+ \otimes T_n^- = T_n^+ \otimes T_n^-$ must have an equal number of positive and negative crossings. Since there are $4n$ total crossings, $n_- = n_+ = 2n$. Thus
\begin{align*}
    P_n \otimes {T_n^+ \otimes T_n^-}
    &= h^{n_-}q^{-n_-} {P_n \otimes T_n \otimes T_n}\\
    &\simeq 
    h^{-2n} q^{2n+2} P_n
    \oplus 2(h^0q^0 P_n)
    \oplus h^{2n}q^{-2n-2} P_n.
\end{align*}
\end{lemma}

We now describe the dotted annulus map on $\symmetrized{P_{n}\dual \otimes{T_{n}^{\otimes{k}}}}$ with orientations. 
Abusing notation, denote the morphisms in $\Kom(\TLcat_{n})$ induced by the ribbon cobordisms  which wrap (resp.\ unwrap) two antiparallel rings around $\one_n$ by 
 
\[
     \acup, \acupdot: \cl{\one_n} \to \cl{T_n^+ \otimes T_n^-}
    \qquad \text{and} \qquad
    \acap, \acapdot: \cl{T_n^+ \otimes T_n^-} \to \cl{\one_n}.
\]
Here, we assume $\one_n$ is given some orientation. 

We will need to understand the symmetrized map of bidegree $(0,2)$
\begin{equation}
\label{eq:Pn-to-3Pn}
    \symmetrized{\idmap_{P_n\dual} \otimes \acupdot}:
    P_n\dual 
    \to 
    h^{-2n} q^{2n+2} P_n\dual
    \oplus P_n\dual
    \oplus h^{2n}q^{-2n-2} P_n\dual
\end{equation}
that appear in the cabling system defining $P_n\dual \otimes T_n^{\omega_\alpha}$. 
To do this, we will rely on Hogancamp's computations for morphisms between categorified projectors, specifically the $\Ext$ groups listed below:

\begin{lemma}\label{lem:hogextgroups}\cite[Corollary 3.35]{Hog-polyn-action}
    The group $\Ext^{i,j}(P_{n}\dual,P_{n}\dual)$ of maps $h^iq^jP_n\dual \to P_n\dual$ mod homotopy satisfies the following.
    \begin{enumerate}
        \item[(1)] If $k<0$, then $\Ext^{k-i,i}(P_{n}\dual,P_{n}\dual)=0$ for all $i$.\\
        \item[(2)] $\Ext^{0-i,i}(P_{n}\dual,P_{n}\dual)\cong\mathbb{F}$ when $i=0$ and zero otherwise.\\
        \item[(4)] If $i\in{\{2,4,...,2n\}}$, then $\Ext^{2-i,i}(P_{n}\dual,P_{n}\dual)\cong{\mathbb{F}}$ and zero otherwise.
    \end{enumerate}
\end{lemma}

\begin{proposition}
    \label{prop:projectorkirbyprop}
    For $n>0$, The homotopy colimit $P_n\dual\otimes{T_{n}^{\Omega_\alpha}}$ for $\alpha\in{\{0,1\}}$ is contractible.

\begin{proof}

Let $\alpha\in{\{0,1\}}$ and consider the directed system
    \[  \symmetrized{P_n\dual\otimes{T_{n}^{\pm{\alpha}}}}\rightarrow{\symmetrized{P_n\dual\otimes{T_{n}^{\pm{\alpha}}}\otimes(T_{n}^{+}\otimes{T_{n}^{-})}}}\rightarrow{\symmetrized{P_n\dual\otimes{T_{n}^{\pm{\alpha}}}\otimes(T_{n}^{+\otimes{2}}\otimes{T_{n}^{-\otimes{2}})}}}\rightarrow{\cdots}
    \]
denoted $P_n\dual\otimes{\mathcal{A}_{n}^{\alpha}}$, where the arrows are given by $\symmetrized{\idmap_{P_n\dual}\otimes{\acupdot}}$.
By \eqref{eq:17}, each object 
each object in the directed system $P_n\dual\otimes{\mathcal{A}_{n}^{\alpha}}$ is a sum of shifted projectors,

Let us study the components of the (symmetrized) dotted annulus map in \eqref{eq:Pn-to-3Pn} as shown in the diagram below:
\begin{equation}
\label{eq:Pn-to-3Pn-diagram}
    \begin{tikzcd}
    & P_n\dual 
    \arrow[swap]{dl}{\phi_1} 
    \arrow{d}{\phi_2}
    \arrow{dr}{\phi_3}
    & \\
    h^{-2n} q^{2n+2} P_n\dual
        & P_n\dual
        &h^{2n}q^{-2n-2}P_n\dual.
    \end{tikzcd}
\end{equation}

\begin{enumerate}
\item[($\phi_1$)] The bidegree $(0,2)$ map
$\phi_1: P_n \to h^{-2n} q^{2n+2} P_n$
corresponds to a degree-preserving map
\[
    h^{2n}q^{(-2n-2)+2} P_n \to P_n,
\]
up to homotopy,
i.e.\ an element of $\Ext^{2n, -2n}(P_n, P_n)$.
This falls into Case (2) of Lemma \ref{lem:hogextgroups}, with $k=0$ and $i = -2n \neq 0$. 
Therefore $\phi_1 \simeq 0$. 
\item[($\phi_2$)] By Lemma \ref{lem:hogextgroups}, the space of endomorphisms of $P_n\dual$ in bidegree $(0,2)$, up to homotopy, is isomorphic to $\mathbb{F}$. 
In the same paper (\cite{Hog-polyn-action} Theorem 1.13), Hogancamp shows that $\Ext^{0,2}(P_n, P_n)$ is generated by the dotted identity map $U_1^{(n)}$, depicted below:
\begin{center}
\tikzset{every picture/.style={line width=0.75pt}} 

\begin{tikzpicture}[x=0.75pt,y=0.75pt,yscale=-1,xscale=1]

\draw   (331.81,55.42) -- (399.18,55.42) -- (399.18,77.25) -- (331.81,77.25) -- cycle ;
\draw    (338.77,26.19) -- (338.87,55.5) ;
\draw    (348.15,26.03) -- (348.24,55.34) ;
\draw    (391.64,26.03) -- (391.74,55.34) ;
\draw [line width=1.5]  [dash pattern={on 1.69pt off 2.76pt}]  (364.1,47.29) -- (380.24,47.37) ;
\draw    (338.97,77.41) -- (339.06,110.62) ;
\draw    (348.34,77.25) -- (348.44,110.46) ;
\draw    (391.84,77.25) -- (391.93,110.46) ;
\draw [line width=1.5]  [dash pattern={on 1.69pt off 2.76pt}]  (362.5,99.29) -- (378.64,99.37) ;
\draw  [fill={rgb, 255:red, 3; green, 0; blue, 0 }  ,fill opacity=1 ] (335.51,94.01) .. controls (335.51,92.08) and (337.08,90.51) .. (339.01,90.51) .. controls (340.95,90.51) and (342.51,92.08) .. (342.51,94.01) .. controls (342.51,95.95) and (340.95,97.51) .. (339.01,97.51) .. controls (337.08,97.51) and (335.51,95.95) .. (335.51,94.01) -- cycle ;

\draw (357.22,60.24) node [anchor=north west][inner sep=0.75pt]  [font=\large] [align=left] {};
\draw (249.67,69) node [anchor=north west][inner sep=0.75pt]   [align=left] {$\displaystyle U_{1}^{( n)} :=$};

\end{tikzpicture}
\end{center}
This is also true for $P_n\dual$ (the dual morphism to dotted identity is still a dotted identity, but on the dual object). 
Thus $\phi_2 \simeq c U_1^{(n)}$ for some $c \in \mathbb{F}$. Since $U_1^{(n)}$ is nilpotent of order 2, the nilpotency order of $\phi_2$ is at most 2. 
\item[($\phi_3$)] Because $P_n\dual$ is bounded below, the map $\phi_3$ is nilpotent on any class in $P_n\dual$. In other words, fixing a class $x$ at homological degree $\gr_h(x)$ in the complex $P_n\dual$, there exists some $N(x) \in \N$ such that for all $k > N(x)$, the map
\[
    \phi_3^k: P_n\dual \to h^{2nk} q^{(-2n-2)k}P_n\dual
\]
sends $x \to 0$, simply because $h^{2nk} q^{(-2n-2)k}P_n\dual$ has chain group $0$ at homological degree $\gr_h(x)$.
\end{enumerate}

Let $x$ denote a class in the directed system of shifted projectors. 
Any sufficiently long path $\rho$ from $x$ through the directed system will satisfy at least one of the following:
\begin{itemize}
    \item $\rho$ contains an instance of $\phi_1$
    \item $\rho$ contains an instance of $(\phi_2)^2$
    \item $\rho$ contains $>N(x)$ instances of $\phi_3$.
\end{itemize}
By the discussion above, we thus have $x \sim 0$ in the colimit.

 Thus, $\colim{(P_n\dual\otimes{\mathcal{A}_{n}^{\alpha}})}=P_n\dual\otimes{T_{n}^{\omega_{\alpha}}}=0$. Finally, by Proposition \ref{prop:contractiblehocolim}, we have that $P_n\dual\otimes{T_{n}^{\Omega_{\alpha}}}\simeq{0}$ as desired.
\end{proof}
\end{proposition}

\subsection{Homological levels with at least one odd term}

In this section, we prove that the complex $T_{n}^{\Omega_{\alpha}}$ is $0$ when $n$ is odd. To do so, we use the resolution of identity (\ref{eq:resofidentity}) and prove that $P^{\vee}_{n,k}\otimes{T_{n}^{\Omega_{\alpha}}}\simeq{0}$ for $k\leq{n}$ and odd. We begin by proving the following commuting property for $T_{n}^{\Omega_{\alpha}}$.

\begin{lemma}
    \label{lem:kirbyTLgenerator}
    Let $\tau_{i}$ denote the Temperley-Lieb generator on the strands at positions $i$ and $i+1$. Then $\tau_{i}\otimes{T_{n}^{\Omega_{\alpha}}}\simeq{T_{n}^{\Omega_{\alpha}}}\otimes{\tau_{i}}$, for $\alpha\in{\{0,1\}}$ (see Figure \ref{fig:swap-Kirby-TL}).
\end{lemma}

\begin{figure}[!h]
    \centering
    \includegraphics[width=.7\linewidth]{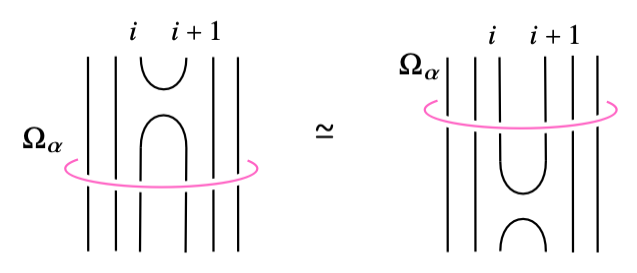}
    \caption{Commuting rule for a Kirby-colored belt and a $\TLcat_{n}$-generator.}
    \label{fig:swap-Kirby-TL}
\end{figure}

\begin{proof}
To obtain the equivalence in Figure \ref{fig:swap-Kirby-TL}, we produce a chain homotopy equivalence map $\Phi:\tau_{i}\otimes{T^{\Omega_{\alpha}}_{n}}\rightarrow{T^{\Omega_{\alpha}}\otimes{\tau_{i}}}$. Note that, for some fixed number of belts, we have a chain homotopy equivalence given by a composition of Reidemeister II moves. The following figure is an example for $\tau_{i}\otimes{T_{n}^{\otimes{2}}}\simeq{T_{n}^{\otimes{2}}\otimes{\tau_{i}}}$.

\begin{center}
    \includegraphics[width=.6\linewidth]{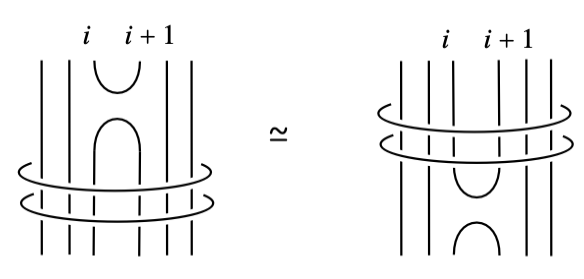}
\end{center}

Let $\mathcal{R}_{i,k}:\tau_{i}\otimes{T_{n}^{\otimes{k}}}\rightarrow{T_{n}^{\otimes{k}}\otimes{\tau_{i}}}$ denote the chain maps associated to the composition of cobordisms that slide the $k$ belts through the diagram $\tau_{i}$. Note that, by functoriality, the maps $\mathcal{R}_{i,k}$ commute with our dotted ribbon maps up to homotopy, and commute as well with the cobordisms associated to the homotopy $\mathfrak{S}_{k}$-action on belts. Hence, we have a homotopy equivalence: 

\begin{center}
    \includegraphics[width=.6\linewidth]{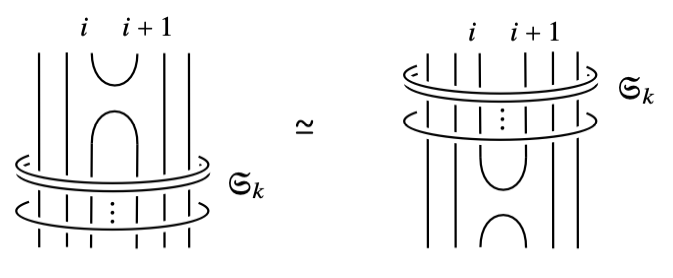}
\end{center}

Let $\alpha_{i,k}:\tau_{i}\otimes{\symmetrized{T_{n}^{\otimes{k}}}}\rightarrow{\symmetrized{T_{n}^{\otimes{k}}}\otimes{\tau_{i}}}$ be the chain map that induces the homotopy equivalence above. Then there is an induced comparison chain map $\alpha:\tau_{i}\otimes{T^{\Omega_{j}}_{n}}\rightarrow{T^{\Omega_{j}}_{n}\otimes{\tau_{i}}}$ that gives a homotopy equivalence of homotopy colimits by Lemma \ref{lem:forSwapsandSlides}(a) and \ref{lem:forSwapsandSlides}(c). 

\end{proof}

\begin{remark}
    The arguments in the proof of Lemma \ref{lem:kirbyTLgenerator} also hold for Temperley-Lieb diagrams with different numbers of endpoints. In particular, if $\tau$ is a chain complex associated to a planar diagram with no crossings in $\Cob_{n,k}$, then by an argument identical to the above, we obtain
    \[
    \tau\otimes{T_{n}^{\Omega_{\alpha}}} \simeq{T_{k}^{\Omega_{\alpha}}} \otimes{\tau}.
    \]
\end{remark}

We will require the following property of a Kirby-colored belt with a cone stacked on top. 

\begin{proposition}\label{prop:kirbycone}
Let $f:A\rightarrow{B}$ be a chain map in $\Kom(\TLcat_{n})$, then there is a well-defined map $f\otimes{\idmap}:A\otimes{T_{n}^{\Omega_{\alpha}}}\rightarrow{B\otimes{T_{n}^{\Omega_{\alpha}}}}$. Furthermore, we have that
    \[
    \cone{(A\otimes{T_{n}^{\Omega_{\alpha}}}\xrightarrow{f\otimes{\idmap}}{B\otimes{T_{n}^{\Omega_{\alpha}}}})}=(\cone{(A\xrightarrow{f}{B}))\otimes{T_{n}^{\Omega_{\alpha}}}}.
    \]
    See Figure \ref{fig:cone-replace-kirby}.
\end{proposition}

\begin{figure}[!h]
    \begin{center}
        \includegraphics[width=\linewidth]{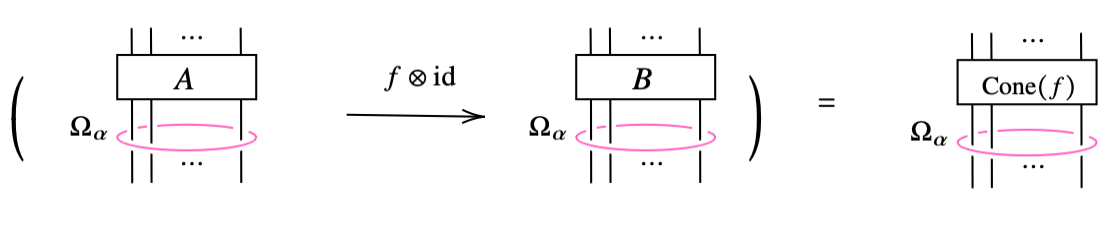}
    \end{center}
    \caption{The cone property for a Kirby-colored belt described in Proposition \ref{prop:kirbycone}.}
    \label{fig:cone-replace-kirby}
\end{figure}

\begin{proof}
    Let $\idmap^{k}:T_{n}^{\otimes{k}}\rightarrow{T_{n}^{\otimes{k}}}$ denote the identity map on $T_{n}^{\otimes{k}}$. Note first that, by cobordism invariance in $\Kom(\TLcat_{n})$, the chain maps $f\otimes{\idmap^{k}}:A\otimes{T_{n}^{\otimes{k}}}\rightarrow{B\otimes{T_{n}^{\otimes{k}}}}$ commute with the dotted ribbon map and the $\mathfrak{S}_{k}$-action permuting the $k$ belts. Thus, the collection of maps $\{f\otimes{\idmap^{k}}\}$ satisfies the hypothesis of Lemma \ref{lem:forSwapsandSlides} for directed systems:
    
    \[  A\otimes{\mathcal{A}_{n}^{\alpha}}:=A\otimes{\symmetrized{T_{n}^{\otimes{\alpha}}}\rightarrow{A\otimes{\symmetrized{T_{n}^{\otimes{\alpha+2}}}}}}\rightarrow{A\otimes{\symmetrized{T_{n}^{\otimes\alpha+4}}}}\rightarrow{   \cdots}
    \]
    \vspace{0.2mm}
    \[ B\otimes\mathcal{A}_{n}^{\alpha}:=B\otimes{\symmetrized{T_{n}^{\otimes{\alpha}}}}\rightarrow{B\otimes{\symmetrized{T_{n}^{\otimes{\alpha+2}}}}}\rightarrow{B\otimes{\symmetrized{T_{n}^{\otimes{\alpha+4}}}}}\rightarrow{\cdots}.
    \]
    \vspace{0.2mm}
    
    By notational abuse, let $f\otimes{\idmap}$ denote the collection $\{f\otimes{\idmap^{k}}\}$, by Lemma \ref{lem:forSwapsandSlides}(a), we have that $f\otimes{\idmap}$ is a well-defined map on the homotopy colimits 
    \[
    f\otimes{\idmap}:\hocolim(A\otimes\mathcal{A}_{n}^{\alpha})\rightarrow{\hocolim(B\otimes\mathcal{A}_{n}^{\alpha}).}
    \]
    
    Let $\mathcal{C}_{n}^{\alpha}$ denote the directed system of cones $\{\cone{(A\otimes{T_{n}^{\otimes\alpha+2m}}\xrightarrow{f\otimes{\idmap^{m}}}B\otimes{T_{n}^{\otimes\alpha+2m}}})\}_{m\in{\mathbb{N}}}$. By Lemma \ref{lem:forSwapsandSlides}(b), we also have the equality
    \begin{equation}\label{eq:cone=hocolim}
    \cone{(f\otimes{\idmap})}=\hocolim{(\mathcal{C}_{n}^{\alpha})}.
    \end{equation}
    
    Since $\cone{(A\otimes{T_{n}^{\otimes\alpha+2m}}\xrightarrow{f\otimes{\idmap^{m}}}{B\otimes{T_{n}^{\otimes\alpha+2m}}}})=\cone{(A\xrightarrow{f}B)}\otimes{T_{n}^{\otimes{m}}}$ by monoidality, we have that $\hocolim{(\CC^{\alpha})}$ is identically the chain complex $\cone{(A\xrightarrow{f}{B})}\otimes{T_{n}^{\Omega_{\alpha}}}$, so  \eqref{eq:cone=hocolim} becomes
    \[
    \cone{(A\otimes{T_{n}^{\Omega_{\alpha}}}\xrightarrow{f\otimes{\idmap}}{B\otimes{T_{n}^{\Omega_{\alpha}}}})}=\cone{(A\xrightarrow{f}{B})\otimes{T_{n}^{\Omega_{\alpha}}}},
    \]
    as desired.
\end{proof}

Since any chain complex $B$ in $\Kom(\TLcat_{n})$ is an iterated mapping cone of chain complexes associated to Temperley-Lieb diagrams, from Proposition \ref{prop:kirbycone} and Lemma \ref{lem:kirbyTLgenerator} we obtain the following corollary.

\begin{corollary}\label{cor:KirbyBSwap}
    Let $B$ be a chain complex in $\Kom(\TLcat_{n}^{k})$. Then $P^{\vee}_{k}\otimes{B}\otimes{T_{n}^{\Omega_{\alpha}}}\simeq{P^{\vee}_{k}\otimes{T_{k}^{\Omega_{\alpha}}}\otimes{B}}\simeq{0}$ for $0<{k}\leq{n}$. See Figure \ref{fig:swap(kirby-TL+projector}. 

\begin{figure}[!h]
    \centering
    \includegraphics[width=.8\linewidth]{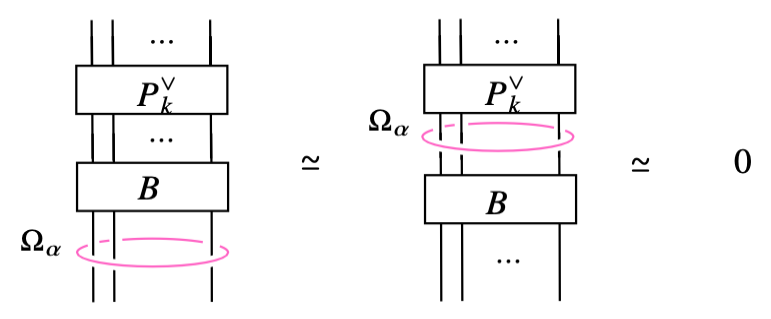}
    \caption{An illustration of Corollary \ref{cor:KirbyBSwap}.}
    \label{fig:swap(kirby-TL+projector}
\end{figure}

\begin{proof}
    Suppose first that the complex $B$ is given by a single $\TLcat_{n}^{k}$-diagram for a positive integer $k\leq{n}$. Then by Lemma \ref{lem:kirbyTLgenerator} we have that $P^{\vee}_{k}\otimes{B}\otimes{T_{n}^{\Omega_{\alpha}}}\simeq{P^{\vee}_{k}\otimes{T_{k}^{\Omega_{\alpha}}}}\otimes{B}$. The desired equivalence then follows as $P^{\vee}_{k}\otimes T_{k}^{\Omega_{\alpha}}\simeq{0}$ by Proposition \ref{prop:projectorkirbyprop}. Next, suppose that $B$ is an arbitrary chain complex in $\Kom(\TLcat_{n}^{k})$, then $P^{\vee}_{k}\otimes{B}\otimes{T_{n}^{\Omega_{\alpha}}}$ decomposes as a multicone where each chain complex is of the form $P^{\vee}_{k}\otimes{\tau}\otimes{T_{n}^{\Omega_{\alpha}}}$ where $\tau$ is a $\TLcat_{n}^{k}$-diagram. By above, each term of $P^{\vee}_{k}\otimes{B}\otimes{T_{n}^{\Omega_{\alpha}}}$ is chain homotopy equivalent to $0$ and therefore $P^{\vee}_{k}\otimes{B}\otimes{T_{n}^{\Omega_{\alpha}}}\simeq{P^{\vee}_{k}\otimes{T_{k}^{\Omega_{\alpha}}\otimes{B}}\simeq{0}}$ as desired.
\end{proof}

\end{corollary}

Recalling that higher order projectors factor as $P^{\vee}_{n,k}=A\otimes{P^{\vee}_{k}}\otimes{B}$, Corollary \ref{cor:KirbyBSwap} and Proposition \ref{prop:projectorkirbyprop} allows us to conclude the following.

\begin{corollary}\label{cor:PnkkillsTnOmega}
    For an integer $0<k\leq{n}$, the complex $P^{\vee}_{n,k}\otimes{T_{n}^{\Omega_{\alpha}}}$ is contractible.
\end{corollary}

We are now ready to prove the main result of this section.

\begin{proposition}\label{prop:kirbybeltcontractible} If $n$ is an odd positive integer, then $T_{n}^{\Omega_{\alpha}}\simeq{0}$.

\begin{proof}
    For any $n$, by \eqref{eq:resofidentity}, we can express $T_{n}^{\Omega_{\alpha}}=\one_{n}\otimes{T_{n}^{\Omega_{\alpha}}}$ as 

    \[
    T_{n}^{\Omega_{\alpha}}\simeq
    P^{\vee}_{n,n\text{(mod 2)}}\otimes{T_{n}^{\Omega_{\alpha}}}\rightarrow{...}\rightarrow{P^{\vee}_{n,n-2}\otimes{T_{n}^{\Omega_{\alpha}}}\rightarrow{P_n\dual\otimes{T_{n}^{\Omega_{\alpha}}}}}.
    \]

    If $n$ is odd, then Corollary \ref{cor:PnkkillsTnOmega} implies $T_{n}^{\Omega_{\alpha}}\simeq{0}$. 
\end{proof}
\end{proposition}

Since the homology of the trace of $T_{n}^{\Omega_{\alpha}}$ is isomorphic to the skein lasagna module $\skein_{0}^{2}(S^{2}\times{B^{2}};\widetilde{\one}_{n},\alpha)$, Proposition \ref{prop:kirbybeltcontractible} produces the following immediate corollary.

\begin{corollary}\label{cor:S2xB2is0} Let $n$ be odd and let $\alpha\in{H_{2}^{\widetilde{\one}_{n}}(S^{2}\times{B^{2}})}\cong{H_{2}(S^{2}\times{B^{2}})}\cong \Z$. We have that $\skein^{2}_{0}(S^{2}\times{B^{2}};\widetilde{\one}_{n},\alpha)\cong{0}$. 

\begin{proof}
    If $n$ is odd, then $T_{n}^{\Omega_{\alpha}}\simeq{0}$, implying that $H^{*}(\close(T_{n}^{\Omega_{\alpha}}))\cong\skein_{0}^{2}(S^{2}\times{B^{2}};\widetilde{\one}_{n},\alpha)\cong{0}$ by Proposition \ref{prop:everythingthesame}.
\end{proof}
\end{corollary}

We apply this Corollary \ref{cor:S2xB2is0} to compute $\skein_{0}^{2}(S^{2}\times{S^{2}};\emptyset,\underline{\alpha})$ for specific homological levels.

\begin{theorem}\label{thrm:S2xS2lasagna}
Let $\underline{\alpha}=(\alpha_{1},\alpha_{2})\in{H_{2}(S^{2}\times{S^{2}};\mathbb{Z})}\cong{\mathbb{Z}^{2}}$ with at least one $\alpha_{1}$ or $\alpha_{2}$ odd, we have that $\skein_{0}^{2}(S^{2}\times{S^{2}};\emptyset,\underline{\alpha})\cong{0}$.

\begin{proof} Recall that a Kirby diagram of $S^{2}\times{S^{2}}$ is the Hopf link $L=L_{1}\cup{L_{2}}$ with $0$-framing on both components. Let $I=\mathbb{Z}_{\geq{0}}$ and $J=\mathbb{Z}_{\geq{0}}$, both equipped with the usual poset relation. The cabling directed system of $(S^{2}\times{S^{2}},\emptyset)$ at homological level $\underline{\alpha}$, denoted $\mathcal{B}^{\underline{\alpha}}$, lies over the indexing set $I\times{J}$. Let $D^{\underline{\alpha}}(0,0)$ denote the cabling of the Hopf link corresponding to $\underline{\alpha}$ 
and associated to the index $(0,0)$ (that is, the cable of the Hopf link with $|\alpha_{i}|$ parallel strands for the $i$th component $L_{i}$, oriented according to the sign of $\alpha_{i}$. See the left-most diagram in Figure \ref{fig:alpha1alpha2cabledhopf}).

\begin{figure}[!h]
    \begin{center}
        \includegraphics[width=\linewidth]{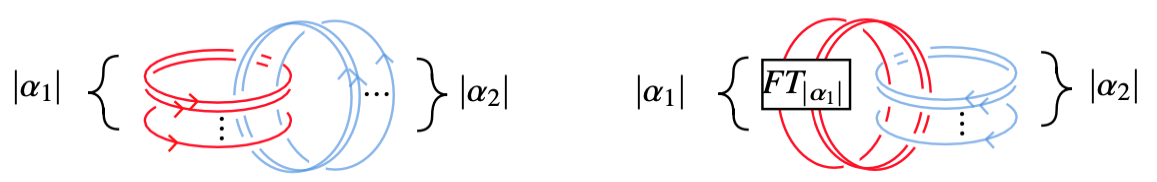}
    \end{center}
    \caption{\textbf{Left:} The diagram $D^{\underline{\alpha}}(0,0)$ for $\underline{\alpha}=(\alpha_{1},\alpha_{2})$, representing the $(0,0)$ object in the cabling directed system of $S^{2}\times{S^{2}}$. \textbf{Right:} The diagram $D^{(\alpha_{1},\alpha_{2})}(0,0)$ representing the $(0,0)$ object in the cabling directed system of $\mathbb{C}P^{2}\# \overline{\mathbb{C}P}^{2}$.}
    \label{fig:alpha1alpha2cabledhopf}
\end{figure}
 
Let $D^{\underline{\alpha}}(i,j)$ be the link diagram obtained from $D^{\underline{\alpha}}(0,0)$ by adding $2i$ parallel strands to the cable of $L_{1}$, with $i$ positively oriented and $i$ negatively oriented, and adding $2j$ parallel strands to the cable of $L_{2}$, with $j$ positively oriented and $j$ negatively oriented.
By Definition (\ref{def:CDS}), the $(i,j)$th object of $\mathcal{B}^{\underline{\alpha}}$ is $\KhR_{2}(\symmetrized{D^{\underline{\alpha}}(i,j)})$ where $\symmetrized{D^{\underline{\alpha}}(i,j)}$ is the complex symmetrized under the $\mathfrak{S}_{|\alpha_{1}|+2i}\times{\mathfrak{S}_{|\alpha_{2}|+2j}}$ action on parallel strands. The morphisms $(i,j)\rightarrow{(i+2,j)}$ (respectively, $(i,j)\rightarrow{(i,j+2)}$) are the symmetrized dotted ribbon maps associated to cables of $L_{1}$ (respectively, $L_{2}$).

Since $\skein_{0}^{2}(S^{2}\times{S^{2}};\emptyset,\underline{\alpha})\cong{\colim_{I\times{J}}(\mathcal{B}^{\underline{\alpha}})}\cong{\colim_{I}\colim_{J}(\mathcal{B}^{\underline{\alpha}})}$, we can compute the skein lasagna module of $(S^{2}\times{S^{2}},\emptyset)$ by computing the colimits of the directed systems given by a fixed $i\in{I}$ (or fixed $j\in{J})$. Without loss of generality, suppose that $\alpha_{1}$ is an odd integer, and fix an $i\in{I}$. The corresponding cabling directed system is of the form
\[
\cdots{\rightarrow{\KhR_{2}(\symmetrized{D^{\underline{\alpha}}(|\alpha_{1}|+2i,|\alpha_{2|}+2j)}}})\rightarrow{\KhR_{2}(\symmetrized{D^{\underline{\alpha}}(|\alpha_{1}|+2i,|\alpha_{2}|+2(j+1)}))}\rightarrow{\cdots}.
\]

Observe that the colimits of these directed systems are precisely $\skein_{0}^{2}(S^{2}\times{B^{2}};\widetilde{\one}_{|\alpha_{1}|+2i},\alpha_{2})$ with the strands of $\widetilde{\one}_{|\alpha_{1}|+2i}$ oriented. In particular, we have that

\[
\colim_{I\times{J}}(\mathcal{B}^{\underline{\alpha}})=\colim_{I}\colim_{J}(\mathcal{B}^{\underline{\alpha}})\cong{\colim_{I}(\skein_{0}^{2}(S^{2}\times{B^{2}};\widetilde{\one}_{|\alpha_{1}|+2i}},\alpha_{2})).
\]

As $\alpha_{1}$ is odd, $|\alpha_{1}|+2i$ is odd for all $i$, so $\skein_{0}^{2}(S^{2}\times{B^{2}};\widetilde{\one}_{|\alpha_{1}|+2i},\alpha_{2})\cong{0}$ for all $i$ by Corollary \ref{cor:S2xB2is0}. Therefore, $\colim_{I\times{J}}(\mathcal{B}^{\underline{\alpha}})\cong{0}$, as desired.
\end{proof}
\end{theorem}

There is a corresponding result for $S^{2}\tilde{\times}S^{2}\cong{\mathbb{C}P^{2}\#{\overline{\mathbb{C}P}^{2}}}$. Recall that a Kirby diagram representing $\mathbb{C}P^{2}\# \overline{\mathbb{C}P}^{2}$ is a Hopf link $L=L_{1}\cup{L_{2}}$ where $L_{1}$ has $+1$ framing and $L_{2}$ has $0$-framing. Although the cabling directed system corresponding to $(\mathbb{C}P^{2}\# \overline{\mathbb{C}P}^{2},\emptyset)$ does not admit the same symmetry of indexing sets, we have the following.

\begin{corollary}\label{cor:twistedS2xS2lasagna}
    Let $L=L_{1}\cup{L_{2}}$ be the framed oriented Hopf link in the Kirby diagram of $\mathbb{C}P^{2}\#{\overline{\mathbb{C}P}^{2}}$, and let $\alpha_{1}$  (respectively $\alpha_{2}$) represent the generator of $H_{2}(\mathbb{C}P^{2}\# \overline{\mathbb{C}P}^{2};\mathbb{Z})$ corresponding to the $(+1)$-framed component $L_{1}$ (respectively, the $0$-framed component $L_{2}$). Then, if $\alpha_{1}$ is odd, we have 
    \[
        \skein_{0}^{2}(\mathbb{C}P^{2}\# \overline{\mathbb{C}P}^{2};\emptyset,(\alpha_{1},\alpha_{2}))\cong{0}.
    \]

    \begin{proof}
        Let $\FT_{n}$ denote the full-twist tangle on $n$ strands. The skein lasagna module of $\mathbb{C}P^{2}\#{\overline{\mathbb{C}P}^{2}}$ at level $(\alpha_{1},\alpha_{2})\in{H_{2}(\mathbb{C}P^{2}\# \overline{\mathbb{C}P}^{2};\mathbb{Z})}$ is isomorphic to the colimit of the cabling directed system of $L$. Denote this cabling directed system by $\widetilde{\mathcal{B}}^{(\alpha_{1},\alpha_{2})}$. Define indexing sets $I$ and $J$ as in the proof of Theorem \ref{thrm:S2xS2lasagna} and observe that, unlike the case for $S^{2}\times{S^{2}}$, cables of the component $L_{1}$ are $T(n,n)$ torus links. Therefore, by fixing $i\in{I}$, the colimit of the corresponding directed system is instead isomorphic to $H^{*}(\close(\FT_{|\alpha_{1}|+2i}\otimes{T_{|\alpha_{1}|+2i}^{\Omega_{\alpha_{2}}})})\cong{\skein_{0}^{2}(S^{2}\times{B^{2}};\widetilde{\FT}_{|\alpha_{1}|+2i},\alpha_{2})}$ by Corollary \ref{cor:everythingthesame(w/braid)}. However, since $\alpha_{1}$ is odd and therefore $|\alpha_{1}|+2i$ is odd for all $i$, we have that $T^{\Omega_{\alpha_{2}}}_{|\alpha_{1}|+2i}\simeq{0}$, and therefore $H^{*}(\close(\FT_{|\alpha_{1}|+2i}\otimes{T_{|\alpha_{1}|+2i}^{\Omega_{\alpha_{2}}})})\cong{0}$ for all $i$. It follows that

        \begin{align*}
        \skein_{0}^{2}(\mathbb{C}P^{2}\# \overline{\mathbb{C}P}^{2};\emptyset,(\alpha_{1},\alpha_{2}))&\cong{\colim_{I\times{J}}(\widetilde{\mathcal{B}}^{(\alpha_{1},\alpha_{2})}})\\
        &\cong{\colim_{I}(\skein_{0}^{2}(S^{2}\times{B^{2}};\widetilde{\FT}_{|\alpha_{1}|+2i},\alpha_{2})})\\
        &\cong{0}.
        \end{align*}

    \end{proof}
\end{corollary}

Theorem \ref{thrm:S2xS2lasagna} and Corollary \ref{cor:twistedS2xS2lasagna} provide a partial picture of the skein lasagna modules of $S^{2}\times{S^{2}}$ and $S^{2}\tilde{\times}S^{2}$. To complete this picture, we now comment on the case where the homological levels have only even values.

\subsection{Even homological levels}

If the number of strands $n$ is even, by the resolution of the $n$ strand identity braid, by the argument used in the proof of Proposition \ref{prop:kirbybeltcontractible}, we have instead that $T_{n}^{\Omega_{\alpha}}\simeq{P^{\vee}_{n,0}\otimes{T_{n}^{\Omega_{\alpha}}}}$. The higher order projector $P^{\vee}_{n,0}$ has through-degree $0$. Cobordisms between tangles with through-degree $0$ have a certain splitting property.

\begin{definition}
    Let $T=T_{0}\cup{T_{1}}$ be a split tangle (so the connected components $T_{i}$ may each be placed in a $3$-ball $B^{3}_{i}$ such that $B_{0}^{3}\cap{B_{1}^{3}}=\emptyset)$. A cobordism between split tangles $C:T\rightarrow{T^{\prime}}$ is a \emph{split cobordism} if it can be written as $C=C_{0}\cup{C_{1}}$, where each $C_{i}:T_{i}\rightarrow{T^{\prime}_{i}}$ is a tangle cobordism entirely contained in $B_{i}\times{[0,1]}$.
\end{definition}

By neck-cutting, cobordism maps in a Bar-Natan cobordism category between through-degree $0$ tangles can be reinterpreted as a sum of split cobordism maps.
Hence, the differentials of a chain complex of through-degree $0$ tangles can be realized as linear combinations of split cobordism maps. With this in mind, we prove the following sliding-off property for the Kirby-colored belt on through-degree $0$ chain complexes.

\begin{theorem}\label{thrm:kirbyslideoff}
        Let $A$ be a chain complex in $\Kom(\TLcat_{n})$ of through-degree $0$, let $U^{k}$ denote the $k$ component unlink, and let $U^{\Omega_{\alpha}}$ denote the Kirby colored $0$-framed unknot. Then there is a chain homotopy equivalence $A\otimes{T_{n}^{\Omega_{\alpha}}}\simeq{A\sqcup{U^{\Omega_{\alpha}}}}$ (using the notation from Remark \ref{rmk:sqcupforcomplexes}); see Figure \ref{fig:kirbyslideoff}.
\end{theorem}

\begin{proof}
    We begin by showing that the chain complex $A\otimes{T_{n}^{\otimes{k}}}$ is chain homotopy equivalent to the complex $A\sqcup{U^{k}}$ for each $k$. By assumption, each chain group of $A$ is given by a direct sums of shifted flat tangles of through-degree $0$, denoted $A_{i}$. Furthermore, the differentials of $A$ are matrices of linear combinations of chain maps induced by split cobordisms. Let $A_{i}=\bigoplus_{j}q^{k_{j}}\tau_{j}^{i}$, where each $\tau_{j}^{i}$ is a through-degree $0$ tangle as above. Observe that there exist natural cobordism maps $\Sigma_{j}^{i}:\tau_{j}^{i}\otimes{T_{n}^{\otimes{k}}}\rightarrow{\tau_{j}^{i}\sqcup{U^{k}}}$ given simply by a composition of isotopies that slide the belts off of $\tau_{j}^{i}$ (see Figure \ref{fig:greenphone}). These $\Sigma_{j}^{i}$ maps are given by compositions of Reidemeister II moves and are therefore homotopy equivalence maps $\tau_{j}^{i}\otimes{T_{n}^{\otimes{k}}}\simeq{\tau_{j}^{i}\sqcup{U^{k}}}$. Then, letting $\Sigma^{i}:=\bigoplus_{j}\Sigma_{j}^{i}$, these maps are also chain homotopy equivalence maps $A_{i}\otimes{T_{n}^{\otimes{k}}}\simeq{A_{i}\sqcup{U^{k}}}$.

    By Proposition \ref{prop:tot=twistedcomplex}, we have
    \[
        A\otimes{T_{n}^{\otimes{k}}}
        \simeq
        \Tot(
        \{A_{i}\sqcup{U^{k}},g_{i,j}\},
        \qquad
        g_{i,i+1} = \Sigma^{i+1}\circ{(\partial_{i}\otimes{\idmap})}\circ{(\Sigma^{i})^{-1}}
    \]
    where 
     $\partial_{i}: A_i \to A_{i+1}$ is a differential of $A$
    and  $g_{i,j}$ are morphisms of homological degree $j-i-1$
    satisfying Equation \eqref{eq:twisted-condition} of Definition \ref{def:twistedcomplex} (see Figure \ref{fig:AsqcupUk}). Note that if $j-i>1$ for $g_{i,j}:A_{i}\sqcup{U^{k}}\rightarrow{A_{j}\sqcup{U^{k}}}$, then $g_{i,j}$ is the zero map, as $A_{j}\sqcup{U^{k}}$ is itself a chain complex supported only in homological degree $0$ (it is a direct sum of flat tangles). Thus, the twisted complex $\{A_{i}\sqcup{U^{k}},g_{i,j}\}$ is an actual chain complex. 
    {
    Furthermore, by functoriality, $g_{i,i+1}$ and $\partial_{i}\sqcup{\idmap}$ are homotopic maps, since the cobordisms they represent are isotopic.
    }
    {
    In other words, the following diagram homotopy commutes:
    }
    \begin{center}
    \begin{tikzcd}
        A_i \otimes T^{\otimes k}_n \arrow{r}{\partial_i \otimes \idmap} \arrow[swap]{d}{\Sigma^i}
            & A_{i+1} \otimes T^{\otimes k}_n  \arrow{d}{\Sigma^{i+1}} \\
        A_i \sqcup U^k \arrow{r}{\partial_i \sqcup \idmap} 
            & A_{i+1} \sqcup U^{k}\\
    \end{tikzcd}
    \end{center}
    {
    Hence $A\sqcup U^k = (\bigoplus_i A_i \sqcup U_k, \bigoplus_i \partial_i \sqcup \idmap)$ is chain homotopy equivalent to the complex $A \otimes T_n^{\otimes k}$.
    }

    Now let $\Sigma^{(k)}:A\otimes{T_{n}^{\otimes{k}}}\rightarrow{A\sqcup{U^{k}}}$ denote the chain homotopy equivalence map provided by Proposition \ref{prop:tot=twistedcomplex}. 
    Then, the cobordism maps $\Sigma^{(k)}$ commute with the symmetrizing cobordisms and dotted cup cobordisms. We may then define the following directed systems:

    \[
    A\otimes \mathcal{A}^{\alpha}_{n}:=A\otimes{\symmetrized{T_{n}^{\otimes{|\alpha|}}}}\xrightarrow{\idmap\otimes{\symmetrized{\acupdot}}}{A\otimes{\symmetrized{T_{n}^{\otimes{(|\alpha|+2)}}}}}\xrightarrow{\idmap\otimes{\symmetrized{\acupdot}}}{A\otimes{\symmetrized{T_{n}^{\otimes{|\alpha|+4}}}}}\rightarrow{\cdots}
    \]
    \vspace{0.1mm}
    \[
    A\sqcup{\mathcal{A}^{\alpha}_{n}}:=A\sqcup{\symmetrized{T_{n}^{\otimes{|\alpha|}}}}\xrightarrow{\idmap\sqcup{\symmetrized{\acupdot}}}{A\sqcup{\symmetrized{T_{n}^{\otimes{(|\alpha|+2)}}}}}\xrightarrow{\idmap\sqcup{\symmetrized{\acupdot}}}{A\sqcup{\symmetrized{T_{n}^{\otimes{|\alpha|+4}}}}}\rightarrow{\cdots}
    \]
    \vspace{0.3mm}

    Observe also that $A\otimes{T_{n}^{\Omega_{\alpha}}}=\hocolim{(A\otimes{\mathcal{A}^{\alpha}_{n}})}$ and $A\sqcup{U^{\Omega_{\alpha}}}=\hocolim(A\sqcup{\mathcal{A}_{n}^{\alpha}})$. Since we have homotopy equivalence maps $\Sigma^{(m)}$ between each object in our directed systems, by Lemma \ref{lem:forSwapsandSlides}(a), there is a well-defined chain map $\Sigma:A\otimes{T_{n}^{\Omega_{\alpha}}}\rightarrow{A\sqcup{U^{\Omega_{\alpha}}}}$ on homotopy colimits. Furthermore, by Lemma \ref{lem:forSwapsandSlides}(c), we immediately have that $A\otimes{T_{n}^{\Omega_{\alpha}}}\simeq{A\sqcup{U^{\Omega_{\alpha}}}}$, as desired.
\end{proof}

\begin{figure}[!h]
    \centering
        \includegraphics[width=.5\linewidth]{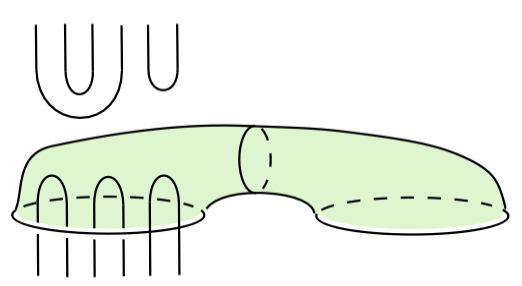}
    \caption{Belt slide-off cobordism from Reidemeister II moves. The split cobordisms $A_{i}\rightarrow{A_{j}}$ between split tangles do not intersect the shaded surface.}
    \label{fig:greenphone}
\end{figure}

A Kirby diagram of $S^{2}\times{B^{2}}$ is the $0$-framed unknot. In \cite{MN22}, the $N=2$ skein lasagna module of $(S^{2}\times{B^{2}},\emptyset)$, equivalent to the cabled Khovanov homology of the $0$-framed unknot, was shown to be isomorphic to $\mathbb{F}[A_{0},A_{0}^{-1},A_{1}]$ for formal variables $A_{0}$ and $A_{1}$ in $q$-degrees $0$ and $-2$ respectively.
At homological level $\alpha$, the skein lasagna module $\skein_{0}^{2}(S^{2}\times{B^{2};\emptyset,\alpha)}$ is isomorphic to the subgroup of $\mathbb{F}[A_{0},A_{0}^{-1},A_{1}]$ generated by homogeneous polynomials of degree $\alpha$. 
Denote this subgroup by $\mathbb{F}_{|\alpha|}[A_{0},A_{0}^{-1},A_{1}]$.

    \begin{figure}[!h]
        \begin{center}
            \includegraphics[width=\linewidth]{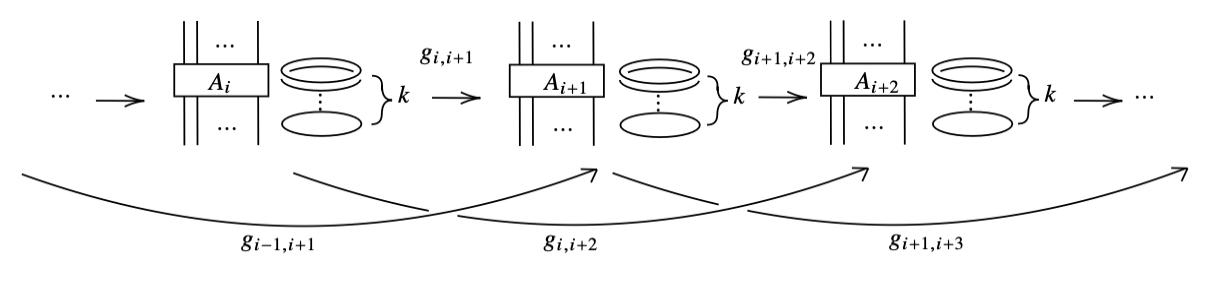}
        \end{center}
        \caption{The twisted complex $\{A_{i}\sqcup{T_{0}^{\otimes{k}}},g_{i,j}$\}}
        \label{fig:AsqcupUk}
    \end{figure}

Theorem \ref{thrm:kirbyslideoff} then has the immediate corollary for pairs $(S^{2}\times{B^{2}},\widetilde{\one}_{n})$ for an even integer $n$.

\begin{corollary}\label{cor:even-evensplit}
    Let $\alpha\in{H_{2}(S^{2}\times{B^{2}};\Z)\cong{\Z}}$ and let $k\in{\mathbb{N}}$. Then there is an isomorphism $\skein_{0}^{2}(S^{2}\times{B^{2}};\widetilde{\one}_{2k},\alpha)\cong{H^{*}(\close(P^{\vee}_{2k,0}))}\otimes{\mathbb{F}_{|\alpha|}[A_{0},A_{0}^{-1},A_{1}]}$.
\end{corollary}

\begin{proof}
    The skein lasagna module of the pair $(S^{2}\times{B^{2}},\widetilde{\one}_{2k})$ is isomorphic to $H^{*}(\close(T_{2k}^{\Omega_{\alpha}}))$ by Corollary \ref{cor:everythingthesame(w/braid)}. After tensoring $T_{2k}^{\Omega_{\alpha}}$ with the resolution of the identity $\one_{2k}$, we obtain a chain homotopy equivalence
    \[    T_{2k}^{\Omega_{\alpha}}\simeq{P^{\vee}_{2k,0}\otimes{T_{2k}^{\Omega_{\alpha}}}}.
    \]
    However, since $P^{\vee}_{2k,0}$ is a through-degree $0$ complex in $\Kom(\TLcat_{n})$, by Theorem \ref{thrm:kirbyslideoff}, we have that $T_{2k}^{\Omega_{\alpha}}\simeq{P^{\vee}_{2k,0}\sqcup{U^{\Omega_{\alpha}}}}$. Note that $\close(P^{\vee}_{2k,0}\sqcup{U^{\Omega_{\alpha}}})=\close(P^{\vee}_{2k,0})\sqcup{U^{\Omega_{\alpha}}}$, implying
    
    \begin{align*}
    \skein_{0}^{2}(S^{2}\times{B^{2}};\widetilde{\one}_{2k},\alpha)&\cong{H^{*}(\close(P^{\vee}_{2k,0}\otimes{T_{2k}^{\Omega_{\alpha}}}))}\\
    &\cong{H^{*}(\close(P^{\vee}_{2k,0})\sqcup{U^{\Omega_{\alpha}}})}\\
    &\cong{H^{*}(\close(P^{\vee}_{2k,0})})\otimes{\mathbb{F}_{|\alpha|}[A_{0},A_{0}^{-1},A_{1}]}.
    \end{align*}
    proving the claim.
\end{proof}

We can similarly extend the result of Corollary \ref{cor:even-evensplit} to the pair $(S^{2}\times{S^{2}},\emptyset)$, and may now complete the proof of Corollary \ref{cor:main}. 

\begin{proof}[Proof of Corollary \ref{cor:main}]
By Theorem \ref{thrm:S2xS2lasagna}, it remains to show that the skein lasagna module of $S^{2}\times{S^{2}}$ vanishes for $(\alpha_{1},\alpha_{2})\in{H_{2}(S^{2}\times{S^{2})}}$ where both entries are even. Let $\acupdot^{*}$ denote the morphisms on colimits induced by $\acupdot$ and let $\alpha_{1}$ and $\alpha_{2}$ be even integers. The skein lasagna module $\skein^{2}_{0}(S^{2}\times{S^{2}};\emptyset,(\alpha_{1},\alpha_{2}))$ is isomorphic to $\colim(\mathcal{V})\otimes{\mathbb{F}_{|\alpha_{2}|}[A_{0},A_{0}^{-1},A_{1}]}$, where $\mathcal{V}$ is the directed system

\[
\mathcal{V}:=H^{*}(\close(P^{\vee}_{|\alpha_{1}|,0}))\xrightarrow{{\acupdot}^{*}}H^{*}(\close(P^{\vee}_{|\alpha_{1}|+2,0}))\xrightarrow{{\acupdot}^{*}}{H^{*}(\close(P^{\vee}_{|\alpha_{1}|+4,0}))\xrightarrow{{\acupdot}^{*}}\cdots}
\]
However, if $\alpha_{2}$ is taken to be odd instead, we have that $\colim(\mathcal{V})\otimes{\mathbb{F}_{|\alpha_{2}|}}[A_{0},A_{0}^{-1},A_{1}]\cong{0}$, implying that $\colim(\mathcal{V})=0$. Therefore, $\skein_{0}^{2}(S^{2}\times{S^{2}};\emptyset,(\alpha_{1},\alpha_{2}))\cong{0}$ for all $(\alpha_{1},\alpha_{2})\in{H_{2}(S^{2}\times{S^{2}})}$.

\end{proof}


\begin{ack}
We wholeheartedly thank Eugene Gorsky for his continued guidance and encouragement. We are heavily indebted to Matt Hogancamp, Paul Wedrich, Ciprian Manolescu, and Dave Rose for their enlightening ideas and conversations. We thank Qiuyu Ren for their careful reading of a first draft of this article and for pointing out to us the argument at the end of Section \ref{sec:mainsec} which shows that the skein lasagna module of $S^2 \times S^2$ at (even, even) homological levels is also trivial. 
We are also thankful to Shanon Rubin and Trevor Oliveira-Smith for some very helpful conversations. This work was partially supported by the NSF grant  DMS-2302305.
{Finally, we thank the referee for their constructive comments which helped greatly improve the readability of this article.}
\end{ack}

\begin{funding}
This work was partially supported by the U.S. National Science Foundation grant DMS-2302305.
\end{funding}


\bibliographystyle{emss}
\bibliography{main}









\end{document}